\documentclass[reqno,11pt]{amsart}
\usepackage{amsthm,amsfonts,amssymb,euscript,
mathrsfs,graphics,color,amsmath,latexsym,marginnote}
\usepackage{dsfont}   
\usepackage{mathtools}

\theoremstyle{plain}

\usepackage{amsmath}
\usepackage{hyperref}
\usepackage{fourier}

\numberwithin{equation}{section}

\newcommand{\NLS}{\scriptscriptstyle\mathrm{NLS}}
\newcommand{\KG}{\scriptscriptstyle\mathrm{KG}}

\newtheorem{Teo}{Theorem}
\newtheorem{theorem}{Theorem}[section]
\newtheorem{proposition}[theorem]{Proposition}
\newtheorem{lemma}[theorem]{Lemma}
\newtheorem{corollary}[theorem]{Corollary}

\newtheorem{remark}[theorem]{Remark}
\newtheorem{remarks}[theorem]{Remark}
\newtheorem{hypo}[theorem]{Hypothesis}
\newtheorem{definition}[theorem]{Definition}
\newcommand{\be}{\begin{equation}}
\newcommand{\ee}{\end{equation}}

\newcommand{\uno}{\mathds{1}}
\newcommand{\e}{\varepsilon}

\newcommand{\ov}{\overline}

\newcommand{\R}{\mathbb R}
\newcommand{\C}{\mathbb C}

\newcommand{\Z}{\mathbb Z}

\newcommand{\N}{\mathbb N}
\newcommand{\T}{\mathbb T}

\newcommand{\s }{\sigma }
\newcommand{\ii }{{\rm i} }

\newcommand{\g }{\gamma}

\newcommand{\z }{\zeta }

\newcommand{\x }{\xi }

\newcommand{\Nc}{\mathcal{N}}

\newcommand{\pa}{\partial}

\newcommand{\opbw}{{Op^{{\scriptscriptstyle{\mathrm BW}}}}}

\newcommand{\diag}{{\rm diag}}

\def\hat{\widehat}

\def\bar{\overline}

\renewcommand{\Re}{\mathrm{Re}\,}
\renewcommand{\Im}{\mathrm{Im}\,}

\renewcommand{\S}{\mathbb{S}}
\newcommand{\Cc}{\mathcal C}
\newcommand{\Dc}{\mathcal D}
\renewcommand{\Mc}{\mathcal M}

\def\ba{\begin{aligned}}
\def\ea{\end{aligned}}
\def\beginm{\begin{multline}}
\def\endm{\end{multline}}

\providecommand{\vect}[2]{{\bigl[\begin{smallmatrix}#1\\#2\end{smallmatrix}\bigr]}}   
\providecommand{\sm}[4]{{\bigl[\begin{smallmatrix}#1&#2\\#3&#4\end{smallmatrix}\bigr]}}

\makeatletter

\def\l@subsection{\@tocline{2}{0pt}{2.5pc}{5pc}{}}
\def\l@subsubsection{\@tocline{3}{0pt}{4.5pc}{5pc}{}}
\renewcommand\tocchapter[3]{%
  \indentlabel{\@ifnotempty{#2}{\ignorespaces#2.\quad}}#3%
}
\newcommand\@dotsep{4.5}
\def\@tocline#1#2#3#4#5#6#7{\relax
  \ifnum #1>\c@tocdepth 
  \else
    \par \addpenalty\@secpenalty\addvspace{#2}%
    \begingroup \hyphenpenalty\@M
    \@ifempty{#4}{%
      \@tempdima\csname r@tocindent\number#1\endcsname\relax
    }{%
      \@tempdima#4\relax
    }%
    \parindent\z@ \leftskip#3\relax \advance\leftskip\@tempdima\relax
    \rightskip\@pnumwidth plus1em \parfillskip-\@pnumwidth
    #5\leavevmode\hskip-\@tempdima{#6}\nobreak
    \leaders\hbox{$\m@th\mkern \@dotsep mu\hbox{.}\mkern \@dotsep mu$}\hfill
    \nobreak
    \hbox to\@pnumwidth{\@tocpagenum{#7}}\par
    \nobreak
    \endgroup
  \fi}
\def\l@subsection{\@tocline{2}{0pt}{2.5pc}{5pc}{}}

\begin{document}
\bibliographystyle{plain}

\title[Long time solutions for quasi-linear NLS and KG on tori]{Long time solutions for 
quasi-linear Hamiltonian  perturbations of Schr\"odinger and Klein-Gordon equations on tori}

\date{}

\author{Roberto Feola}
\address{\scriptsize{Dipartimento di matematica e fisica, Universit\'a degli studi Roma Tre\\
1, Largo San Leonardo Murialdo \\
00146 Rome, Italy}}
\email{roberto.feola@uniroma3.it}

\author{Beno\^it Gr\'ebert}
\address{\scriptsize{Laboratoire de Math\'ematiques Jean Leray, Universit\'e de Nantes, UMR CNRS 6629\\
2, rue de la Houssini\`ere \\
44322 Nantes Cedex 03, France}}
\email{benoit.grebert@univ-nantes.fr}

\author{Felice Iandoli}
\address{\scriptsize{Laboratoire Jacques-Louis Lions, Sorbonne Universit\'e, UMR CNRS 7598\\
4, Place Jussieu\\
 75005 Paris Cedex 05, France}}
\email{felice.iandoli@sorbonne-universite.fr}



\thanks{Felice Iandoli has been supported  by ERC grant ANADEL 757996.
Roberto Feola and Beno\^it Gr\'ebert have been supported by 
 the Centre Henri Lebesgue ANR-11-LABX- 0020-01 
 and by ANR-15-CE40-0001-02 ``BEKAM'' of the ANR}  
 

\begin{abstract}  
We consider quasi-linear, Hamiltonian perturbations of the cubic Schr\"odinger  and of the cubic (derivative) Klein-Gordon equations  on the $d$ dimensional torus. 
 If
$\varepsilon\ll1$ is the size of the initial datum, we prove that the lifespan of solutions is \emph{strictly} larger than the  local existence time  $\varepsilon^{-2}$.
More precisely, concerning the Schr\"odinger equation we  
show that the lifespan is at least of order $O(\varepsilon^{-4})$,  in the Klein-Gordon case we prove that the  solutions exist at least for a time  of order 
$O(\varepsilon^{-{8/3}^{-}})$ as soon as $d\geq3$.  
Regarding the Klein-Gordon equation, our result presents novelties  also in the case of semi-linear perturbations:  we show that the lifespan is at least of order 
$O(\varepsilon^{-{10/3}^{-}})$, 
improving, for cubic non-linearities and $d\geq4$, 
the general results in \cite{Delort-Tori, fang}.

\end{abstract}   


\maketitle
\setcounter{tocdepth}{2}
\tableofcontents

\section{Introduction}
This paper is concerned with the study of the lifespan of
solutions of two classes of quasi-linear, Hamiltonian equations on the 
$d$-dimensional torus 
$\mathbb{T}^{d}:=(\mathbb{R}/2\pi\mathbb{Z})^{d}$, $d\geq 1$.
We  study quasi-linear perturbations of the Schr\"odinger
and  Klein-Gordon equations.

The Schr\"odinger equation we consider is the following
\begin{equation*}\label{NLS}\tag{NLS}
\left\{\begin{aligned}
&\ii\partial_t u+\Delta u-V*u
+\big[\Delta (h(|u|^2))\big]h'(|u|^2)u-|u|^2u=0,\\
&u(0,x)=u_0(x)\end{aligned}\right.
\end{equation*}
where $\C\ni u:=u(t,x)$, $x\in \T^d$, $d\geq 1$, $V(x)$ 
is a real valued  potential even with respect to $x$, $h(x)$ 
is a function in $C^{\infty}(\R;\mathbb{R})$ such that $h(x)=O( x^{2})$ as $x\to 0$.
The initial datum $u_0$ has  small size 
and belongs to the Sobolev space $H^s(\T^d)$ 
(see \eqref{Sobnorm}) with $s\gg 1$.

We examine also the Klein-Gordon equation
 \begin{equation*}\label{KG}\tag{KG}
\left\{ \begin{aligned}
& \pa_{tt}\psi-\Delta \psi+m\psi+f(\psi)+g(\psi)=0 \,,\\
&\psi(0,x)=\psi_0\,,\\
&\pa_{t}\psi(0,x)=\psi_1\,,
 \end{aligned}\right.
 \end{equation*}
 where $\R\ni\psi:=\psi(t,x)$,  
 $x\in \mathbb{T}^{d}$, $d\geq 1$ 
 and $m>0$. The initial data $(\psi_0, \psi_1)$ have small size and belong to the Sobolev space
 $H^{s}(\mathbb{T}^{d})\times H^{s-1}(\mathbb{T}^{d})$, for some $s\gg1$.
 The nonlinearity $f(\psi)$ has the form
  \begin{equation}\label{KGnon}
f(\psi):=-\sum_{j=1}^{d}\pa_{x_j}\big(\pa_{\psi_{x_j}}F(\psi,\nabla \psi)\big)
+(\pa_{\psi}F)(\psi,\nabla \psi)
 \end{equation}
 where $F(y_0,y_1,\ldots,y_{d})\in C^{\infty}(\mathbb{R}^{d+1},\mathbb{R})$, 
 has a zero of order at least $5$ at the origin. 
 The non linear term $g(\psi)$ has the form 
 \begin{equation}\label{KGnon2}
g(\psi)=(\pa_{y_0}G)(\psi,\Lambda_{\KG}^{\frac{1}{2}}\psi)
+\Lambda_{\KG}^{\frac{1}{2}}(\pa_{y_1}G)(\psi,\Lambda_{\KG}^{\frac{1}{2}}\psi)
\end{equation}
where $G(y_0,y_1)\in C^{\infty}(\mathbb{R}^{2};\mathbb{R})$ is a homogeneous polynomial 
of degree $4$ and $\Lambda_{\KG}$ is the operator 
\begin{equation}\label{def:Lambda}
\Lambda_{\KG}:=(-\Delta+m)^{\frac{1}{2}}\,,
\end{equation}
defined by linearity as
\begin{equation}\label{def:Lambda2}
\Lambda_{\KG} e^{\ii j\cdot x}=\Lambda_{\KG}(j) e^{\ii j\cdot x}\,,\qquad \Lambda_{\KG}(j)=\sqrt{|j|^{2}+m}\,,\;\;\;
\forall\, j\in\mathbb{Z}^{d}\,.
\end{equation}

%
%
%
%

{\bf{Historical introduction for \eqref{NLS}.} }
Quasi-linear Schr\"odinger equations of the specific form \eqref{NLS} 
appear in many domains 
of physics like plasma physics and fluid mechanics \cite{sergio,goldman}, 
quantum mechanics \cite{hasse}, condensed matter theory \cite{fedayn}. 
They are also important  in the study of Kelvin waves 
in the superfluid turbulence \cite{LLNR}.
Equations of the form \eqref{NLS} posed in the \emph{Euclidean space}  
have received the attention of many mathematicians.  
The first result, concerning the  local well-posedness, 
is due to Poppenberg \cite{Pop1} in the one dimensional case. 
This  has been generalized by Colin to any dimension \cite{colin}. 
A more general class of equations is considered in the pioneering work 
by Kenig-Ponce-Vega \cite{KPV}. 
These results of local well-posedness have been recently optimized, 
in terms of regularity of the initial condition, 
by Marzuola-Metcalfe-Tataru \cite{MMT3} (see also references therein). 
Existence of standing waves  has been studied 
by Colin \cite{colin2003} and Colin-Jeanjean \cite{colinjj}. 
The global well-posedness has been established 
by de Bouard-Hayashi-Saut \cite{saut-glob} 
in dimension two and three for small data. 
This proof is based on dispersive estimates and energy method. 
New ideas have been introduced in studying the global well-posedness 
for other quasi-linear equations on the \emph{Euclidean space}. 
Here the aforementioned tools are combined with \emph{normal form} reductions. 
We quote Ionescu-Pusateri \cite{pusa1,pusa2} 
for the water-waves equation in two dimensions.

Very little is known when the equation \eqref{NLS} is posed on a \emph{compact manifold}. 
The first local well-posedness results on the circle 
are given in the work by Baldi-Haus-Montalto \cite{BHM}  
and in the paper \cite{Feola-Iandoli-Loc}. 
Recently these results have been  generalized to the case of tori of 
any dimension in \cite{Feola-Iandoli-local-tori}. 
Except these local existence results, 
nothing is known concerning the long time behavior  of  the solutions. 
The problem of  \emph{global existence/blow-up}  is completely open. 
In the aforementioned paper \cite{saut-glob} it is exploited  
the dispersive character of the flow of the linear Schr\"odinger equation. 
This property is not present on compact manifolds: 
the solutions of the linear Schr\"odinger equation 
do not decay when the time goes to infinity. 
However in the one dimensional case 
 in \cite{Feola-Iandoli-Long, Feola-Iandoli-Totale} it is proven 
that \emph{small} solutions of quasi-linear Schr\"odinger equations 
exist for long, but finite, times. 
In these works two of us exploit the fact that quasi-linear Schr\"odinger 
equations may be   reduced to constant coefficients 
through a \emph{para-composition} generated by a diffeomorphism of the circle. 
This powerful tool has been used 
for the same purpose by other authors in the context 
of water-waves equations, firstly by Berti-Delort in \cite{BD} 
in a  \emph{non resonant} regime, 
secondly by Berti-Feola-Pusateri in \cite{BFP,BFP1} 
and Berti-Feola-Franzoi \cite{BFF1} in the \emph{resonant} case.
We also mention that this feature has been used  
in other contexts for the same equations, for instance 
Feola-Procesi \cite{FP} prove the existence of a 
large set of \emph{quasi-periodic} 
(and hence globally defined) solutions when the problem is posed on the circle.
This ``reduction to constant coefficients'' is a peculiarity  
of one dimensional problems, in higher dimensions 
new ideas have to be introduced. 
For quasi-linear equations on tori of dimension 
two we quote the paper about long-time solutions 
for water-waves problem  by Ionescu-Pusateri \cite{IPtori}, 
where a  different \emph{normal form} analysis  has been presented. 

{\bf Historical introduction for \eqref{KG}.} The local existence for \eqref{KG} is classical  and we refer to Kato \cite{kato}. Many analyses have been done for \emph{global/long time} existence.

 When the equation is posed on  the Euclidean space we have global existence for small and localized data  in the papers by Delort \cite{Delort-glob} and Stingo \cite{Sti}, here the authors use dispersive estimates on the linear flow combined with quasi-linear normal forms.
 
 For \eqref{KG} on compact manifolds we quote Delort \cite{Delort-circle, Delort-sphere} on $\mathbb{S}^d$ and Delort-Szeftel \cite{DelortSzeft1} on $\T^d$. The results obtained, in terms of length of the lifespan of solutions, are stronger in the case of the spheres. More precisely in the case of spheres the authors show the following. If $m$ in \eqref{KG} is chosen outside of a set of zero Lebesgue measure, then for any natural number $N$,  any initial condition of size $\varepsilon$ (small depending on $N$) produces a solution whose lifespan is at least of magnitude $\varepsilon^{-N}$. In the case of tori in \cite{DelortSzeft1} they consider a quasi-linear equation, vanishing quadratically at the origin and they prove that the lifespan of solutions is of order $\varepsilon^{-2}$ if the initial condition has size $\varepsilon$ small enough. 
The differences between the two results are due to the different behaviors of the eigenvalues of the square root of the Laplace-Beltrami operator on $\S^d$ and $\T^d$. The difficulty on the tori is a consequence of the fact that the set of differences of eigenvalues of $\sqrt{-\Delta_{\T^d}}$ is dense in $\R$ if $d\geq 2$, this does not happen in the case of spheres. A more general set of manifolds where this does not happen is the Zoll manifolds, in this case we quote the paper by Delort-Szeftel \cite{DelortSzeft2} and 
Bambusi-Delort-Gr\'ebert-Szeftel \cite{BDGS} for semi-linear Klein-Gordon equations. For semi-linear Klein-Gordon equations on tori we have the results by 
Delort \cite{Delort-Tori}  and Fang-Zhang \cite{fang}. 
In \cite{Delort-Tori} the author proves that if the non-linearity 
is vanishing at order $k+1$ at zero then any initial 
datum of small size $\varepsilon$ produces a solution 
whose lifespan is at least of magnitude 
$\varepsilon^{-k(1+\frac{2}{d})}$, up to a logarithmic loss. 
In \cite{fang} the authors obtain a time $O(\varepsilon^{-k\frac{3}{2}^-})$.
We improve these results, see 
Theorems \ref{main:KGsemi} and \ref{main:KG}, when $k=2$.

{\bf Statement of the main results.} The aim of this paper is to prove, 
in the spirit of \cite{IPtori}, that we may go 
beyond the trivial time of existence, 
given by the local well-posedness theorem which is $\varepsilon^{-2}$ since we are considering equations vanishing cubically at the origin and initial conditions of size $\varepsilon$.

In order to state our main theorem for \eqref{NLS} we need to 
make some hypotheses on the potential $V$.
We consider potentials having the following form
\begin{equation}\label{insPot}
V(x)=(2\pi)^{-d/2}\sum_{\x\in \mathbb{Z}^{d}} \hat{V}(\x)e^{\ii\x\cdot x}\,,\quad
\hat{V}(\x)=\frac{x_{\x}}{4 (1+|\x|^2)^{\frac{m}{2}}}\,, \quad x_{\x}\in 
\left[-\frac{1}{2},\frac{1}{2}\right]\subset\mathbb{R} \,,\,\,\, \N\ni m>d/2.
\end{equation}
We  endow the set $\mathcal{O}:=[-1/2,1/2]^{\mathbb{Z}^{d}}$ 
 with 
the standard probability measure on product spaces. This choice of the function defining the convolution potential is standard (\cite{FaouGreNek,BG}): essentially one needs that the Fourier coefficients decay at a certain rate and that the function $V(x)$ depends on some free parameters $x_\xi$.
Our main theorem is the following.
\begin{Teo}{(\bf Long time existence for NLS).}\label{main:NLS}
Consider the equation \eqref{NLS} with $d\geq 2$. There exists 
$\mathcal{N}\subset\mathcal{O}$ having zero  Lebesgue measure 
such that if $x_{\xi}$ in \eqref{insPot}
is in $\mathcal{O}\setminus\mathcal{N}$, we have the following. 
There exists $s_0=s_0(d,m)\gg 1$ such that for any $s\geq s_0$ there are 
constants  $c_0>0$ and $\e_0>0$ such that for any $0<\e\leq \e_0$
we have the following. 
If  $\|u_0\|_{H^s}<1/4\varepsilon$,  there exists a unique solution 
of the Cauchy problem \eqref{NLS} such that 
\begin{equation}\label{tesi}
u(t,x)\in C^0\big([0,T);H^s(\T^d)\big)\,, 
\quad \sup_{t\in[0,T)}\|u(t,\cdot)\|_{H^s}\leq \varepsilon\,, 
\quad T\geq c_0 \varepsilon^{-4}\,.
\end{equation}
\end{Teo}

In the one dimensional case we do not need any external parameter and we shall prove the following Theorem.

\begin{Teo}\label{main2}
Consider \eqref{NLS} with $V\equiv 0$ and $d=1$. 
There exists $s_0 \gg 1$ such that for any $s\geq s_0$ there are 
constants  $c_0>0$ and $\e_0>0$ such that for any $0<\e\leq \e_0$
we have the following. 
If  $\|u_0\|_{H^s}<1/4\varepsilon$,  there exists a unique solution 
of the Cauchy problem \eqref{NLS} such that 
\begin{equation}\label{tesi2}
u(t,x)\in C^0\big([0,T);H^s(\T^d)\big)\,, 
\quad \sup_{t\in[0,T)}\|u(t,\cdot)\|_{H^s}\leq \varepsilon\,, 
\quad T\geq c_0 \varepsilon^{-4}\,.
\end{equation}
\end{Teo}
These are, to the best of our knowledge,  
the firsts results of this kind for quasi-linear Schr\"odinger 
equations posed on compact manifolds of dimension greater than one. 

Our main theorem regarding the problem \eqref{KG} is the following.
\begin{Teo}{(\bf Long time existence for KG).}\label{main:KG}
Consider the equation \eqref{KG} with $d\geq 2$. There exists 
$\mathcal{N}\subset[1,2]$ having zero  Lebesgue measure 
such that if $m\in [1,2]\setminus\mathcal{N}$  we have the following. 
There exists $s_0=s_0(d)\gg 1$ such that for any $s\geq s_0$ 
  the following holds.
For any $\delta>0$ there exists 
$\e_0=\e_0(s,m,\delta)>0$ such that for any $0<\e\leq \e_0$
and 
any initial data 
$(\psi_0,\psi_1)\in H^{s+1/2}(\mathbb{T}^{d})\times H^{s-1/2}(\mathbb{T}^{d})$
such that
\[
\|\psi_0\|_{H^{s+1/2}}+\|\psi_1\|_{H^{s-1/2}}\leq \tfrac{1}{32}\e\,,
\]
there exists a unique solution 
of the Cauchy problem \eqref{KG} such that 
\begin{equation}\label{tesiKG}
\begin{aligned}
&\psi(t,x)\in C^0\big([0,T);H^{s+\frac{1}{2}}(\T^d)\big)\bigcap 
C^1\big([0,T);H^{s-\frac{1}{2}}(\T^d)\big)\,, \\
& \sup_{t\in[0,T)}\Big(\|\psi(t,\cdot)\|_{H^{s+\frac{1}{2}}} +
\|\pa_t \psi(t,\cdot)\|_{H^{s-\frac{1}{2}}}    \Big)\leq \varepsilon\,, 
\qquad T\geq \varepsilon^{-\mathtt{a}+\delta}\,.
\end{aligned}
\end{equation}
where $\mathtt{a}=3$ if $d=2$ and $\mathtt{a}=8/3$ if $d\geq3$.
\end{Teo}


The time of existence in \eqref{tesiKG} is intimately connected with the lower bounds on the four waves interactions given in Section \ref{sec:NonresCondKG}. More precisely the time of existence is larger then $\e^{-2-{2}/{\beta}} $ with $\beta$ given in Proposition \ref{NonresCondKG}. This is the reason for the difference between the result in $d=2$ (where $\beta=2^+$) and $d\geq3$ (where $\beta=3^+$). We do not know if this result is sharp, this is an open problem.
Despite this fact, Theorem \ref{main2} improves the general result in \cite{Delort-Tori,fang} in the particular case of cubic non-linearities in the following sense. First of all we can consider more general equations containing derivatives in the non-linearity (with ``small'' quasi-linear term). 
Furthermore, adapting our proof to the semi-linear case (i.e. when  $f=0$ in \eqref{KG} and \eqref{KGnon} and $G$ in \eqref{KGnon2} does not depend on $y_1$), we 
 obtain the better time of existence
 $\varepsilon^ {-{10/3}^{-}}$ for any $d\geq 4$. Indeed, in this case, the time of existence is $\e^{-2-4/\beta}$ with $\beta$ as above.
   This is the content of the next Theorem.

\begin{Teo}\label{main:KGsemi}
Consider \eqref{KG} with $f=0$ and $g$ independent of $y_1$.
Then the same  of Theorem \ref{main:KG} holds true, replacing $\mathtt{a}=3$ and $\mathtt{a}=8/3$
with $\mathtt{a}=4$ and $\mathtt{a}=10/3$ respectively.
\end{Teo}


{\bf Comments on the results.} We begin by discussing the \eqref{NLS} case.  
Our method covers also  more general cubic terms. 
For instance we could replace the term $|u|^2u$ 
with  $g(|u|^2)u$, where $g(\cdot)$ is any analytic 
function vanishing at the origin and 
having a primitive $G'=g$. We preferred not to write the paper 
in the most general case since the non-linearity $|u|^2u$ 
is a good representative for the aforementioned class 
and allows us to avoid to complicate the notation furtherly. 
We also remark that we consider 
  a class of potentials $V$ more general than the one we used in \cite{Feola-Iandoli-Long,Feola-Iandoli-Totale} and more similar to the one used in \cite{BG} in a semi-linear context.
  
  Secondly,
we  remark that, beside the mathematical interest, 
it would be very interesting, from a physical point of view, to be able to deal with the case $h(\tau)\sim \tau$ with $\tau\sim 0$. 
Indeed, for instance, if we choose $h(\tau)=\sqrt{1+\tau}-1$; 
the respective equation \eqref{NLS}  models the self-channeling 
of a high power, ultra-short laser pulse in matter, see \cite{boro-ga}. Unfortunately we need in our estimates $h(\tau)\sim \tau^{1+\s}$ with $\s>0$. More precisely we need the purely quasi-linear part 
of the equation $[\Delta (h(|u|^2))]h'(|u|^2)u$ to be \emph{smaller} ($O(\e^{3+4\s})$, $\e\ll1$) 
than the semi-linear one ($O(\e^3)$). 
At present we are not able to perform a normal form analysis which is able to reduce the size of the purely quasi-linear part.  
Whence, if such a quasi-linear term was $O(\e^{3})$, then the time of existence
we are able to obtain would not be better than $O({\e^{-2}})$.
Since $h$ has to be smooth this leads to  $h(\tau)\sim \tau^2$, $\tau\sim 0$.

Also  in the  \eqref{KG} case we are not 
able to deal with the interesting case of cubic quasi-linear term.
This is the reason why we require that the nonlinearity $f$ 
in \eqref{KGnon} has a zero of order at least $4$ at the origin. 

We introduce the following notation:
given $j_1,\ldots, j_p\in \mathbb{R}^+$, $p\geq2$ we define
\begin{equation}\label{maxii}
{\max}_{i}\{j_1, \ldots,j_p\}= i{\rm -th} \; {\rm largest \; among } \;j_1,\ldots,j_p\,.
\end{equation}
We use normal forms (the same strategy is used for \eqref{NLS} as well) and therefore  small divisors' problems arise. The small divisors, coming from the  four waves interaction, are of the form
\begin{equation}\label{fuffettaKG}
\Lambda_{\KG}(\x-\eta-\zeta)-\Lambda_{\KG}(\eta)
+\Lambda_{\KG}(\zeta)-\Lambda_{\KG}(\x)\end{equation}
with $\Lambda_{\KG}$ defined in \eqref{def:Lambda2}. In this case we prove the  lower bound (see \eqref{maxii})
\begin{equation}\label{NR-KG}
|\Lambda_{\KG}(\x-\eta-\zeta)-\Lambda_{\KG}(\eta)
+\Lambda_{\KG}(\zeta)-\Lambda_{\KG}(\x)|\gtrsim  {\max}_{2}\{|\x-\eta-\zeta|,|\eta|,|\zeta|\}^{-N_0}
 \max\{|\x-\eta-\zeta|,|\eta|,|\zeta|\}^{-\beta},
\end{equation}
for almost any value of the mass $m$ in the interval $[1,2]$ and where $\beta$ is any real number in the open interval 
$(3,4)$. 
The second factor in the r.h.s. of the above inequality 
represents a loss of derivatives when dividing by the 
quantity \eqref{fuffettaKG} which may be transformed 
in a loss of length of the lifespan through partition of frequencies. 
This is an extra difficulty, compared with the \eqref{NLS} case ({for which  lower bounds without loss have been proved in \cite{FaouGreNek}}), 
which makes the problem challenging already in a semi-linear setting. 
The estimate \eqref{NR-KG} with $\beta\in (3,4)$ 
has been already obtained 
in \cite{fang}. We provide here a different and simpler proof,
in the particular case of four waves interaction, 
which does not use the theory of
sub-analytic functions.
%
We also quote \cite{BFG} where Bernier-Faou-Gr\'ebert use 
a control of the small divisors involving only the largest index 
(and not ${\max}_2$ as in \eqref{NR-KG}). 
They obtained, in the semi-linear case, the  control of the 
Sobolev norm for a time $T\sim \e^{-\mathtt{a}}$,  with $\mathtt{a}$ arbitrary large, 
but assuming that the initial datum satisfies 
$\|\psi_0\|_{H^{s'+1/2}}+\|\psi_1\|_{H^{s'-1/2}}<c_0\e\,$ 
for some $s'\equiv s'(\mathtt{a})> s$, i.e. allowing a loss of regularity.

%

{\bf Ideas of the proof.} In our proof we shall use a \emph{quasi-linear normal forms/modified energies} approach, this seems to be the only successful  one in order to improve the time of existence implied by the local theory. We recall, indeed, that on $\T^d$ the dispersive character of the solutions is absent. Moreover, the lack of conservation laws and the quasi-linear nature of the equation prevent the use of \emph{semi-linear} techniques as done by  Bambusi-Gr\'ebert \cite{BG}    and Bambusi-Delort-Gr\'ebert-Szeftel 
\cite{BDGS}.

The most important  feature of equation \eqref{NLS} and \eqref{KG}, for our purposes,  is their Hamiltonian structure. This property guarantees some key cancellations in the \emph{energy-estimates} that will be explained later on in this introduction.

 The equation  \eqref{NLS} may be indeed rewritten as follows:
\begin{equation*}
\partial_tu=-\ii\nabla_{\bar{u}}\mathcal{H}_{\NLS}(u,\bar{u})=\ii\big(\Delta u-V*u-p(u)\big),
\end{equation*}
where $\nabla_{\bar{u}}:=(\nabla_{{\rm Re}(u)}+\ii\nabla_{{\rm Im}(u)})/2$,
 $\nabla$ denote the $L^{2}$-gradient, the Hamiltonian function $\mathcal{H}_{\NLS}$ and the nonlinearity $p$ are
 \begin{equation}\label{hamiltoniana}
 \begin{aligned}
 \mathcal{H}_{\NLS}(u,\bar{u}):=&\int_{\T^d}|\nabla u|^2+(V*u)\bar{u}+P(u,\nabla u)dx\,,\\
  P(u,\nabla u):=&\,\,\frac12\big(\left|\nabla \left(h(|u|^2)\right)\right|^2+|u|^4\big)\,, \quad p(u):=(\pa_{\bar{u}}P)(u,\nabla u)-\sum_{j=1}^{d}\pa_{x_j}
  \big(\pa_{\bar{u}_{x_{j}}}P\big)(u,\nabla u)\,.
\end{aligned} \end{equation}

The equation \eqref{KG} is Hamiltonian as well.   Thanks to  \eqref{KGnon}, \eqref{KGnon2}
we have that  \eqref{KG} can be written as 
 \begin{equation}\label{exHameq}
 \left\{
 \begin{aligned}
 &\pa_{t}\psi=\pa_{\phi}\mathcal{H}_{\KG}(\psi,\phi)=\phi\,,\\
 &\pa_{t}\phi=-\pa_{\psi}\mathcal{H}_{\KG}(\psi,\phi)=
 -\Lambda_{\KG}^{2}\psi-f(\psi)-g(\psi)\,,
 \end{aligned}\right.
 \end{equation}
where $\mathcal{H}_{\KG}(\psi,\phi)$ is the Hamiltonian
\begin{equation}\label{exHamKG}
\mathcal{H}_{\KG}(\psi,\phi)=
\int_{\mathbb{T}^{d}}\frac{\phi^{2}}{2}+\frac{(\Lambda_{\KG}^{2}\psi) \psi}{2}+F(\psi,\nabla \psi)+
G(\psi,\Lambda_{\KG}^{\frac{1}{2}}\psi) dx\,.
\end{equation}

We describe below our strategy in the case of the \eqref{NLS} equation.
 The strategy for \eqref{KG} is similar.

In \cite{Feola-Iandoli-local-tori} we prove an energy estimate, without any assumption of smallness on the initial condition, for a more general class of equations. This energy estimate, on the equation \eqref{NLS} with small initial datum, would read
\begin{equation}\label{energiafuffa}
E(t)-E(0)\lesssim \int_0^t\|u(\tau,\cdot)\|_{H^{s}}^2E(\tau)d\tau,
\end{equation}
where $E(t)\sim \|u(t,\cdot)\|_{H^{s}}^2$. An estimate of this kind implies, by a standard bootstrap argument, that the lifespan of the solutions is of order at least $O(\varepsilon^{-2})$, where $\varepsilon$ is the size of the initial condition. To increase the time to $O(\varepsilon^{-4})$ one would like  to show the improved inequality
\begin{equation}\label{energiaforte}
E(t)-E(0)\lesssim \int_0^t\|u(\tau,\cdot)\|_{H^{s}}^4E(\tau)d\tau.
\end{equation}
Our main goal is to obtain such an
 estimate.

{\textsc{Para-linearization of the equation \eqref{NLS}.}} The first step is the para-linearization, \emph{à la} Bony \cite{bony}, of the equation as a system of the variables $(u,\bar{u})$, see Prop. \ref{NLSparapara}. We rewrite \eqref{NLS} as a system of the form (compare with \eqref{QLNLS444})
\begin{equation*}
\partial_tU=-\ii E\big((-\Delta+V*)U+\mathcal{A}_2(U)U+\mathcal{A}_1(U)U\big)+X_{H_4}(U)+R(U), \quad E:=\sm{1}{0}{0}{-1}, \quad U:=\vect{u}{\bar{u}},
\end{equation*}
where $\mathcal{A}_2(U)$ is a $2\times 2$ self-adjoint  matrix of \emph{para-differential operators} of order two (see \eqref{matriceA2}, \eqref{simboa2}), $\mathcal{A}_1(U)$ is a self-adjoint, diagonal matrix of para-differential operators of order one (see \eqref{QLNLS444}, \eqref{simboa2}). These algebraic configuration of the matrices (in particular the fact that $\mathcal A_1(U)$ is diagonal) is a  consequence of the Hamiltonian structure of the equation.
The summand $X_{H_4}$ is the cubic term (coming from the para-linearization of $|u|^2u$, see \eqref{X_H}) and $\|R(U)\|_{H^{s}}$ is bounded from above by $\|U\|_{H^{s}}^7$ for $s$ large enough.\\
Both the matrices $\mathcal{A}_2(U)$ and $\mathcal{A}_1(U)$ vanish when $U$ goes to $0$. Since we assume that the function $h$, appearing in \eqref{NLS}, vanishes quadratically at zero, as a consequence of \eqref{simboa2}, we have that 
\begin{equation*}
\|\mathcal{A}_2(U)\|_{\mathcal{L}(H^{s};H^{s-2})}, \|\mathcal{A}_1(U)\|_{\mathcal{L}(H^{s};H^{s-1})}\lesssim\|U\|^6_{H^{s}}\,,
\end{equation*}
where by $\mathcal{L}(X;Y)$ we denoted the space of linear operators from $X$ to $Y$.
We also remark that the summand $X_{H_4}$ is an \emph{Hamiltonian vector field} with Hamiltonian function $H_4(u)=\int_{\T^d}|u|^4dx$.

{\textsc{Diagonalization of the second order operator.}}
The matrix of para-differential operators $\mathcal{A}_2(U)$ is not diagonal, therefore the first step, in order to be able to get at least the weak estimate \eqref{energiafuffa}, is to diagonalize the system at the maximum order. This is possible since, because of the smallness assumption, the operator $E(-\Delta+\mathcal{A}_2(U))$ is locally elliptic. In section \ref{blockdiago2} we introduce a new unknown $W=\Phi_{\NLS}(U)U$, where $\Phi_{\NLS}(U)$ is a parametrix built from the matrix of the eigenvectors of $E(-\Delta+\mathcal{A}_2(U))$, see \eqref{diago_para}, \eqref{matriceS}.
The system in the new coordinates reads 
\begin{equation*}
\partial_tW=-\ii E\big((-\Delta+V*)U+\mathcal{A}^{(1)}_2(U)W+\mathcal{A}_1^{(1)}(U)W\big)+X_{H_4}(W)+R^{(1)}(U),
\end{equation*}
where both $\mathcal{A}^{(1)}_2(U), \mathcal{A}^{(1)}_1(U)$ are diagonal, see 
\eqref{QLNLS444KK} and where $\|R^{(1)}(U)\|_{H^{s}}\lesssim \|U\|^7_{H^{s}}$ for $s$ large enough. We note also that the cubic vector field $X_{H_4}$ remains the same because the map $\Phi_{\NLS}(U)$ is equal to the identity plus a term vanishing at order six at zero, see \eqref{stime-descentTOTA}.

{\textsc{Diagonalization of the cubic vector-field.}} In the second step, in section \ref{cubo}, we diagonalize the cubic vector-field $X_{H_4}$. It is fundamental for our purposes to preserve the Hamiltonian structure of this cubic vector-field
in this diagonalization procedure. In view of this we perform a (approximatively) \emph{symplectic} change of coordinates generated from the Hamiltonian in \eqref{HamFUNC} and \eqref{gene1} (note that this is not the case for the diagonalization at order two). 
Actually the simplecticity of this change of coordinates is one of the most delicate points in our paper. The entire Section \ref{flow-ham} is devoted to this.
This diagonalization  is implemented in order to simplify a \emph{low-high} frequencies analysis. More precisely we prove that the cubic vector field may be conjugated to a diagonal one modulo a smoothing remainder. The diagonal part shall cancel out in the energy estimate  due to a symmetrization argument based on its Hamiltonian character. As a consequence the time of existence shall be completely determined by the smoothing reminder. Since this remainder is smoothing, the contribution coming from high frequencies is already ``small'', therefore the normal form analysis involves only the \emph{low modes}. This will be explained later on in this introduction.

We explain the result of this diagonalization. We define a new variable $Z=\Phi_{\mathcal{B}_{\NLS}}(W)$, see \eqref{novavarV}, and we obtain the new diagonal system (compare with \eqref{eq:ZZZ})
\begin{equation}\label{equazetaintro}
\partial_tZ=-\ii E\big((-\Delta+V*)Z+\mathcal{A}^{(1)}_2(U)Z+\mathcal{A}_1^{(1)}(U)Z\big)+X_{\mathtt{H}_4}(Z)+R^{(2)}_5(U),
\end{equation}
where the new vector-field $X_{\mathtt{H}_4}(Z)$ is still Hamiltonian, with Hamiltonian function defined in \eqref{Hamfinale}, and it is equal to a skew-selfadjoint and diagonal matrix of bounded  para-differential operators modulo smoothing reminders, see \eqref{cubicifine}. Here $R_5^{(2)}(U)$ satisfies the quintic estimates \eqref{stimaRRRfina}. 


{\textsc{Introduction of the energy-norm.}} 
Once achieved the diagonalization of the system we introduce an \emph{energy norm} which is equivalent to the Sobolev one. Assume for simplicity $s=2n$ with $n$ a natural number. Thanks to the smallness condition on the initial datum we prove in Section \ref{sec:highder} that  $\|(-\Delta\uno+\mathcal{A}_2(U)+\mathcal{A}_1(U))^{s/2}f\|_{L^2}\sim\|f\|_{H^s}$ for any function $f$ in $H^s(\T^d)$. Therefore by setting  \footnote{To be precise the definition of $Z_n=(z_n,\bar{z}_n)$ in \ref{sec:highder} is slightly different than the one presented here, but they coincide modulo smoothing 
 corrections. 
 For simplicity of notation, 
 and in order to avoid technicalities,  
 in this introduction we presented it in this way.}
\[
Z_n:=[E(-\Delta\uno+\mathcal{A}_2(U)+\mathcal{A}_1(U))]^{s/2}Z\,,
\]
 we are reduced to study the $L^2$ norm of the function $Z_n$.  This has been done in Lemma \ref{equivLemma}. 
 Since the system is now diagonalized, 
 we write the scalar equation, see Lemma \ref{equaVVnn}, solved by $z_n$ 
\begin{equation*}
\pa_{t}z_n=-\ii T_{\mathcal{L}}z_{n}-\ii V*z_n
-\Delta^n X^{+}_{\mathtt{H}_4}(Z)
+R_{n}(U)\,,
\end{equation*}
where we have denoted by $T_{\mathcal{L}}$ the element on the diagonal of the self-adjoint operator   $-\Delta\uno+\mathcal{A}_2(U)+\mathcal{A}_1(U)$, see \eqref{simboLLL}, \eqref{quantiWeyl};
$X^{+}_{\mathtt{H}_4}(Z)$ is the first component of the Hamiltonian vector-field $X_{\mathtt{H}_4}(Z)$ and $R_n(U)$ is a bounded remainder satisfying the quintic estimate  \eqref{stimeRRRN}.

{\textsc{Cancellations and normal-forms.}}  
At this point, still  in Lemma  \ref{equaVVnn}, we split the Hamiltonian vector-field  
$
X^{+}_{\mathtt{H}_4}=X^{+,{\rm res}}_{\mathtt{H}_4}+X^{+,{\bot}}_{\mathtt{H}_4},
$
where  $X^{+,{\rm res}}_{\mathtt{H}_4}$ is the \emph{resonant} part, see \eqref{alien7} and Definition \ref{resonantSet}. The first important fact, which is an effect of the Hamiltonian and Gauge preserving structure, is that the resonant term $\Delta^nX^{+,{\rm res}}_{\mathtt{H}_4}$ does not give any contribution to the energy estimates. 
This  key cancellation  may be interpreted 
as a consequence of 
 the fact that the \emph{super actions} 
 \[
 I_{p}:= \sum_{{j\in\mathbb{Z}^{d},\,
 |j|=p}}|\hat{z}(j)|^{2}\,,\quad p\in\mathbb{N}\,,\quad Z:=\vect{z}{\bar{z}}\,,
 \]
 where $\hat{z}$ is defined in \eqref{complex-uU},
 are prime integrals of the resonant Hamiltonian vector field 
$X_{\mathtt{H}_4}^{+,{\rm res}}(Z)$ in the same spirit\footnote{More generally, this cancellation can be viewed as a consequence of the commutation of the linear flow with the resonant part of the nonlinear perturbation which is a key of the Birkhoff normal form theory (see for instance \cite{Gre}).} of \cite{Faouplane}. This is the content of Lemma \ref{superazioni}, more specifically equation \eqref{superazioni1}. \\
We are left with the study of the term $-\Delta^nX^{+,{\bot}}_{\mathtt{H}_4}$. In Lemma \ref{equaVVnn} we prove that $-\Delta^nX^{+,{\bot}}_{\mathtt{H}_4}=B_{n}^{(1)}(Z)+B_{n}^{(2)}(Z)$, where $B_{n}^{(1)}(Z)$ does not contribute to energy estimates and   $B_{n}^{(2)}(Z)$ is  smoothing, gaining one space derivative, see \eqref{alien9} and Lemma \ref{lem:trilineare}. The cancellation for $B_{n}^{(1)}(Z)$ is again  a consequence of the Hamiltonian structure and it is proven in Lemma \ref{superazioni}, more specifically equation \eqref{cancello1}. \\
Summarizing we obtain the following energy estimate (see \eqref{scalarL})
\begin{align}
\frac12\frac{d}{dt}\|z_n(t)\|_{L^2}^2=&\,\Re(-\ii T_{\mathcal{L}}z_n,z_n)_{L^2}+\Re(-\ii V*z_n,z_n)_{L^2}\label{fuff1}\\
					   +&\,\Re(R_n(U),z_n)_{L^2}\label{fuff2}\\
					   +&\,\Re(-\Delta^nX_{\mathtt{H}_4}^{+,{\rm res}}(Z),z_n)_{L^2}\label{fuff3}\\
					   +&\, \Re(B_n^{(1)}(Z),z_n)_{L^2}\label{fuff4}\\
					   +&\, \Re(B_n^{(2)}(Z),z_n)_{L^2}\label{fuff5}.
\end{align}
The r.h.s. in \eqref{fuff1} equals to zero because 
$\ii T_{\mathcal{L}}$ is skew-self-adjoint and the Fourier coefficients of $V$ in \eqref{insPot} are real valued.  
The term \eqref{fuff2} is bounded from above by 
$\|z_n\|_{L^2}\|U\|_{H^s}^5$ 
; \eqref{fuff3} equals to zero thanks to \eqref{superazioni1}, the summand \eqref{fuff4} equals to zero as well because of \eqref{cancello1}. Setting $E(t)=\|z_n(t)\|_{L^2}^{2}$, the only term which is still not good in order to obtain an estimate of the form \eqref{energiaforte} is the \eqref{fuff5}.\\
In order to improve the time of existence we need to reduce the size of this new term \eqref{fuff5} by means of \emph{normal forms/integration by parts}. 
Our aim is to prove that
\begin{equation}\label{fuff6}
\int_0^t\Re (B_n^{(2)}(Z),z_n)_{L^2}(\sigma)d\sigma\lesssim\varepsilon^{2}
\end{equation}
as long as $t\lesssim \varepsilon^{-4}$ and $\|z_n\|_{L^2}\lesssim \varepsilon$. The thesis follows from this fact by using a classical bootstrap argument. Let us set $\mathcal{B}_{\NLS}(\sigma):=\Re (B_n^{(2)}(Z),z_n)_{L^2}(\sigma)$. The term $\mathcal{B}_{\NLS}$ may be expressed as (see  Prop. \ref{lem:basicenergy})
\begin{equation}\label{fuff7}
\mathcal{B}_{\NLS}\sim \sum_{\xi,\eta,\zeta\in\mathcal{R}^c}\langle\xi\rangle^{2n}b(\xi,\eta,\zeta)\hat{z}(\xi-\eta-\zeta)\hat{\bar{z}}(\eta)\hat{z}(\zeta)\hat{\bar{z}}(-\xi),
\end{equation}
the sum is restricted to the set of \emph{non resonant indexes}, 
see \eqref{resonantSet}, and  the coefficients satisfy
\begin{equation*}
|b(\xi,\eta,\zeta)|\lesssim \frac{\langle\xi\rangle^{2n}}{\max_1(\langle\xi-\eta-\zeta\rangle,\langle\eta\rangle,\langle\zeta\rangle)},
\end{equation*}
where the constant depends on ${\max_2(\langle\xi-\eta-\zeta\rangle,\langle\eta\rangle,\langle\zeta\rangle)}$ and where we have defined the Japanese bracket $\langle\xi\rangle:=\sqrt{1+|\xi|^2}$ for $\xi\in\mathbb{R}^d$.
We fix $N\in\R^+$ and we split 
$\mathcal{B}_{\NLS}:=\mathcal{B}_{\NLS,\leq N}+\mathcal{B}_{\NLS,>N}$ 
where $\mathcal{B}_{\NLS,\leq N}$ is as in \eqref{fuff7} with the sum restricted to those  
indexes such that $\max_1(\langle\xi-\eta-\zeta\rangle,\langle\eta\rangle,\langle\zeta\rangle)\leq N$. 
It is easy to show (see Lemma \ref{lem:iraq}) that 
$\int_0^t\|\mathcal{B}_{\NLS,>N}(\sigma)\|_{H^s}d\sigma\lesssim tN^{-1}\|z\|^4_{H^s}$. 
This is due to the fact that the coefficients $b(\xi,\eta,\zeta)$ are decaying.
Let us analyze the contribution given by $\mathcal{B}_{\NLS,\leq N}$.
\\We define the operator $\Lambda_{\NLS}$
as the Fourier multiplier acting on periodic functions as follows:
\begin{equation}\label{pollini1}
\Lambda_{\NLS} e^{\ii \x\cdot x}=\Lambda_{\NLS}(\x)e^{\ii\x\cdot x}\,,\qquad
\R\ni\Lambda_{\NLS}(\x):=|\x|^{2}+\hat{V}(\x)\,,
\quad \x\in\mathbb{Z}^{d}\,,
\end{equation}
where $\hat{V}(\x)$ are the \emph{real} Fourier coefficients 
of the convolution potential $V(x)$
given in \eqref{insPot}. Recalling \eqref{equazetaintro},  we have 
\begin{equation*}
\partial_t \hat{z}(\xi)=-\ii \Lambda_{\NLS}(\xi)\hat{z}(\xi)+\hat{Q}(\xi),
\end{equation*}
where $Q:=-\ii T_{\Sigma}z+ {X}_{\mathtt{H}_4}^+(z)+R_5^{(2)}$. $T_{\Sigma}$ is a paradifferential operator (see \eqref{quantiWeyl}) with symbol $\Sigma$, which is   real, of order two and homogeneity six in $z$, $R_5^{(2)}$ is a quintic reminder. 
We set $\hat{g}(\xi):= e^{\ii t\Lambda_{\NLS}(\xi)}\hat{z}(\xi)$ and we obtain
\begin{equation*}
\int_0^t\mathcal{B}_{\NLS,\leq N}(\s)d\s\sim \int_0^t\sum_{\xi,\eta,\zeta\in\mathcal{R}^c}
\mathtt{1}_{\{\max\{\langle\xi-\eta-\zeta\rangle,\langle\eta\rangle,\langle\zeta\rangle\}\leq N\}} b(\xi,\eta,\zeta) e^{-\ii\sigma\omega_{\NLS}(\xi,\eta,\zeta)}\hat{g}(\xi-\eta-\zeta)\hat{\bar{g}}(\eta)\hat{g}(\zeta)\hat{\bar{g}}(-\xi)\langle\xi\rangle^{2n}d\sigma,
\end{equation*}
with $\omega_{\NLS}$ defined in \eqref{fasePhi}.
Integrating by parts in $\sigma$, we obtain 
\begin{equation}\label{fuff11}
\begin{aligned}
\int_0^t\mathcal{B}_{\NLS,\leq N}(\sigma)d\s &\sim 
\int_{0}^{t}
(\mathcal{T}_{<}[z,\bar{z},z],T_{\langle\x\rangle^{2n}}(\pa_{t}+\ii\Lambda_{\NLS})z)_{L^{2}}d\s
\\&
+\int_{0}^{t}(\mathcal{T}_{<}[(\pa_{t}+\ii\Lambda_{\NLS})z,\bar{z},z],T_{\langle\x\rangle^{2n}}z)_{L^{2}}d\s
\\&
+\int_{0}^{t}(\mathcal{T}_{<}[z,\bar{z},(\pa_{t}+\ii\Lambda_{\NLS})z],T_{\langle\x\rangle^{2n}}z)_{L^{2}}d\s
\\&
+\int_{0}^{t}(\mathcal{T}_{<}[z,\ov{(\pa_{t}+\ii\Lambda_{\NLS})z},z],T_{\langle\x\rangle^{2n}}z)_{L^{2}}d\s
+O(\|u\|^{4}_{H^{s}})\,,
\end{aligned}
\end{equation}
where $\mathcal{T}_{<}(z_1,z_2,z_3)$ is the multilinear form whose Fourier coefficient is
\begin{equation*}
\widehat{\mathcal{T}_{<}}(\x)=\frac{1}{(2\pi)^d}\sum_{\eta,\zeta\in\mathbb{Z}^d}
\mathtt{t}_{<}(\x,\eta,\zeta)
\hat{z}_1(\x-\eta-\zeta)\hat{{z}}_2(\eta)\hat{z}_3(\zeta)\, ,\quad  \mathtt{t}_{<}(\x,\eta,\zeta)=\tfrac{-1}{\ii \omega_{\NLS}(\x,\eta,\zeta)}b(\x,\eta,\zeta).
\end{equation*}
The denominators $\omega_{\NLS}$ are never dangerous since we have very good lower bounds on them, see Prop. \ref{NonresCond} (see also Lemma \ref{lem:iraq}).
Let us consider, for instance, the first term in the r.h.s. of \eqref{fuff11}, we have 
\begin{equation*}
\begin{aligned}
\int_0^t(\mathcal{T}_{<}[z,\bar{z},z],T_{\langle\x\rangle^{2n}}(\pa_{t}+\ii\Lambda_{\NLS})z)_{L^{2}}(\sigma)d\s&=
\int_0^t(T_{\langle\x\rangle^{2}}\mathcal{T}_{<}[z,\bar{z},z],-T_{\langle\x\rangle^{2n-2}}\ii T_{\Sigma}z)_{L^{2}}(\s)d\s\\
&+\int_0^t\big(\mathcal{T}_{<}[z,\bar{z},z],T_{\langle\x\rangle^{2n}}(X_{\mathtt{H}_{\NLS}^{(4)}}^{+}(Z)
+R_{5}^{(2,+)}(U))\big)_{L^{2}}(\s)d\s\,.
\end{aligned}
\end{equation*}
The first term may be estimated by the Cauchy-Schwartz inequality obtaining 
\begin{equation}\label{fuff12}
\int_0^t\|\mathcal{T}_{<}(z,\bar{z},z)\|_{H^2}(\sigma)\|T_{\Sigma}z\|_{H^{s-2}}(\sigma)d\s.
\end{equation}
Since $\Sigma$ is a symbol of order two and homogeneity six, the second factor is bounded from above by $\varepsilon^6$ as soon as $\|z(\sigma)\|_{H^s}\lesssim\varepsilon$. Since $\mathcal{T}_{<}$ is supported on frequencies lower than $N$, the $\langle\xi\rangle^2$ symbol of $H^2$ norm, multiplied by the coefficients $b(\xi,\eta,\zeta)$ of the first term in \eqref{fuff12} provides a factor $N$ (see Lemma \ref{lem:iraq} for details), since it as homogeneity four we have also a factor $\varepsilon^4$ as soon as $\|z(\sigma)\|_{H^s}\lesssim\varepsilon$.   We have  eventually bounded \eqref{fuff12} by $t N \varepsilon^{10}.$ Analogously, the second term in \eqref{fuff12} may be bounded from above by $t\varepsilon^6$. \\
Recalling the contribution given by $\mathcal{B}_{\NLS,>N}$ we have bounded $\int_0^t\mathcal{B}_{\NLS}(\s)d\s$ from above by $t(\varepsilon^4N^{-1}+\varepsilon^{10}N+\varepsilon^6)+\varepsilon^4$. Choosing $N=\varepsilon^{-2}$ we immediately note that the last quantity stays of size $\varepsilon^2$ as soon as $t\lesssim \varepsilon^{-4}$.


As said before the strategy for \eqref{KG} is similar except for the control 
of the small divisors \eqref{NR-KG}.

We summarize the plan concerning \eqref{KG} focussing on the main differences with respect to \eqref{NLS}. In Section \ref{sec:3KG} we paralinearize the equation obtaining, after passing to the complex variables \eqref{simboa2KG}, the system of equations of order one \eqref{QLNLS444KG}. In Section \ref{DiagonaleKG} we diagonalize the system: the operator of order one is treated in  Prop. \ref{diagoOrd1} and the order zero in Prop. \ref{prop:blockdiag00KG}. As done for \eqref{NLS}, we diagonalize the operator of order zero paying attention to preserve its  Hamiltonian structure.  We consider the function $z$ solving \eqref{eq:ZZZKG} and we define the new variable $z_n:=\langle D\rangle^n z$, where $\langle D\rangle$ is the Fourier multiplier having symbol $\langle\xi\rangle$. We want to bound the $L^2$-norm of the variable $z_n$, which solves the equation \eqref{equa:VVTLKG}. The evolution of the $L^2$-norm is studied in Prop. \ref{lem:basicenergyKG}. From this proposition we understand that, in order to improve the energy estimates, we need to perform a normal form on the non resonant term $\mathcal{B}$ in \eqref{BBBTKG}, which has coefficients decaying as in \eqref{alienKG999}. 
We proceed as done in the \eqref{NLS} case. 
We fix $N\in\mathbb{R}^+$ and we split $\mathcal{B}$ in two pieces, 
one supported for frequencies such that 
$\max_1\{\langle\xi-\eta-\zeta\rangle,\langle\eta\rangle,\langle\zeta\rangle\}\leq N$ and the other for 
$\max_1\{\langle\xi-\eta-\zeta\rangle,\langle\eta\rangle,\langle\zeta\rangle\}> N$. 
The contribution to the energy estimate of the second one is $tN^{-1}\varepsilon^4$. Again in this point we exploit the smoothing property in \eqref{BBBTKG}. 
We focus on the part of $\mathcal{B}$ coming from small frequencies. We perform in Proposition \ref{prop:stimeEnergiaKG} an integration by parts in the same spirit of what done in the \eqref{NLS} case, see \eqref{iraqKG10}. When integrating by parts, the small denominators $\omega_{\KG}(\xi,\eta,\zeta)$ appear. In this case we do not have nice bounds as in the \eqref{NLS} case, indeed we only know that $|\omega_{\KG}(\xi,\eta,\zeta)|\gtrsim \max_1\{|\xi-\eta-\zeta|,|\eta|,|\zeta|\}^{-\beta}$, where $\beta$ is bigger then $3$ in dimension $d\geq 3$ and it is bigger that $2$ in dimension $d=2$. Hence such divisors give an extra factor $N^{\beta}$ in the energy estimates (recall that we are dealing with the case of small frequencies $\leq N$). 
After the integration by parts one has to use \eqref{eq:ZZZKGscal}. 
The term $\tilde{a}^+_2(x,\xi)\Lambda_{\KG}(\xi)$ therein has homogeneity $3$ and order $1$, 
so that its contribution to the energy estimates in \eqref{lemansKG1} is $tN^{\beta}\varepsilon^7$. 
Indeed the unboundedness of $\Lambda_{\KG}$ is compensated by the coefficients of $\mathcal{B}$, 
which gain one derivative. The vector field ${X}^+_{\mathtt{H}^{(4)}_{\KG}}(Z)$ has homogeneity $3$ 
and has no loss of derivatives, 
so that its contribution to \eqref{lemansKG2} 
is $tN^{\beta-1}\varepsilon^6$ (the ``$-1$'' is still coming from the coefficients of $\mathcal{B}$). 
The contribution of the reminder in \eqref{eq:ZZZKGscal} is negligible. We have a last term which is the one coming from the boundary term of the integration by parts which is bounded by $N^{\beta-1}\varepsilon^4$. Summarising we have obtained (compare with  \eqref{stimaeneNonresCaso2KG})
\begin{equation*}
\Big|\int_0^t\mathcal{B}(\sigma)d\sigma\Big|\lesssim t(\varepsilon^7N^{\beta}+\varepsilon^6N^{\beta-1}+\varepsilon^4N^{-1})+\varepsilon^4N^{\beta-1},
\end{equation*} 
where the term $\varepsilon^4N^{-1}t$ is coming from the high-frequencies of $\mathcal{B}$. Choosing $N:=\varepsilon^{-\frac{2}{\beta}}$ we note that the r.h.s. of the above inequality is controlled by $\varepsilon^2$ as soon as $t\lesssim \varepsilon^{-2(1+\frac{1}{\beta})}$, which is the time announced just after the statement of Theorem \ref{main:KG}.\\
We conclude this introduction explaining the numerology of Theorem \ref{main:KGsemi}. In the semilinear case we have $f=0$ and $g$ independent of $y_1$ in \eqref{KG}, so the there are no derivatives in the nonlinearity. When we pass to the system of order $1$ in \eqref{QLNLS444KG}, one has  $\mathcal{A}_1\equiv 0$ and that the cubic term $X_{\mathcal{H}^{(4)}_{\KG}}(U)$ may be decomposed as a paradifferential operator of order -1 plus a trilinear reminder whose coefficients have the better (compared to the quasilinear case  \eqref{alienKG999}) decay $\max_{1}(\langle\xi-\eta-\zeta\rangle,\langle \eta\rangle,\langle\zeta\rangle)^{-2}$ (see Remark \ref{semilin1}). We perform the integration by parts as in the quasi-linear case. Here we do not have the contribution coming from $\tilde{a}_2^+(x,\xi)\Lambda_{\KG}(\xi)$ (because this term equals to zero in the semilinear case) which was $\varepsilon^7N^{\beta}$. Moreover the contribution of the cubic semilinear term is  $\varepsilon^6N^{\beta-2}$ (instead of $\varepsilon^6N^{\beta-1}$ as before), thanks to the better decay of the
coefficients in the cubic reminder. The high frequency part is also smaller and it gives $N^{-2}\varepsilon^6$, instead of $N^{-1}\varepsilon^6$. One eventually obtains $|\int_0^t\mathcal{B}(\s)d\s|\lesssim t(\varepsilon^6N^{\beta-2}+\varepsilon^4N^{-2})+\varepsilon^4N^{\beta-2}$. If one choses $N=\varepsilon^{-\frac{2}{\beta}}$ one can bound the previous quantity as soon as $t\lesssim \varepsilon^{-(2+\frac{4}{\beta})}$, which means $t\lesssim \varepsilon^{-\frac{10}{3}^-}$ when $d\geq3$ and $t\lesssim \varepsilon^{-4^-}$ if $d=2$. 

\bigskip
{{\textsc{Acknowledgements.}}}  
We would like to warmly thank prof. Fabio Pusateri, Joackim Bernier and the anonymous referees  
for their useful comments.
\bigskip

\section{Small divisors}
As pointed out in the introduction the proofs our main theorems are based on a normal form approach. As a consequence we shall deal with small divisors problems.
This section is devoted to establish suitable lower bounds for generic (in a probabilistic way) choices of the parameters ($x_{\xi}$ in \eqref{insPot} for \eqref{NLS} and $m$ in \eqref{def:Lambda2} for \eqref{KG})  excepted for exceptional indices for which the small divisor is identically zero.

\subsection{Non-resonance conditions for (\ref{NLS})}\label{sec:NonresCondNLS}
In the following proposition we give lower bounds for the small divisors arising from the normal form for \eqref{NLS}.
\begin{proposition}\label{NonresCond}
Consider
the phase $\omega_{\NLS}(\x,\eta,\zeta)$ 
defined as
\begin{equation}\label{fasePhi}
\omega_{\NLS}(\x,\eta,\zeta):=\Lambda_{\NLS}(\x-\eta-\zeta)-\Lambda_{\NLS}(\eta)
+\Lambda_{\NLS}(\zeta)-\Lambda_{\NLS}(\x)\,,
\qquad (\x,\eta,\zeta)\in \mathbb{Z}^{3d}
\end{equation}
where $\Lambda_{\NLS}$ is in \eqref{pollini1}
and  the potential $V$ is in \eqref{insPot}. We have the following.

\noindent $(i)$ Let $d\geq 2$. There exists $\mathcal{N}\subset \mathcal{O}$ 
 with zero Lebesgue measure such that, for any 
 $(x_{i})_{i\in \mathbb{Z}^{d}}\in \mathcal{O}\setminus \mathcal{N}$,
there exist $\g>0$, $N_0:=N_0(d,m)>0$  ($m>d/2$ see \eqref{insPot}) such that
for any $(\x,\eta,\zeta)\notin \mathcal{R}$ 
(see \eqref{resonantSet})
one has
\begin{equation}\label{lowerPhistrong}
|\omega_{\NLS}(\x,\eta,\zeta)|\geq \gamma \max_{2}\{|\x-\eta-\zeta|,|\eta|,|\zeta|\}^{-N_0}\,.
\end{equation}
$(ii)$ Let $d=1$ and assume that $V\equiv0$. Then one has 
$|\omega_{\NLS}(\x,\eta,\zeta)|\gtrsim 1$ unless
\begin{equation}\label{nantesnantes}
\x=\zeta\,,\;\;\;\eta=\x-\eta-\zeta\,,\quad {\rm or}\quad  \x=\x-\eta-\zeta\,,\;\;\;\eta=\zeta\,, \,\,\quad  \xi,\eta,\zeta\in\Z.
\end{equation}
\end{proposition} 

\begin{proof}
Item $(i)$
follows by Proposition $2.8$ in \cite{FaouGreNek}.
Item $(ii)$ is classical.
\end{proof}

\subsection{Non-resonance conditions for (\ref{KG})}\label{sec:NonresCondKG}

Recall the symbol  $\Lambda_{\KG}(j)$ in \eqref{def:Lambda2}.
 We shall prove the following important proposition.
%
\begin{proposition}{\bf (Non-resonance conditions).}\label{NonresCondKG}
Consider
the phase $\omega_{\KG}^{\vec{\s}}(\x,\eta,\zeta)$ 
defined as
\begin{equation}\label{fasePhiKG}
\omega_{\KG}^{\vec{\s}}(\x,\eta,\zeta):=\s_1\Lambda_{\KG}(\x-\eta-\zeta)+\s_2\Lambda_{\KG}(\eta)
+\s_3\Lambda_{\KG}(\zeta)-\Lambda_{\KG}(\x)\,,
\qquad (\x,\eta,\zeta)\in \mathbb{Z}^{3d},
\end{equation}
where $\vec{\s}:=(\s_1,\s_2,\s_3)\in\{\pm\}^{3}$, 
$\Lambda_{\KG}$ is in \eqref{def:Lambda2}. Let $0<\s\ll 1$
and set 
$\beta:=2+\s$ if $d=2$, and $\beta:=3+\s$ if $d\geq3$. 
There exists $\mathcal{C}_{\beta}\subset [1,2]$ 
 with Lebesgue measure $1$ such that, for any 
 $m\in  \mathcal{C}_{\beta}$,
there exist $\g>0$, $N_0:=N_0(d,m)>0$ such that
for any $(\x,\eta,\zeta)\notin \mathcal{R}$ 
(see \eqref{resonantSet}) 
one has
\begin{equation}\label{lowerPhistrongKG}
|\omega_{\KG}^{\vec{\s}}(\x,\eta,\zeta)|\geq \gamma \max_{2}\{|\x-\eta-\zeta|,|\eta|,|\zeta|\}^{-N_0}
 \max\{|\x-\eta-\zeta|,|\eta|,|\zeta|\}^{-\beta}\,.
\end{equation}
\end{proposition} 
The case $d=2$ follows by Theorem $2.1.1$ in  \cite{Delort-Tori}, the rest of this subsection is devoted to the proof of Proposition \ref{NonresCondKG} in the case $d\geq 3$. Throughout this subsection, in order to lighten the notation, we shall write $\Lambda_{\KG}(j)\rightsquigarrow \Lambda_{j}$
for any $j\in \mathbb{Z}^{d}$ and $d\geq3$.
The main ingredient  is the following.
\begin{proposition}\label{mainNR}
Let $4>\beta >3$, there exist $\alpha>0$ and  $\mathcal{C}_\beta
\subset [1,2]$ a set of Lebesgue measure $1$ and for $m\in\Cc_\beta$  
there exists $\kappa(m)>0$ such that
\begin{equation}\label{NR}
|\sigma_1{\Lambda}_{j_1}
+\sigma_2{\Lambda}_{j_2}+\sigma_3{\Lambda}_{j_3}+\sigma_4{\Lambda}_{j_4} |
 \geq \frac{\kappa(m)}{|j_3|^{\alpha}|j_1|^{\beta}}
 \end{equation}
for all $\sigma_1,\sigma_2,\sigma_3,\sigma_4\in\{-1,1\}$, 
$j_1,j_2,j_3,j_4\in\Z^d$ satisfying 
$|j_1|\geq|j_2|\geq|j_3|\geq |j_4|$ and 
$\sigma_1 j_1+\sigma_2 j_2+\sigma_3 j_3+\sigma_4 j_4=0$, 
except when 
$\sigma_1=\sigma_4=-\sigma_2=-\sigma_3$ 
and $|j_1|=|j_2|\geq|j_3|= |j_4|$.
\end{proposition}
The Proposition \ref{mainNR}  implies  Proposition \ref{NonresCondKG}. 
Its proof
is done in three steps.
\vspace{0.5cm}

\noindent {\bf Step 1: control with respect to the highest index.}
\begin{lemma}\label{mu1}
There exist $\nu>0$ and  $\mathcal{M}_\nu
\subset [1,2]$ a set of Lebesgue measure $1$ 
and for $m\in\mathcal{M}_\nu$  there exists $\gamma(m)>0$ such that
\begin{equation}\label{NR1}
|\sigma_1{\Lambda}_{j_1}+\sigma_2{\Lambda}_{j_2}
+\sigma_3{\Lambda}_{j_3}+\sigma_4{\Lambda}_{j_4} | 
\geq \gamma(m)|j_1|^{-\nu}
\end{equation}
for all $\sigma_1,\sigma_2,\sigma_3,\sigma_4\in\{-1,1\}$, 
$j_1,j_2,j_3,j_4\in\Z^d$ satisfying 
$|j_1|\geq|j_2|\geq|j_3|\geq |j_4|$, 
except when $\sigma_1=\sigma_4=-\sigma_2=-\sigma_3$ 
and $|j_1|=|j_2|\geq|j_3|= |j_4|$.
\end{lemma}
The proof of this Lemma is standard and follows the line of  Theorem $6.5$ in
\cite{Bam03}, see also \cite{BG} or \cite{EGK}. We briefly repeat the steps.\\
Let us assume that $j_1,j_2,j_3,j_4\in\Z^d$ satisfy 
$ |j_1|>|j_2|>|j_3|>|j_4|$.
First of all, by reasoning as in 
Lemma $3.2$ in \cite{EGK}, one can deduce the following.
\begin{lemma}\label{det} 
Consider the matrix $D$ whose entry at place $(p,q)$ is given by
$\frac{d^{p}}{d m^{p}}\Lambda_{j_q}$, $p,q=1,\ldots,4$.
The modulus of the determinant of $D$
is bounded from below: one has $|{\rm det}(D)|\geq C |j_1|^{-\mu}$
where $C>0$  and $\mu>0$ are universal constants.
\end{lemma}

From Lemma $3.3$ in \cite{EGK} we learn
\begin{lemma}\label{m1.1}
 Let
$u^{(1)},...,u^{(4)}$ be $4$ independent vectors in $\R^4$ with
$\|u^{(i)}\|_{\ell^1}\leq1$. Let $w\in\R^4$ be an arbitrary
vector, then there exist $i\in[1,\cdots,4]$, such that
$|u^{(i)}\cdot w|\geq C{\|w\|_{\ell^1}\det(u^{(1)},\ldots,u^{(4)})}$.
\end{lemma}
Let us define
\[
\psi_{\KG}(m)=\sigma_1{\Lambda}_{j_1}(m)+\sigma_2{\Lambda}_{j_2}(m)
+\sigma_3{\Lambda}_{j_3}(m)+\sigma_4{\Lambda}_{j_4}(m)\,.
\]
Combining Lemmata \ref{det} and \ref{m1.1} we deduce the following.
\begin{corollary}
\label{m1.2}For any $m\in [1,2]$ there exists an index $i\in[1,\cdots,4]$ such
that
$\left|\frac{d^i \psi_{\KG}}{dm^i}(m)\right|\geq
C|j_1|^{-\mu}.$
\end{corollary}

Now we need the following result (see Lemma $B.1$ in \cite{H98}):

\begin{lemma}
\label{v.112}
Let  $\mathfrak{g}(x)$ be a $C^{n+1}$-smooth function on the segment $[1,2] $
such that 
\[
|\mathfrak{g}'|_{C^n} =\beta \quad {\rm and}\qquad  
\max_{1\le k\le n}\min_x|\partial^k \mathfrak{g}(x)|=\sigma\,.
\]
 Then 
\[
\mathrm{meas}(\{x\mid |\mathfrak{g}(x)|\leq\rho\} )\leq 
C_n \left(\frac{\beta}{\sigma}+1\right) 
\left(\frac{\rho}{\sigma}\right)^{1/n}\,.
\]
\end{lemma}
\noindent
Define
\[
\mathcal E_j(\kappa):=\big\{m\in[1,2]\mid |\sigma_1{\Lambda}_{j_1}
+\sigma_2{\Lambda}_{j_2}+\sigma_3{\Lambda}_{j_3}
+\sigma_4{\Lambda}_{j_4}|\leq\kappa|j_1|^{-\nu}
\big\}\,.
\]
By combining Corollary \ref{m1.2} and Lemma \ref{v.112} we get 
\begin{align}\label{trenoginevra}
\mathrm{meas}\ (\mathcal E_j(\kappa))&\leq C|j_1|^\mu 
\left(\kappa|j_1|^{\mu-\nu}\right)^{1/4}\leq 
C\kappa^{\frac14}|j_1|^{\frac{5\mu-\nu}{4}} \,.
\end{align}
Define
\[
\mathcal E(\kappa)=\bigcup_{ |j_1|>|j_2|>|j_3|>|j_4|} \mathcal E_j(\kappa)\,,
\]
and set $\nu= 5\mu + 4(4d+1)$. Then the \eqref{trenoginevra} 
implies 
$\mathrm{meas}( \mathcal E(\kappa))\leq C\kappa^{\frac14}$.
Then taking 
$m\in \bigcup_{\kappa>0}\left([1,2]\setminus \mathcal E(\kappa)\right)$ 
we obtain \eqref{NR1} for any 
$ |j_1|>|j_2|>|j_3|>|j_4|$. 
Furthermore 
$ \bigcup_{\kappa>0}\left([1,2]\setminus \mathcal E(\kappa)\right)$ 
has measure $1$. 
Now if for instance $|j_1|=|j_2|$ then we are 
left with a small divisor of the type 
$|2{\Lambda}_{j_1}+\sigma_3{\Lambda}_{j_3}+\sigma_4{\Lambda}_{j_4}|$ 
or $|{\Lambda}_{j_3}+\sigma_4{\Lambda}_{j_4}|$, 
i.e. involving $2$ or $3$ frequencies. 
So following the same line we can also manage this case.

\vspace{0.5cm}
\noindent{\bf Step 2: control with respect to the third highest index.}
In this subsection we show that small dividers can be 
controlled by a smaller power of $|j_1|$  even if it means 
transferring part of the weight to $|j_3|$.
\begin{proposition}\label{mu3}
Let $4>\beta >3$, there exists  $\mathcal{N}_\beta
\subset [1,2]$ a set of Lebesgue measure 1 and for $m\in\Nc_\beta$  there exists $\kappa(m)>0$ such that
$$
|\Lambda_{j_1}-\Lambda_{j_2}+\sigma_3\Lambda_{j_3}+\sigma_4\Lambda_{j_4}| \geq \frac{\kappa(m)}{|j_3|^{2d+6}|j_1|^{\beta}}
$$
for all $\sigma_3,\sigma_4\in\{-1,+1\}$, for all $j_1,j_2,j_3,j_4\in\Z^d$ 
satisfying 
$|j_1|>|j_2|\geq|j_3|> |j_4|$, the momentum condition  $j_1-j_2+\sigma_3j_3+\sigma_4j_4=0$
and 
\[ 
|j_1|\geq J(\kappa,|j_3|):=
\big(\tfrac{C}{\kappa}\big)^{\frac1{4-\beta}}|j_3|^{\frac{2d+11}{4-\beta}}
\] 
where $C$ is a universal constant.
\end{proposition}
We begin with two elementary lemmas
\begin{lemma}\label{lem1} Let $\sigma=\pm1$, $j,k\in\Z^d$, with 
$|j|> |k|>0$ and  $|j|\geq8$, and $[1,2]\ni m\mapsto \mathfrak{g}(m)$ a 
$C^1$ function satisfying 
$|\mathfrak{g}'(m)|\leq \frac1{10|j|^3}$ for $m\in[1,2]$. 
For all $\kappa>0$ there exists 
$\Dc\equiv\Dc(j,k,\sigma,\kappa,\mathfrak{g})\subset[1,2]$ 
such that for $m\in\Dc$
\[
|\Lambda_j+\sigma\Lambda_{k}- \mathfrak{g}(m)|\geq \kappa\]
and
\[
\text{meas}([1,2]\setminus\Dc)\leq 10\kappa|j|^3\,.
\]
\end{lemma}

\begin{proof}
Let $f(m)=\Lambda_j+\sigma \Lambda_k- \mathfrak{g}(m)$. In the case $\s=-1$, which is the worst, we have
\begin{align*} f'(m)&= \frac12 (\frac1{\sqrt{|j|^2+m}}
-\frac1{\sqrt{|k|^2+m}})-\mathfrak{g}'(m)\\&=
\frac{|k|^2-|j|^2}{2(\sqrt{|j|^2+m}+\sqrt{|k|^2+m})
\sqrt{|j|^2+m}\sqrt{|k|^2+m}}-\mathfrak{g}'(m)\,.
\end{align*}
We want to estimate $|f'(m)|$ from above. By using that 
$4(|j|^2+2)^{\frac32}\leq 5|j|^3$ for $|j|\geq 8$ we get
$$|f'(m)|\geq \frac{1}{5|j|^3}-\frac1{10|j|^3}\geq \frac1{10|j|^3}.$$
In the case $\sigma=1$, the same bound holds true.
Then we conclude by a standard argument that
\[
\text{meas}\{m\in[1,2]\mid |f(m)|\leq \kappa \}\leq 10\kappa|j|^3\,,
\]
which is the thesis.
\end{proof}

\begin{lemma}\label{lem2} Let $j,k\in\Z^d$ with $|j|\geq |k|$ and $|j-k|\leq |j|^{\frac12}$ then
\begin{equation}\label{taylor}\Lambda_j-\Lambda_k= \frac{(j,j-k)}{|j|}+\mathfrak{g}(|j|,|j-k|,(j-k,j),m)+
O(\frac{|j-k|^5}{|j|^4})\end{equation}
for some explicit rational function $\mathfrak{g}$.\\
Furthermore 
\begin{align}
\label{g'}|\partial_m \mathfrak{g}(|j|,|j-k|,(j,j-k),m)|\leq \frac 1{2|j|^{\frac32}},\\
\label{g}| \mathfrak{g}(|j|,|j-k|,(j,j-k),m)|\leq \frac {3|j-k|^2}{|j|}
\end{align}
 uniformly with respect to $j,k\in\Z^d$ with $|j|\geq |k|$, $|j-k|\leq |j|^{\frac12}$ and $|j|$ large enough.
\end{lemma}
\proof
By Taylor expansion we have for $|j|$ large
\[
\Lambda_j= |j|(1+\frac{m}{|j|^2})^{\frac12}= |j|+ \frac{m}{2|j|}
-  \frac{m^2}{8|j|^3}+O(\frac1{|j|^5})
\]
and
\begin{align*}
\Lambda_k&= |j|(1+\frac{2(k-j,j)+|j-k|^2+m}{|j|^2})^{\frac12}
\\&= |j|+ \frac{2(k-j,j)+|j-k|^2+m}{2|j|}-  \frac{(2(k-j,j)+|j-k|^2+m)^2}{8|j|^3}
\\&+\frac3{48} \frac{(2(k-j,j)+|j-k|^2+m)^3}{|j|^5}-\frac{15}{16} \frac{1}{4!}
\frac{(2(k-j,j)+|j-k|^2+m)^4}{|j|^7}  +O(\frac{|j-k|^5}{|j|^4}) 
\end{align*}
which leads to \eqref{taylor} where with 
(we use that $|(k-j,j)|\leq |j-k||j|$ and $|j-k|\leq |j|^{\frac12}$)
\begin{align*}
\mathfrak{g}(x,y,z,m)= \frac{-y^2}{2x}+\frac{(-2z+y^2+m)^2-m^2}{8x^3}+\frac3{48}\frac{8z^3-12z^2(y^2+m)}{x^5}+\frac{1}{4!}\frac{15}{16}\frac{16z^4}{x^7}\,.
\end{align*}
\endproof

We are now in position to prove the main result of this subsection.

\begin{proof}[{\bf Proof of Proposition \ref{mu3}}]
We want to control the small divisor 
$$\Delta= \Lambda_{j_1}-\Lambda_{j_2}+\sigma_3\Lambda_{j_3}+\sigma_4\Lambda_{j_4}.$$
Let $\mathfrak{g}$ be the rational function introduced in Lemma \ref{lem2}. We write
\begin{align*}\Delta =&\sigma_3\Lambda_{j_3}+\sigma_4\Lambda_{j_4}+ \frac{(j_1,j_1-j_2)}{|j_1|}\\
&+ \mathfrak{g}(|j_1|,|j_1-j_2|,(j_1,j_1-j_2),m)+ O(\frac{|j_1-j_2|^5}{|j_1|^4})  .\end{align*}
Remember that by assumption $j_1-j_2+\sigma_3j_3+\sigma_4j_4=0$ and in particular $|j_1-j_2|\leq 2|j_3|$.\\
 Fix $\gamma>0$. 
 Choosing $\kappa=\frac{\gamma}{|j_3|^{2d+6}|j_1|^\beta}$ 
 in Lemma \ref{lem1} and assuming  $2|j_3|\leq |j_1|^{\frac12}$ 
 we have by Lemma \ref{lem1} and Lemma \ref{lem2}
\[
|\Delta|\geq \frac{\gamma}{|j_3|^{2d+6}|j_1|^\beta}
- C\frac{|j_3|^5}{|j_1|^4}\geq \frac{\gamma}{2|j_3|^{2d+6}|j_1|^\beta}
\]
as soon as 
$$
|j_1|\geq \big(\frac{C}{\gamma}\big)^{\frac1{4-\beta}}|j_3|^{\frac{2d+11}{4-\beta}}
=:J(\gamma,|j_3|)\geq\footnote{Note that this estimate implies that 
$|\partial_m \mathfrak{g}(|j_1|,[j_1-j_2|, (j_1-j_2,j_1),m)|
\leq \frac1{2|j_1|^{3/2}}\leq \frac1{10|j_3|^3}$ 
and thus Lemma \ref{lem1} applies. }\ 5|j_3|^3
$$
(where $C$ is an universal constant) and $m\in \Dc(j_3,j_4,\sigma,\kappa, \sigma_3  \mathfrak{g}(|j_1|,|j_1-j_2|,(j_1,j_1-j_2),\cdot))$ (the set $\Dc$ is defined in Lemma \ref{lem1} and we set  $\sigma=\sigma_3\sigma_4$). 
Then denoting
\begin{align*}\Cc(\gamma,j_3,j_4,\sigma_3,\sigma_4):=&\{m\in[1,2]\mid |\Delta|\geq \frac{\gamma}{2|j_3|^{2d+6}  |j_1|^\beta},\ \\ \forall (j_1,j_2) \text{ such that }
& |j_1|\geq \max(|j_2|,J(\gamma,|j_3|)),\ j_1-j_2+\sigma_3j_3+\sigma_4 j_4=0\}\end{align*}
we have 
$$\Cc(\gamma,j_3,j_4,\sigma_3,\sigma_4)=\bigcap_{ \mathfrak{g}} \Dc(j_3,j_4,\sigma,\frac{\gamma}{|j_3|^{2d+6}|j_1|^\beta}, \sigma_3\mathfrak{g}(|j_1|,|j_1-j_2|,(j_1,j_1-j_2),\cdot))$$
where the intersection is taken over all functions 
$\mathfrak{g}$ generated by 
$(j_1,j_2)\in (\Z^d)^2$ such that 
$$ 
|j_1|\geq \max(|j_2|,J(\gamma,|j_3|))
$$ 
and $j_1-j_2+\sigma_3j_3+\sigma_4 j_4=0$. Thus
by Lemma \ref{lem1}
\begin{align*}\text{meas}\ &\big( [1,2]\setminus 
\Cc(\gamma,j_3,j_4,\sigma_3,\sigma_4)\big)
\leq \\ &\sum_{n\geq 1}
\frac{10\gamma}{|j_3|^{2d+3}n^{\frac\beta2}}\#\{(|j_1|,|\sigma_3j_3+\sigma_4 j_4|,(j_1,\sigma_3j_3+\sigma_4 j_4))\mid j_1\in\Z^d,\ |j_1|^2=n\}\,.
\end{align*}
The scalar product $(j_1,\sigma_3j_3+\sigma_4 j_4))$ 
takes only integer values smaller than $2|j_1||j_3|$.
Then, since $\beta>3$, we get
$$
\text{meas}\ \Cc(\gamma,j_3,j_4,\sigma_3,\sigma_4)\leq 
\frac{20\gamma}{|j_3|^{2d+2}}\sum_{n\geq1}\frac1{n^{\frac{\beta-1}{2}}}\leq C_\beta \frac\gamma{|j_3|^{2d+2}}\,.
$$
Then it remains to define 
$$\Nc_\beta=\cup_{\gamma>0}\bigcap_{\substack{(j_3,j_4)\in(\Z^d)^2\\|j_4|\leq|j_3|\\ \sigma_3,\sigma_4\in\{-1,1\}}} \Cc(\gamma,j_3,j_4,\sigma_3,\sigma_4)$$
to conclude the proof.
\end{proof}

\vspace{0.5cm}

{\bf Step 3: proof of Proposition \ref{mainNR}}
We are now in position to prove Proposition \ref{mainNR}. 
Let $\sigma_1,\sigma_2,\sigma_3,\sigma_4\in\{-1,1\}$, $j_1,j_2,j_3,j_4\in\Z^d$ satisfying 
$|j_1|\geq|j_2|\geq|j_3|\geq |j_4|$ and $\sigma_1 j_1+\sigma_2 j_2+\sigma_3 j_3+\sigma_4 j_4=0$. If  $\sigma_1=\sigma_2$, then, since $|j_1|\geq|j_2|\geq|j_3|\geq |j_4|$, we conclude that the associated small divisor cannot be small except if $\sigma_1=\sigma_2=-\sigma_3=-\sigma_4$. 
Then we have to control  
$|\Lambda_{j_1}+\Lambda_{j_2}-\Lambda_{j_3}-\Lambda_{j_4} | $ 
knowing that $|j_1|\geq|j_2|\geq|j_3|\geq |j_4|$. 
We first notice that if $|j_1|^2\leq |j_3|^2+1$, then we can conclude using Lemma \ref{mu1} that \eqref{NR} 
is satisfied with $\alpha=\nu$ for $m\in\Mc_\nu$. 
On the other hand if $|j_1|^2\geq |j_3|^2+1$ then 
$$
\Lambda_{j_1}+\Lambda_{j_2}-\Lambda_{j_3}-\Lambda_{j_4} 
\geq  
\Lambda_{j_1}-\Lambda_{j_3}
\geq \frac{\Lambda_{j_1}^2-\Lambda_{j_3}^2}{\Lambda_{j_1}+\Lambda_{j_3}} 
\geq 
\frac1{2\sqrt{|j_3|^2+2}}
$$ 
which implies \eqref{NR}. 
Thus we can assume $\sigma_1=-\sigma_2$ and we can apply Proposition \ref{mu3} 
which implies the control \eqref{NR} for $m\in \Nc_\beta$ with $\alpha=2d+3$ 
under the additional constraint $|j_1|\geq J(\gamma(m),|j_3|)$. 
Now if $|j_1|\leq J(\gamma(m),|j_3|)$ we can apply Lemma \ref{mu1} 
to obtain that there exists $\nu>0$ and full measure set $\Mc_\nu$ such that for 
$m\in\Mc_\nu\cap\Nc_\beta:=\Cc_\beta$ there exists $\kappa(m)>0$ such that 
$$
|\sigma_1\Lambda_{j_1}+\sigma_2\Lambda_{j_2}+\sigma_3\Lambda_{j_3}+\sigma_4\Lambda_{j_4} | 
\geq \frac{\kappa(m)}{|j_1|^{\nu}}\geq 
\frac{\kappa(m)}{J(\gamma(m),|j_3|)^{\nu}}=C \frac{\kappa(m)\gamma(m)^{4-\beta}}{|j_3|^\alpha}
$$
with $\alpha=\nu\frac{2d+8}{4-\beta}$ which, of course, implies \eqref{NR}.

\section{Functional setting}\label{sec:2}
We denote by $H^{s}(\mathbb{T}^d;\mathbb{C})$
(respectively $H^{s}(\mathbb{T}^d;\mathbb{C}^{2})$)
the usual Sobolev space of functions 
$\mathbb{T}^{d}\ni x \mapsto u(x)\in \mathbb{C}$
(resp. $\C^{2}$).
We expand a function $ u(x) $, $x\in \mathbb{T}^{d}$, 
 in Fourier series as 
\be\label{complex-uU}
u(x) = \frac{1}{(2\pi)^{d/2}}
\sum_{n \in \Z^{d} } \hat{u}(n)e^{\ii n\cdot x } \, , \qquad 
\hat{u}(n) := \frac{1}{(2\pi)^{d/2}} 
\int_{\mathbb{T}^{d}} u(x) e^{-\ii n\cdot x } \, dx \, .
\ee
We set $\langle j \rangle:=\sqrt{1+|j|^{2}}$ for $j\in \mathbb{Z}^{d}$.
We endow $H^{s}(\mathbb{T}^{d};\mathbb{C})$ with the norm 
\begin{equation}\label{Sobnorm}
\|u(\cdot)\|_{H^{s}}^{2}:=\sum_{j\in \mathbb{Z}^{d}}\langle j\rangle^{2s}|\hat{u}(j)|^{2}\,.
\end{equation}
For $U=(u_1,u_2)\in H^{s}(\mathbb{T}^d;\mathbb{C}^{2})$
we  set
$\|U\|_{H^{s}}=\|u_1\|_{H^{s}}+\|u_2\|_{H^{s}}$.
Moreover, for $r\in\R^{+}$, we 
denote by $B_{r}(H^{s}(\mathbb{T}^{d};\mathbb{C}))$
(resp. $B_{r}(H^{s}(\mathbb{T}^{d};\mathbb{C}^2))$)
the ball of $H^{s}(\mathbb{T}^{d};\mathbb{C}))$ 
(resp. $H^{s}(\mathbb{T}^{d};\mathbb{C}^2))$)
with radius $r$
centered at the origin.
We shall also write the norm in \eqref{Sobnorm} as
$
\|u\|^{2}_{H^{s}}=(\langle D\rangle^{s}u,\langle D\rangle^{s} u)_{L^{2}}
$, where $
\langle D\rangle e^{\ii j\cdot x}=\langle j\rangle  e^{\ii j\cdot x}\,$,
for any $j\in\mathbb{Z}^{d}$,
and  $(\cdot,\cdot)_{L^{2}}$ denotes the standard complex 
$L^{2}$-scalar product
 \begin{equation}\label{scalarL}
 (u,v)_{L^{2}}:=\int_{\mathbb{T}^{d}}u\cdot\bar{v}dx\,, 
 \qquad \forall\, u,v\in L^{2}(\mathbb{T}^{d};\mathbb{C})\,.
 \end{equation}

  \noindent
{\bf Notation}. 
We shall 
use the notation $A\lesssim B$ to denote 
$A\le C B$ where $C$ is a positive constant
depending on  parameters fixed once for all, for instance $d$
 and $s$.
 We will emphasize by writing $\lesssim_{q}$
 when the constant $C$ depends on some other parameter $q$.

\vspace{0.2em}
\noindent
{\bf Basic Paradifferential calculus.} We follow the notation of \cite{Feola-Iandoli-local-tori}. 
We introduce the   symbols we shall use in this paper.
We shall consider symbols 
$\mathbb{T}^{d}\times \mathbb{R}^{d}\ni (x,\x)\to a(x,\x)$
 in the spaces 
$\mathcal{N}_{s}^{m}$, $m,s\in \mathbb{R}$, $s\geq0$
defined by the 
norms
\begin{equation}\label{normaSimbo}
|a|_{\mathcal{N}_{s}^{m}}:=\sup_{|\alpha|+|\beta|\leq s}\sup_{\x\in \mathbb{R}}
\langle \x\rangle^{-m+|\beta|}\|\pa_{\x}^{\beta}\pa_{x}^{\alpha}a(x,\x)\|_{L^{\infty}}\,.
\end{equation}
The constant $m\in \mathbb{R}$ indicates the \emph{order} of the symbols, while
$s$ denotes its differentiability.
Let $0<\epsilon< 1/2$ and consider
a smooth function $\chi : \mathbb{R}\to[0,1]$
\begin{equation}\label{cutofffunct}
\chi(\x)=\left\{
\begin{aligned}
&1 \quad {\rm if}\,\,\,\, |\x|\leq 5/4 
\\
&0 \quad {\rm if}\,\,\,\, |\x|\geq 8/5 
\end{aligned}\right.\qquad {\rm and \; define}\quad
\qquad \chi_{\epsilon}(\x):=\chi(|\x|/\epsilon)\,.
\end{equation}
For a symbol $a(x,\x)$ in $\mathcal{N}_{s}^{m}$
we define its (Weyl)
quantization as 
\begin{equation}\label{quantiWeyl}
T_{a}h:=\frac{1}{(2\pi)^{d}}\sum_{j\in \mathbb{Z}^{d}}e^{\ii j\cdot x}
\sum_{k\in\mathbb{Z}^{d}}
\chi_{\epsilon}\Big(\frac{|j-k|}{\langle j+k\rangle}\Big)
\widehat{a}\big(j-k,\frac{j+k}{2}\big)\widehat{h}(k)
\end{equation}
where $\widehat{a}(\eta,\x)$ denotes the Fourier transform 
of $a(x,\x)$ in the variable 
$x\in \mathbb{T}^{d}$.
 Moreover the definition of the operator $T_a$ is independent of the choice of the  cut-off function $\chi_{\epsilon}$ up to smoothing terms, this will be proved later in Lemma \ref{azioneSimboo}.

\vspace{0.3em}
\noindent
{\bf Notation.}
 Given a symbol $a(x,\x)$ we shall also write
 \begin{equation}\label{notaPara}
 T_{a}[\cdot]:=\opbw(a(x,\x))[\cdot]\,,
 \end{equation}
 to denote the associated para-differential operator.  In the notation $B$ stands for Bony and $W$ for Weyl.

\noindent
We now collects some fundamental 
properties of para-differential operators on tori. The results are similar to the ones given in  \cite{Feola-Iandoli-local-tori}. One could also look at \cite{BMM1} for recent improvements.

\begin{lemma}\label{azioneSimboo}
The following holds.

\noindent
$(i)$ Let $m_1,m_2\in \mathbb{R}$, $s>d/2$, $s\in \mathbb{N}$ and $a\in\mathcal{N}^{m_1}_s$, 
$b\in \mathcal{N}^{m_2}_s$. One has
\begin{equation}\label{prodSimboli}
|ab|_{\mathcal{N}^{m_1+m_2}_s}+|\{a,b\}|_{\mathcal{N}_{s-1}^{m_1+m_2-1}}\lesssim
|a|_{\mathcal{N}_{s}^{m_1}}|b|_{\mathcal{N}_{s}^{m_2}}
\end{equation}
where
\begin{equation}\label{PoissonBra}
\{a,b\}:=\sum_{j=1}^{d}\Big((\pa_{\x_j}a)(\pa_{x_{j}}b)
-(\pa_{x_j}a)(\pa_{\x_{j}}b)\Big)\,.
\end{equation}

\noindent
$(ii)$ Let $s_0>d$, $s_0\in \mathbb{N}$, $m\in \mathbb{R}$ 
and $a\in\mathcal{N}_{s_0}^{m}$. 
Then, for any $s\in \mathbb{R}$, one has
\begin{equation}\label{actionSob}
\|T_{a}h\|_{H^{s-m}}\lesssim|a|_{\mathcal{N}^{m}_{s_0}}\|h\|_{H^{s}}\,,
\qquad \forall h\in H^{s}(\mathbb{T}^{d};\mathbb{C})\,.
\end{equation}

\noindent
$(iii)$ Let $s_0>d$, $s_0\in \mathbb{N}$, $m\in \mathbb{R}$, $\rho\in\mathbb{N}$,  and 
$a\in\mathcal{N}_{s_0+\rho}^{m}$.
For $0<\epsilon_2\leq \epsilon_1<1/2$ and any 
$h\in H^{s}(\mathbb{T}^{d};\mathbb{C})$, 
we define
\begin{equation}\label{natale}
R_{a}h:=\frac{1}{(2\pi)^{d}}\sum_{j\in \mathbb{Z}^{d}}e^{\ii j\cdot x}
\sum_{k\in\mathbb{Z}^{d}}
\big(\chi_{\epsilon_1}-\chi_{\epsilon_2}\big)\Big(\frac{|j-k|}{\langle j+k\rangle}\Big)
\widehat{a}(j-k,\frac{j+k}{2})\widehat{h}(k)\,,
\end{equation}
where $\chi_{\epsilon_1}, \chi_{\epsilon_2}$ are as in \eqref{cutofffunct}. 
Then one has
 \begin{equation}\label{diffQuanti}
 \|R_a h\|_{H^{s+\rho-m}}\lesssim
 \|h\|_{H^{s}}|a|_{\mathcal{N}^{m}_{\rho+s_0}}\,,
 \qquad \forall h\in H^{s}(\mathbb{T}^{d};\mathbb{C})\,.\end{equation}
 \end{lemma}
 
 \begin{proof}
$(i)$
 For any $|\alpha|+|\beta|\leq s$ we have
 \[
 \pa_{x}^{\alpha}\pa_{\x}^{\beta}\Big(a(x,\x) b(x,\x)\Big)=
 \sum_{\substack{\alpha_1+\alpha_2=\alpha \\ 
\beta_1+\beta_2=\beta }}C_{\alpha,\beta}
( \pa_{x}^{\alpha_1}\pa_{\x}^{\beta_1}a)(x,\x)
( \pa_{x}^{\alpha_2}\pa_{\x}^{\beta_2}b)(x,\x)
 \]
 for some combinatoric coefficients $C_{\alpha,\beta}>0$.
 Then, recalling \eqref{normaSimbo},
 \[
 \|( \pa_{x}^{\alpha_1}\pa_{\x}^{\beta_1}a)(x,\x)
( \pa_{x}^{\alpha_2}\pa_{\x}^{\beta_2}b)(x,\x)\|_{L^{\infty}}
\lesssim_{\alpha,\beta} 
|a|_{\mathcal{N}_s^{m_1}}
|b|_{\mathcal{N}_s^{m_2}}\langle \x\rangle^{m_1+m_2-|\beta|}\,.
 \]
This implies the \eqref{prodSimboli} for the product $ab$. 
 The \eqref{prodSimboli}
for the symbol
$\{a,b\}$ follows similarly 
using \eqref{PoissonBra}.

\smallskip
\noindent $(ii)$ 
First of all notice that, since $a\in \mathcal{N}_{s_0}^{m}$, 
$s_0\in\mathbb{N}$, then
(recall \eqref{normaSimbo})
\[
\|a(\cdot,\x)\|_{H^{s_0}}\lesssim \langle\x\rangle^{m}|a|_{\mathcal{N}_{s_0}^{m}}\,,\;\;\forall\x\in \mathbb{Z}^{d}\,, 
\]
which implies 
\begin{equation}\label{virus5}
|\hat{a}(j,\x)|\lesssim \langle\x\rangle^{m}|a|_{\mathcal{N}_{s_0}^{m}}\langle j\rangle^{-s_0}\,,
\quad \forall\, j,\x\in \mathbb{Z}^{d}\,.
\end{equation}
Moreover, since $0<\epsilon<1/2$ we note that, for $\x,\eta\in \mathbb{Z}^d$,
\begin{equation}\label{equixieta}
\chi_{\epsilon}\left(\frac{|\x-\eta|}{\langle \x+\eta\rangle}\right)\neq0\quad \Rightarrow\quad
\left\{
\begin{aligned}
&(1-\tilde\epsilon)|\x|\leq (1+\tilde\epsilon)|\eta| \\
&(1-\tilde\epsilon)|\eta|\leq (1+\tilde\epsilon)|\x| \,,
\end{aligned}
\right.
\end{equation}
where $0<\tilde{\epsilon}<4/5$, and hence
 $\langle\x+\eta\rangle\sim\langle\x\rangle $.
Therefore
\begin{equation}\label{natale2}
\begin{split}
\|T_{a}h\|^{2}_{H^{s-m}}&
\stackrel{\mathclap{\eqref{Sobnorm}}}{\lesssim}
\sum_{\x\in\mathbb{Z}^{d}}
\langle\x\rangle^{2(s-m)}\Big|\sum_{\eta\in \mathbb{Z}^{d}}
\chi_{\epsilon}\left(\frac{|\x-\eta|}{\langle \x+\eta\rangle}\right)\hat{a}(\x-\eta, \frac{\x+\eta}{2})\hat{h}(\eta)\Big|^{2}\\
&\stackrel{\mathclap{\eqref{virus5}, \eqref{equixieta}}}{\lesssim}\,\,\,\,\,\,
\sum_{\x\in \mathbb{Z}^{d}}\langle\x\rangle^{-2m}\Big(\sum_{\eta\in \mathbb{Z}^{d}}
\frac{\langle\x\rangle^{m}}{\langle\x-\eta\rangle^{s_0}}
|\hat{h}(\eta)|\langle \eta\rangle^{s}
\Big)^{2} |a|_{\mathcal{N}_{s_0}^{m}}^{2}\\
&\lesssim|a|^{2}_{\mathcal{N}^{m}_{s_0}} \sum_{\xi\in\mathbb{Z}^d}\Big(\sum_{\eta\in\mathbb{Z}^d}|\hat{h}(\eta)\langle\eta\rangle^s\frac{1}{\langle\xi-\eta\rangle^{s_0}}|\Big)^2\\
&
\lesssim|a|^{2}_{\mathcal{N}^{m}_{s_0}}\|\hat{h}(\xi)\langle\xi\rangle^s\star{\langle\xi\rangle}^{-s_0}\|_{\ell^2(\mathbb{Z}^d)}^2\leq |a|^{2}_{\mathcal{N}^{m}_{s_0}}\|\hat{h}(\xi)\langle\xi\rangle^s\|_{\ell^2(\mathbb{Z}^d)}^2\|\langle{\xi}\rangle^{-s_0}\|_{\ell^1(\mathbb{Z}^d)}^2\\
&\lesssim
\|h\|_{H^{s}}^{2}|a|^{2}_{\mathcal{N}^{m}_{s_0}}\,,
\end{split}
\end{equation}
where we denoted by $\star$ the convolution between sequences, in the penultimate passage we used the Young inequality for sequences and in the last one that $\langle\xi\rangle^{-s_0}$ is in $\ell^1(\Z^d)$ since $s_0>d$.

\noindent $(iii)$ Notice that the set of $\x,\eta$ such that
$(\chi_{\epsilon_1}-\chi_{\epsilon_2})(|\x-\eta|/\langle\x+\eta\rangle)=0$
contains the set  such that
\[
|\x-\eta|\geq \frac{8}{5}\epsilon_1 \langle \x+\eta\rangle\quad {\rm or}\quad
|\x-\eta|\leq \frac{5}{4}\epsilon_2 \langle\x+\eta\rangle\,.
\]
Therefore $(\chi_{\epsilon_1}-\chi_{\epsilon_2})(|\x-\eta|/\langle\x+\eta\rangle)\neq0$
implies
\begin{equation}\label{condizio}
\frac{5}{4}\epsilon_2 \langle\x+\eta\rangle\leq |\x-\eta|\leq \frac{8}{5}\epsilon_1 \langle\x+\eta\rangle\,.
\end{equation}
For $\x\in \mathbb{Z}^{d}$ 
we denote $\mathcal{A}(\x)$ the set of 
$\eta\in \mathbb{Z}^{d}$ such that the \eqref{condizio}
holds. Moreover (reasoning as in \eqref{virus5}), since
$a\in \mathcal{N}_{s_0+\rho}^{m}$, we have that
\begin{equation}\label{virus6}
|\hat{a}(j,\x)|\lesssim \langle\x\rangle^{m}|a|_{\mathcal{N}_{s_0+\rho}^{m}}
\langle j\rangle^{-s_0-\rho}\,,
\quad \forall\, j,\x\in \mathbb{Z}^{d}\,.
\end{equation}
To estimate the remainder in \eqref{natale}
we reason as in \eqref{natale2}.
By \eqref{condizio} and setting $\rho=s-s_0$ we have
\begin{equation}\label{stimarestoresto}
\begin{aligned}
\|R_ah\|_{H^{s+\rho-m}}^{2}
&\stackrel{\mathclap{\eqref{Sobnorm}}}{\lesssim}\sum_{\x\in\mathbb{Z}^{d}}
\langle\x\rangle^{2(s+\rho-m)}\Big| (\chi_{\epsilon_1}-\chi_{\epsilon_2})\left(\frac{|\x-\eta|}{\langle\x+\eta\rangle}\right)\hat{a}(\x-\eta, \frac{\x+\eta}{2})\hat{h}(\eta)\Big|^{2}\\
&\stackrel{\mathclap{\eqref{virus6}}}{\lesssim}
\sum_{\x\in \mathbb{Z}^{d}}
\langle \x\rangle^{-2m}
\Big(\sum_{\eta\in \mathcal{A}(\x)}
\frac{\langle\x-\eta\rangle^{\rho}\langle\x+\eta\rangle^{m}}{\langle\x-\eta\rangle^{\rho+s_0}}
|\hat{h}(\eta)|\langle \eta\rangle^{s}
\Big)^{2}|a|_{\mathcal{N}^{m}_{s_0+\rho}}^{2}\\
&\lesssim \|\hat{h}(\xi)\langle\xi\rangle^s\star\langle\xi\rangle^{-s_0}\|^2_{\ell^2(\mathbb{Z}^d)}|a|^{2}_{\mathcal{N}^{m}_{\rho+s_0}}\lesssim \|\hat{h}(\xi)\langle\xi\rangle^s\|^2_{\ell^2(\mathbb{Z}^d)} \|\langle\xi\rangle^{-s_0}\|^2_{\ell^1(\mathbb{Z}^d)}|a|^{2}_{\mathcal{N}^{m}_{\rho+s_0}}\\
&\lesssim
\|h\|_{H^{s}}^{2}|a|^{2}_{\mathcal{N}^{m}_{\rho+s_0}}\,,
\end{aligned}
\end{equation}
where we have denoted by $\star$ the convolution between sequences, in the penultimate step we used Young inequality for sequences, in the last one we used that $\langle\xi\rangle^{-s_0}$ is in $\ell^1(\Z^d)$ since $s_0>d$.

 \end{proof}

\begin{proposition}{\bf (Composition).}\label{prop:compo}
Fix $s_0>d$, $s_0\in \mathbb{N}$, and $m_1,m_2\in \mathbb{R}$.  For 
 $a\in \mathcal{N}_{s_0+2}^{m_1}$ and $b\in\mathcal{N}_{s_0+2}^{m_2}$
we have (recall \eqref{PoissonBra})
\begin{equation}\label{composit}
\begin{aligned}
T_{a}\circ T_{b}=T_{ab}+R_1(a,b)\,,\,\, \,\,\,
T_{a}\circ T_{b}=T_{ab}+\frac{1}{2\ii }T_{\{a,b\}}+R_2(a,b)\,,
\end{aligned}
\end{equation}
where $R_j(a,b)$ are remainders  satisfying, for any $s\in \mathbb{R}$,
\begin{equation}\label{composit2}
\|R_j(a,b)h\|_{H^{s-m_1-m_2+j}}\lesssim\|h\|_{H^{s}}|a|_{\mathcal{N}^{m_1}_{s_0+j}}
|b|_{\mathcal{N}^{m_2}_{s_0+j}}\,.
\end{equation}
Moreover, if $a,b\in H^{\rho+s_0}(\mathbb{T}^{d};\mathbb{C})$ are functions 
(independent of $\x\in \mathbb{R}^{n}$)
then, $\forall s\in\mathbb{R}$,
\begin{equation}\label{composit3}
\|(T_aT_b-T_{ab})h\|_{H^{s+\rho}}\lesssim\|h\|_{H^{s}}\|a\|_{H^{\rho+s_0}}
\|b\|_{H^{\rho+s_0}}\,.
\end{equation}
\end{proposition}

\begin{proof}
We start by proving the \eqref{composit3}. 
For $\x,\theta,\eta\in \mathbb{Z}^{d}$ we define
\begin{equation}\label{cutofffunctions}
r_1(\x,\theta,\eta):=\chi_{\epsilon}\left(\frac{|\x-\theta|}{\langle\x+\theta\rangle}\right)
\chi_{\epsilon}\left(\frac{|\theta-\eta|}{\langle\theta+\eta\rangle}\right)\,,
\qquad
r_2(\x,\eta):=\chi_{\epsilon}\left(\frac{|\x-\eta|}{\langle\x+\eta\rangle}\right)\,.
\end{equation}
Recalling \eqref{quantiWeyl} 
and that $a, b$ are functions we have
\begin{equation}\label{achille}
\begin{aligned}
&R_0h:= (T_aT_b-T_{ab})h\,,\\
&\widehat{(R_0h)}(\x)=(2\pi)^{-\frac{3d}{2}}\sum_{\eta,\theta\in \mathbb{Z}^{d}}
(r_1-r_2)(\x,\theta,\eta)
 \widehat{a}(\x-\theta)
\widehat{b}(\theta-\eta)\hat{h}(\eta)\,.
\end{aligned}
\end{equation}
Let us define the sets
\begin{align}
D&:=\Big\{
(\x,\theta,\eta)\in \mathbb{Z}^{3d}\; :\;
(r_1-r_2)(\x,\theta,\eta)=0\Big\}\,,\label{insiemeZero}\\
A&:=\Big\{
(\x,\theta,\eta)\in \mathbb{Z}^{3d}\; :\;
\frac{|\x-\theta|}{\langle\x+\theta\rangle}\leq \frac{5\epsilon}{4}\,,\;\;
\frac{|\x-\eta|}{\langle\x+\eta\rangle}\leq \frac{5\epsilon}{4}\,,\;\;
\frac{|\theta-\eta|}{\langle\theta+\eta\rangle}\leq \frac{5\epsilon}{4}\Big\}\,,
\label{insiemeZero2}\\
B&:=\Big\{
(\x,\theta,\eta)\in \mathbb{Z}^{3d}\; :\;
\frac{|\x-\theta|}{\langle\x+\theta\rangle}\geq \frac{8\epsilon}{5}\,,\;\;
\frac{|\x-\eta|}{\langle\x+\eta\rangle}\geq \frac{8\epsilon}{5}\,,\;\;
\frac{|\theta-\eta|}{\langle\theta+\eta\rangle}\geq \frac{8\epsilon}{5}\Big\}\,.\label{insiemeZero3}
\end{align}
We note that
\[
D\supseteq A\cup B\quad \Rightarrow\quad D^{c}\subseteq A^{c}\cap B^{c}\,.
\]
Let $(\x,\theta,\eta)\in D^{c}$ and assume in particular  that 
$(\x,\theta,\eta)\in{\rm Supp}(r_1):=\ov{\{(\x,\theta,\eta) : r_1\neq0\}}$. 
Then, reasoning as in \eqref{equixieta}, we can note that
\begin{equation}\label{navyseal2}
|\x-\eta|\leq \epsilon \langle\x+\eta\rangle\,\quad {\rm and}\quad \langle\x\rangle \sim \langle\eta\rangle.
\end{equation}
Notice also that $(\x,\theta,\eta)\in{\rm Supp}(r_2)$ implies  the \eqref{navyseal2} as well.
The rough idea of the proof is based on the fact that, if $(\x,\theta,\eta)\in D^{c}$, then 
there are \emph{at least three} equivalent frequencies among 
$\x, \x-\theta,\theta-\eta, \eta$, therefore \eqref{achille} 
restricted to  $(\x,\theta,\eta)\in D^{c}$ is a regularizing operator.
We need to estimate
\[
\|R_0h\|_{H^{s+\rho}}^{2}\lesssim\sum_{\x\in\mathbb{Z}^{d}} \Big(
\sum_{\eta,\theta}^{*} |\hat{a}(\x-\theta)||\hat{b}(\theta-\eta)||\hat{h}(\eta)
|\langle\x\rangle^{s+\rho}
\Big)^{2}=I+II+III\,,
\]
where $\sum_{\eta,\theta}^{*} $ denotes the sum over indexes satisfying 
\eqref{navyseal2}, the term $I$ denotes the sum on indexes satisfying also
$|\x-\theta|>c\epsilon |\x|$, $II$
denotes the sum on indexes satisfying also
$|\eta-\theta|>c\epsilon |\eta|$, for some $0<c\ll1$ and $III$ is defined by difference.
We estimate the term $I$. By using \eqref{navyseal2} and $|\x-\theta|>c\epsilon |\x|$ we get
\begin{equation*}
\begin{aligned}
I&\lesssim\sum_{\xi\in\Z^d}\Big(\sum_{\eta,\theta}^{*}|\hat{a}(\x-\theta)||\hat{b}(\theta-\eta)||\hat{h}(\eta)| \langle \eta\rangle^s\langle\xi-\theta\rangle^{\rho}\Big)^2\\
&\lesssim \||\hat{h}(\xi)| \langle \xi\rangle^s\star|\hat{a}(\xi)|\langle\xi\rangle^{\rho}\star|\hat{b}(\xi)|\|_{\ell^2(\Z^d)}^2\\
&\lesssim \||\hat{h}(\xi)| \langle \xi\rangle^s\|_{\ell^2(\Z^d)}^2\||\hat{a}(\xi)| \langle \xi\rangle^{\rho}\|_{\ell^1(\Z^d)}^2 \||\hat{b}(\xi)|\|_{\ell^1(\Z^d)}^2\\
&\lesssim\|h\|_{H^s}^2\|a\|_{H^{s_0+\rho}}^2\|b\|_{H^{s_0}}^2,
\end{aligned}
\end{equation*}
where in the last inequality we used Cauchy-Schwartz and $s_0>d>d/2$.

Reasoning similarly one obtains  $II\lesssim
\|h\|_{H^{s}}^{2}\|a\|^{2}_{H^{s_0}}\|b\|^{2}_{H^{s_0+\rho}}$.
The sum $III$
is restricted to indexes satisfying 
\eqref{navyseal2} and $|\x-\theta|\leq c\epsilon |\x|$, $|\eta-\theta|\leq c\epsilon |\eta|$.
For $c\ll1$ small enough these restrictions imply that $(\x,\eta,\zeta)\in A$, which 
is a contradiction since $(\x,\eta,\zeta)\in D^{c}\subseteq A^{c}$.

Let us check the \eqref{composit2}. 
We  prove that 
\begin{equation}\label{composit22}
T_{a}\circ T_{b}=T_{ab}+\frac{1}{2\ii }T_{\{a,b\}}+R_2(a,b)\,,
\qquad 
\|R_2(a,b)h\|_{H^{s-m_1-m_2+2}}\lesssim\|h\|_{H^{s}}|a|_{\mathcal{N}^{m_1}_{s_0+2}}
|b|_{\mathcal{N}^{m_2}_{s_0+2}}\,.
\end{equation}
First of all
we note that
\begin{align}
\widehat{(T_{a}T_{b}h)}(\x)&=\frac{1}{(\sqrt{2\pi})^{3d}}
\sum_{\eta,\theta\in\mathbb{Z}^{d}}
r_1(\x,\theta,\eta) \widehat{a}\big(\x-\theta,\frac{\x+\theta}{2}\big)
\widehat{b}\big(\theta-\eta,\frac{\theta+\eta}{2}\big)\hat{h}(\eta)\,,\label{def:prodotto1}\\
\widehat{(T_{ab}h)}(\x)&=\frac{1}{(\sqrt{2\pi})^{3d}}
\sum_{\eta,\theta\in\mathbb{Z}^{d}}
r_2(\x,\eta) \widehat{a}\big(\x-\theta,\frac{\x+\eta}{2}\big)
\widehat{b}\big(\theta-\eta,\frac{\x+\eta}{2}\big)\hat{h}(\eta)\,,\label{def:prodotto2}\\
\frac{1}{2\ii}\widehat{(T_{\{a,b\}}h)}(\x)&=
\frac{1}{2\ii (\sqrt{2\pi})^{3d}}
\sum_{\eta,\theta\in\mathbb{Z}^{d}}
r_2(\x,\eta) \widehat{(\pa_{\x}a)}\big(\x-\theta,\frac{\x+\eta}{2}\big)\cdot
\widehat{(\pa_{x}b)}\big(\theta-\eta,\frac{\x+\eta}{2}\big)\hat{h}(\eta)
\label{def:prodotto3}\\
&-
\frac{1}{2\ii (\sqrt{2\pi})^{3d}}
\sum_{\eta,\theta\in\mathbb{Z}^{d}}
r_2(\x,\eta) \widehat{(\pa_{x}a)}\big(\x-\theta,\frac{\x+\eta}{2}\big)\cdot
\widehat{(\pa_{\x}b)}\big(\theta-\eta,\frac{\x+\eta}{2}\big)\hat{h}(\eta)\,.\nonumber
\end{align}
In the formul\ae\, above we used the notation $\pa_{x}=(\pa_{x_1},\ldots,\pa_{x_d})$,
similarly for $\pa_{\x}$.
We remark that we can substitute the cut-off function
$r_2$ in \eqref{def:prodotto2}, \eqref{def:prodotto3} with $r_1$ up to smoothing remainders.
This follows because one can treat the 
cut-off function
$r_1(\x,\theta,\eta)-r_2(\x,\eta)$
as done in the proof of \eqref{composit3}.
Write $\x+\theta=\x+\eta+(\theta-\eta)$.
By Taylor expanding the symbols at $\x+\eta$, 
we have
\begin{align}
\widehat{a}\big(\x-\theta,\frac{\x+\theta}{2}\big)&=
\widehat{a}\big(\x-\theta,\frac{\x+\eta}{2}\big)
+\widehat{(\pa_{\x}a)}\big(\x-\theta,\frac{\x+\eta}{2}\big)
\cdot\frac{\theta-\eta}{2}\label{expsimbo1}\\
&+\frac{1}{4}\sum_{j,k=1}^{d}
\int_{0}^{1}(1-\s)
\widehat{(\pa_{\x_{j}\x_{k}}a)}\big(\x-\theta,\frac{\x+\eta}{2}+
\s\frac{\theta-\eta}{2}\big)
(\theta_{j}-\eta_{j})(\theta_{k}-\eta_{k})
d\s\,.\nonumber
\end{align}
Similarly one obtains 
\begin{align}
\widehat{b}\big(\theta-\eta,\frac{\theta+\eta}{2}\big)&=
\widehat{b}\big(\theta-\eta,\frac{\x+\eta}{2}\big)
+\widehat{(\pa_{\x}b)}\big(\theta-\eta,\frac{\x+\eta}{2}\big)
\cdot\frac{\theta-\x}{2}\label{expsimbo2}\\
&+\frac{1}{4}\sum_{j,k=1}^{d}
\int_{0}^{1}(1-\s)
\widehat{(\pa_{\x_{j}\x_{k}}b)}\big(\theta-\eta,\frac{\x+\eta}{2}+
\s\frac{\theta-\x}{2}\big)(\theta_{j}-\x_{j})(\theta_{k}-\x_{k})d\s\,.\nonumber
\end{align}
By \eqref{expsimbo1}, \eqref{expsimbo2}
we deduce that
\begin{equation}\label{virus}
\begin{aligned}
&T_{a}T_bh-T_{ab}h-\frac{1}{2\ii }T_{\{a,b\}}h=\sum_{p=1}^{6}R_p h\,,
\\&
\widehat{(R_p h)}(\x):=\frac{1}{(\sqrt{2\pi})^{3d}}\sum_{\eta,\theta\in\mathbb{Z}^{d}}
r_{1}(\x,\theta,\eta)
g_{p}(\x,\theta,\eta)\hat{h}(\eta)\,,
\end{aligned}
\end{equation}
where the symbols $g_i$ are defined as
\begin{align}
&g_1:= 
\frac{-1}{4}
\sum_{j,k=1}^{d}\int_{0}^{1}(1-\s)
\widehat{(\pa_{x_{k}x_{j}}a)}\big(\x-\theta,\frac{\x+\eta}{2}\big)
\widehat{(\pa_{\x_{k}\x_{j}}b)}\big(\theta-\eta,\frac{\x+\eta}{2}+
\s\frac{\theta-\x}{2}\big)d\s\,,
\label{simboGG1}\\
&g_2:= 
\frac{-1}{4}
\sum_{j,k=1}^{d}\int_{0}^{1}(1-\s)
\widehat{(\pa_{\x_{k}\x_{j}}a)}\big(\x-\theta,\frac{\x+\eta}{2}+\s\frac{\theta-\eta}{2}\big)
\widehat{(\pa_{x_{k}x_{j}}b)}\big(\theta-\eta,\frac{\x+\eta}{2}\big)d\s\,,
\label{simboGG2}\\
&g_3:= 
\frac{1}{4}
\sum_{j,k=1}^{d}
\widehat{(\pa_{x_{j}}\pa_{\x_{k}}a)}\big(\x-\theta,\frac{\x+\eta}{2}\big)
\widehat{(\pa_{x_{k}}\pa_{\x_{j}}b)}\big(\theta-\eta,\frac{\x+\eta}{2}\big)\,,\label{simboGG3}
\end{align}
\begin{align}
&g_4:= 
\frac{-1}{8\ii}
\sum_{j,k,p=1}^{d}\int_{0}^{1}(1-\s)
\widehat{(\pa_{x_{k}x_{j}\x_{p}}a)}\big(\x-\theta,\frac{\x+\eta}{2}\big)
\widehat{(\pa_{x_{p}\x_{k}\x_{j}}b)}\big(\theta-\eta,\frac{\x+\eta}{2}+
\s\frac{\theta-\x}{2}\big)d\s\,,
\label{simboGG4}\\
&g_5:= 
\frac{-1}{8\ii}
\sum_{j,k,p=1}^{d}\int_{0}^{1}(1-\s)
\widehat{(\pa_{\x_{k}\x_{j}x_{p}}a)}\big(\x-\theta,\frac{\x+\eta}{2}
+\s\frac{\theta-\eta}{2}\big)
\widehat{(\pa_{\x_{p}x_{k}x_{j}}b)}\big(\theta-\eta,\frac{\x+\eta}{2}\big)d\s\,,
\label{simboGG5}\\
&g_6:= 
\frac{1}{16}
\sum_{j,k,p,q=1}^{d}\int\int_{0}^{1}(1-\s_1)(1-\s_2)
\widehat{(\pa_{\x_{j}\x_{k}x_{p}x_{q}}a)}\big(\x-\theta,\frac{\x+\eta}{2}
+\s_1\frac{\theta-\eta}{2}\big)\,,\nonumber\\
&\qquad\qquad\qquad\qquad\qquad\times
\widehat{(\pa_{\x_{p}\x_{q}x_{j}x_{k}}b)}\big(\theta-\eta,\frac{\x+\eta}{2}
+\s_{2}\frac{\theta-\x}{2}\big)d\s_1d\s_2\,.\label{simboGG6}
\end{align}
We prove the estimate \eqref{composit2} (with $j=2$) on each 
term of the sum in \eqref{virus}.
First of all we note that
$r_1(\x,\theta,\eta)\neq0$ implies that
\begin{equation}\label{virus2} 
(\theta,\eta)
\in\big\{\frac{|\x-\theta|}{\langle\x+\theta\rangle}\leq \frac{8}{5}\epsilon\big\}
\bigcap\big\{\frac{|\theta-\eta|}{\langle\theta+\eta\rangle}\leq \frac{8}{5}\epsilon\big\}=:\mathcal{B}(\x)\,,
\quad \x\in \mathbb{Z}^{d}\,.
\end{equation} 
 Moreover we note that
\begin{equation}\label{virus3}
(\theta,\eta)\in \mathcal{B}(\x)\;\;\;\Rightarrow\;\;\; |\x|\lesssim|\theta|\,,\;\;
|\theta|\lesssim|\eta|\,,\;\;|\eta|\lesssim|\x|\,.
\end{equation}
We now study the term  $R_{3}h$ in \eqref{virus} depending 
on  $g_{3}(\x,\theta,\eta)$ in \eqref{simboGG3}.
We need to bound from above, for any $j,k=1,\ldots,d$, 
the $H^{s-m_1-m_2+2}$-Sobolev norm  (see \eqref{virus2}) of a term like
\begin{equation}\label{virus11}
\begin{aligned}
\hat{F}_{j,k}(\x)&:=
\sum_{(\theta,\eta)\in\mathcal{B}(\x)}
\widehat{(\pa_{x_{j}}\pa_{\x_{k}}a)}\big(\x-\theta,\frac{\x+\eta}{2}\big)
\widehat{(\pa_{x_{k}}\pa_{\x_{j}}b)}\big(\theta-\eta,\frac{\x+\eta}{2}\big)
\hat{h}(\eta)\\
&=
\sum_{\eta\in\mathbb{Z}^{d} } 
\widehat{c_{j,k}}\big(\x-\eta,\frac{\x+\eta}{2}\big)\hat{h}(\eta)\,,
\end{aligned}
\end{equation}
where we have defined
\[
\begin{aligned}
\widehat{c_{j,k}}\big(p,\zeta\big)
&:=\sum_{\ell\in \mathbb{Z}^{d}}
\widehat{(\pa_{x_{j}}\pa_{\x_{k}}a)}\big(p-\ell,\zeta\big)
\widehat{(\pa_{x_{k}}\pa_{\x_{j}}b)}\big(\ell,\zeta\big)
\mathtt{1}_{\mathcal{C}(p,\zeta)}\,,\qquad p,\zeta\in \mathbb{Z}^{d}\,,\\
\mathcal{C}(p,\zeta)&:=
\big\{\ell\in \mathbb{Z}^{d}\,:\,\frac{|p-\ell|}{\langle2\zeta+\ell\rangle}
\leq \frac{8}{5}\epsilon\big\}\bigcap
\big\{\ell\in \mathbb{Z}^{d}\,:\,
\frac{|\ell|}{\langle\ell-p+2\zeta\rangle}\leq \frac{8}{5}\epsilon\big\}
\end{aligned}
\]
and $\mathtt{1}_{\mathcal{C}(p,\zeta)}$ is the 
characteristic function of the set $\mathcal{C}(p,\zeta)$.
Reasoning as in \eqref{virus3}, we can deduce that for $\ell\in \mathcal{C}(p,\zeta)$
one has
\begin{equation}\label{virus12}
|2\zeta|\lesssim \frac{1}{2}|2\zeta+p|\,.
\end{equation}
Indeed $\ell\in \mathcal{C}(p,\zeta) $ implies $(\theta,\eta)\in \mathcal{B}(\x)$
by setting 
\begin{equation}\label{virus13}
2\x=2\zeta+p\,,\quad 2\theta=2\ell+2\zeta-p\,,\quad 2\eta=2\zeta-p\,.
\end{equation}
Hence the \eqref{virus12} follows by \eqref{virus3} by observing  that
$2\zeta=\x+\eta$.
Using that   $a\in \mathcal{N}_{s_0+2}^{m_1}$, 
$b\in\mathcal{N}_{s_0+2}^{m_2}$
and reasoning as in \eqref{virus5}
we deduce
\begin{equation}\label{virus10}
|\hat{c_{j,k}}(p,\zeta)|\lesssim \langle\zeta\rangle^{m_1+m_2-2}\langle p\rangle^{-s_0}
|a|_{\mathcal{N}^{m_1}_{s_0+2}}|b|_{\mathcal{N}^{m_2}_{s_0+2}}\,.
\end{equation}
By \eqref{virus11}, \eqref{virus3},  \eqref{Sobnorm}, we get
\[
\begin{aligned}
\|F_{j,k}\|_{H^{s-m_1-m_2+2}}^{2}&\lesssim
\sum_{\x\in\mathbb{Z}^{d}}
\langle \x\rangle^{-2m_1-2m_2+2}\Big(
\sum_{\eta\in \mathbb{Z}^{d}}
|\hat{c_{j,k}}\big(\x-\eta,\frac{\x+\eta}{2}\big)|
|\hat{h}(\eta)|\langle\eta\rangle^{s}\Big)^{2}\\
&\stackrel{\mathclap{\eqref{virus10}, \eqref{virus12}, \eqref{virus13}}}{\lesssim}
|a|^{2}_{\mathcal{N}^{m_1}_{s_0+2}}|b|^{2}_{\mathcal{N}^{m_2}_{s_0+2}}
\sum_{\xi\in\mathbb{Z}^{d}}\Big(\sum_{\eta\in \mathbb{Z}^{d}}
|\hat{h}(\eta)|\langle\eta\rangle^{s}
\frac{1}{\langle \x-\eta\rangle^{s_0}}\Big)^2\\
&\lesssim|a|^{2}_{\mathcal{N}^{m_1}_{s_0+2}}
|b|^{2}_{\mathcal{N}^{m_2}_{s_0+2}}\||\hat{h}(\xi)|\langle\xi\rangle^{s}\star\langle\xi\rangle^{-s_0}\|_{\ell^{2}(\mathbb{Z}^d)}\\
&\lesssim \|h\|_{H^{s}}^{2}
|a|^{2}_{\mathcal{N}^{m_1}_{s_0+2}}
|b|^{2}_{\mathcal{N}^{m_2}_{s_0+2}}\,,
\end{aligned}
\]
where in the last step we used Young inequality for sequences, the Cauchy-Schwartz inequality and that $\langle\xi\rangle^{-s_0}$ is in $\ell^1(\mathbb{Z}^d)$ if $s_0>d$.
Since the estimate above holds for any $j,k=1,\ldots,d$, we may absorb the remainder $R_{3}h$ in \eqref{virus} in $R_2(a,b)h$ satisfying
\eqref{composit22}. 
One can deal with the other terms $g_1,g_2,g_4,g_5,g_6$ similarly.

\end{proof}

\begin{lemma}\label{lem:paraproduct}
Fix $s_0>d/2$ and 
let $f,g,h\in H^{s}(\mathbb{T};\mathbb{C})$ for $s\geq s_0$. Then
\begin{equation}\label{eq:paraproduct}
fgh=T_{fg}h+T_{gh}f+T_{fh}g+\mathcal{R}(f,g,h)\,,
\end{equation}
where
\begin{equation}\label{eq:paraproduct2}
\begin{aligned}
\widehat{\mathcal{R}(f,g,h)}(\x)&=\frac{1}{(2\pi)^{d}}
\sum_{\eta,\zeta\in \mathbb{Z}^{d}}
a(\x,\eta,\zeta)\hat{f}(\x-\eta-\zeta)\hat{g}(\eta)\hat{h}(\zeta)\,,\\
 |a(\xi,\eta,\zeta)&|\lesssim_{\rho}\frac{\max_2(|\xi-\eta-\zeta|,|\eta|,|\zeta|)^{\rho}}{\max(|\xi-\eta-\zeta|,|\eta|,|\zeta|)^{\rho}}\,, \,\,\forall \rho\geq 0.
\end{aligned}
\end{equation}
\end{lemma}
\begin{remark}
An estimate of the form \eqref{eq:paraproduct2} implies that the function $(f,g,h)\mapsto \mathcal{R}(f,g,h)$ defines a continuous trilinear form on $H^s\times H^s\times H^s$ with values in $H^{s+\rho}$ as soon as $s>\rho+d/2$. This will be proved in Lemma \ref{lem:trilineare}.
\end{remark}
\begin{proof}
We start by proving the following claim:  the term 
\[
T_{fg}h-\sum_{\x\in \mathbb{Z}^{d}}e^{\ii\x\cdot x}
\sum_{\eta,\zeta\in \mathbb{Z}^{d}}\chi_{\epsilon}\Big(\tfrac{|\x-\eta-\zeta|+|\eta|}{\langle \zeta\rangle}\Big)
\hat{f}(\x-\eta-\zeta)\hat{g}(\eta)\hat{h}(\zeta)
\]
is a remainder of the form \eqref{eq:paraproduct2}. By \eqref{quantiWeyl}
this is actually true with coefficients $a(\x,\eta,\zeta)$ of the form
\[
a(\x,\eta,\zeta):=\chi_{\epsilon}\Big(\tfrac{|\x-\zeta|}{\langle \x+\zeta\rangle}\Big)-
\chi_{\epsilon}\Big(\tfrac{|\x-\eta-\zeta|+|\eta|}{\langle \zeta\rangle}\Big)\,.
\]
In order to prove this, we consider the following partition of the unity:
\begin{equation}\label{partUnity}
\Theta_{\epsilon}(\xi,\eta,\zeta):=1-\chi_{\epsilon}\Big(\tfrac{|\x-\eta-\zeta|+|\zeta|}{\langle \eta\rangle}\Big)-\chi_{\epsilon}\Big(\tfrac{|\eta|+|\zeta|}{\langle \xi-\eta-\zeta\rangle}\Big)-\chi_{\epsilon}\Big(\tfrac{|\x-\eta-\zeta|+|\eta|}{\langle \zeta\rangle}\Big)\,.
\end{equation}
Then we can write
\begin{align*}
a(\x,\eta,\zeta)&=\Big(\chi_{\epsilon}\Big(\tfrac{|\x-\zeta|}{\langle \x+\zeta\rangle}\Big)-1\Big)
\chi_{\epsilon}\Big(\tfrac{|\x-\eta-\zeta|+|\eta|}{\langle \zeta\rangle}\Big)+
\chi_{\epsilon}\Big(\tfrac{|\x-\zeta|}{\langle \x+\zeta\rangle}\Big)\chi_{\epsilon}\Big(\tfrac{|\x-\eta-\zeta|+|\zeta|}{\langle \eta\rangle}\Big)\\
&+\chi_{\epsilon}\Big(\tfrac{|\x-\zeta|}{\langle \x+\zeta\rangle}\Big)
\chi_{\epsilon}\Big(\tfrac{|\eta|+|\zeta|}{\langle \xi-\eta-\zeta\rangle}\Big)+
\chi_{\epsilon}\Big(\tfrac{|\x-\zeta|}{\langle \x+\zeta\rangle}\Big)
\Theta_{\epsilon}(\xi,\eta,\zeta)\,.
\end{align*}
Using \eqref{cutofffunct} one can prove that 
each summand in the r.h.s. of the equation above is non-zero only if
$\max_2(|\xi-\eta-\zeta|,|\eta,||\zeta|)\sim \max_1(|\xi-\eta-\zeta|,|\eta,||\zeta|)$.
This implies that each summand defines a smoothing remainder 
as in \eqref{eq:paraproduct2}.
A similar property holds also for $T_{gh}f$ and $T_{fh}g$. 
At this point we write
\[
\begin{aligned}
fgh=
\sum_{\x\in \mathbb{Z}^{d}}e^{\ii\x\cdot x}
&\sum_{\eta,\zeta\in \mathbb{Z}^{d}}
\Big[\Theta_{\epsilon}(\xi,\eta,\zeta)+\chi_{\epsilon}\Big(\tfrac{|\x-\eta-\zeta|+|\zeta|}{\langle \eta\rangle}\Big)\\
&\qquad+\chi_{\epsilon}\Big(\tfrac{|\eta|+|\zeta|}{\langle \xi-\eta-\zeta\rangle}\Big)+\chi_{\epsilon}\Big(\tfrac{|\x-\eta-\zeta|+|\eta|}{\langle \zeta\rangle}\Big)\Big]
\hat{f}(\x-\eta-\zeta)\hat{g}(\eta)\hat{h}(\zeta)\,.
\end{aligned}
\]
One concludes  by using the claim at the beginning of the proof.
\end{proof}

\vspace{0.4em}
\noindent
{\bf Matrices of symbols and operators.}
Let us consider the subspace 
$\mathcal{U}$ defined as
 \begin{equation}\label{spazioUU}
 \mathcal{U}:=\big\{(u^{+},u^{-})\in 
 L^{2}(\mathbb{T}^{d};\mathbb{C})\times 
 L^{2}(\mathbb{T}^{d};\mathbb{C})\; :\; 
 u^{+}=\ov{u^{-}}
 \big\}\,.
 \end{equation}
Along the paper we shall deal with matrices of linear operators acting on 
$H^{s}(\mathbb{T}^{d};\mathbb{C}^{2})$ preserving the subspace $\mathcal{U}$.
Consider two operators 
$R_{1},R_{2}$ acting on  $C^{\infty}(\mathbb{T}^d;\mathbb{C})$.
We define the operator $\mathfrak{F}$ acting on $C^{\infty}(\mathbb{T}^d;\mathbb{C}^2)$ as
\begin{equation}\label{barrato4}
\mathfrak{F}:=\sm{R_1} {R_2} 
{\ov{R_2}}  {\ov{R_1}},
\end{equation}
where the  linear  operators $\ov{R_i}[\cdot]$, $i=1,2$ are defined by the relation 
$\ov{R_i}[v] := \ov{R_i[\ov{v}]} \,$.
We say that an operator of the form \eqref{barrato4} is \emph{real-to-real}.
It is easy to note that real-to-real operators preserves 
$\mathcal{U}$ in \eqref{spazioUU}.
Consider now a symbol $a(x,\x)$ of order $m$ and set $A:=T_{a}$.
Using \eqref{quantiWeyl} one can check that
\begin{align}
 \bar{A}[h]&=\ov{A[\bar{h}]}\,,  \qquad\qquad \Rightarrow\quad 
\bar{A}=T_{\tilde{a}}\,,\qquad \tilde{a}(x,\x)=\ov{a(x,-\x)}\,;\label{simboBarrato}
\\
{\bf (Adjoint)} \;\;  (Ah , v)_{L^{2}}&=(h,A^{*}v)_{L^{2}}\,, \;\;\quad \Rightarrow \quad 
A^{*}=T_{\bar{a}}\,.\label{simboAggiunto}
\end{align}
By \eqref{simboAggiunto} we deduce that the operator $A$
is self-adjoint with respect to the scalar product \eqref{scalarL}
if and only if the symbol $a(x,\x)$ is real valued.
We need the following definition.
Consider two symbols
$a, b\in \mathcal{N}_s^{m}$ and the matrix
\begin{equation*}
A := A(x,\x):=
\left(\begin{matrix} a(x,\x) & b(x,\x) \vspace{0.2em}\\
\ov{b(x,-\x)} & \ov{a(x,-\x)}
\end{matrix}
\right)\,.
\end{equation*}
Define the operator (recall  \eqref{notaPara})
\begin{equation}\label{matsimboli}
M:=\opbw(A(x,\x)):=
\left(\begin{matrix} \opbw(a(x,\x)) & \opbw(b(x,\x)) \vspace{0.2em}\\
\opbw(\ov{b(x,-\x)}) & \opbw(\ov{a(x,-\x)})
\end{matrix}
\right)\,.
\end{equation}
\noindent
The matrix of paradifferential operators defined above have the following properties:

$\bullet$ \emph{Reality}: by \eqref{simboBarrato} we have that
the operator $M$ in \eqref{matsimboli} 
has the form \eqref{barrato4}, hence 
it is real-to-real;

$\bullet$ \emph{Self-adjointeness}:
using \eqref{simboAggiunto} 
the operator $M$ in \eqref{matsimboli} is self-adjoint with respect to the 
scalar product on \eqref{spazioUU}
\begin{equation}\label{comsca}
(U,V)_{L^{2}}:=
\int_{\mathbb{T}^{d}}U\cdot \ov{V}dx\,, 
\qquad U=\vect{u}{\bar{u}} \,, \;\; V=\vect{v}{\bar{v}}\,.
\end{equation}
if and only if 
\begin{equation}\label{simboAggiunto2}
a(x,\x)=\ov{a(x,\x)}\,,
\qquad
b(x,-\x)=b(x,\x)\,.
\end{equation}

\vspace{0.4em}
\noindent
{\bf Non-homogeneous symbols.}
In this paper we deal with symbols satisfying \eqref{normaSimbo}
which depends nonlinearly on an extra function  $u(t,x)$ (which in the application will be a solution either of \eqref{NLS} or a solution of \eqref{KG}).
We are interested in providing estimates of the semi-norms \eqref{normaSimbo}
in terms of the Sobolev norms of the function $u$.

We recall classical tame estimates for composition of functions, we refer to \cite{baldi-BO}
(see also \cite{Tay-Para}).
A function $f: \mathbb{T}^{d}\times B_R\to \mathbb{C}$, where $B_R:=\{y\in \mathbb{R}^{m} : |y|<R\}$, $R>0$, 
induces the composition operator (Nemytskii)
\begin{equation}\label{compoNEM}
\tilde{f}(u):=f(x,u(x),Du(x),\ldots, D^{p}u(x))\,,
\end{equation}
where $D^{k}u(x)$ denote the derivatives $\pa_{x}^{\alpha}$ of order $|\alpha|=k$
(the number $m$ of $y$-variables  depends on $p$, $d$).

\begin{lemma}\label{CompfuncLemma}
Fix $\gamma>0$ 
and assume that 
$f\in C^{\infty}(\mathbb{T}^{d}\times B_R; \mathbb{R})$.
Then, for any $u\in H^{\gamma+p}$ with $\|u\|_{W^{p,\infty}}<R$,  one has
\begin{align}
& \|\tilde{f}(u)\|_{H^{\gamma}}\leq C \|f\|_{C^{\gamma}}(1+\|u\|_{H^{\gamma+p}})\,,\\
& \|\tilde{f}(u+h)-\sum_{n=0}^N\frac{1}{n!}\partial_u^n\tilde{f}[h,\ldots,h]\|_{H^{\gamma}}\leq C\|h\|^N_{W^{p,\infty}}(\|h\|_{H^{\gamma}}+\|h\|_{W^{p,\infty}}\|u\|_{H^{\gamma+p}}).
\end{align}
for any $h\in H^{\gamma+p}$ with $\|h\|_{W^{p,\infty}}<R/2$ and where $C>0$ is a constant depending on $\gamma$ and the norm $\|u\|_{W^{p,\infty}}$.
\end{lemma}

Consider a function $F(y_0,y_1,\ldots,y_{d})$  in 
$C^{\infty}(\mathbb{C}^{d+1};\mathbb{R})$ in the \emph{real} sense, 
i.e. $F$ is $C^{\infty}$ as function of ${\rm Re}(y_i)$, ${\rm Im}(y_i)$.
Assume that $F$ has a zero of order at least $p+2\in \mathbb{N}$ 
at the origin.
Consider a symbol $f(\x)$, independent of $x\in \mathbb{T}^{d}$, such that
$|f|_{\mathcal{N}_s^{m}}\leq C<+\infty$, for some 
constant $C$.
Let us define the symbol
\begin{equation}\label{drhouse}
a(x,\x):= \big( \pa_{z_j^{\alpha} z_k^{\beta}}F\big)(u,\nabla u)
f(\x)\,, \quad z_j^{\alpha}:=\pa_{x_j}^{\alpha}u^{\s}, z_k^{\beta}:=\pa_{x_k}^{\beta}u^{\s'}  
\end{equation}
for some $j,k=1,\ldots,d$, $\alpha,\beta\in\{0,1\}$ and $\s,\s'\in \{\pm\}$
where we used the notation $u^{+}=u$ and $u^{-}=\bar{u}$.

\begin{lemma}\label{lem:nonomosimbo}
Fix $s_0>d/2$.
For $u\in B_R( H^{s+s_0+1}(\mathbb{T}^{d};\mathbb{C}))$ with $0<R<1$, 
we have
\begin{equation*}
|a|_{\mathcal{N}_s^{m}}\lesssim  \|u\|^{p}_{H^{s+s_0+1}}\,,
\end{equation*}
where $a$ is the symbol in \eqref{drhouse}.
Moreover,
 the map 
$h\to (\pa_{u}a)(u;x,\x) h$ is a 
$\mathbb{C}$-linear map from $ H^{s+s_0+1}$ to $\mathbb{C}$ and satisfies 
\begin{equation*}
|(\pa_{u}a) h|_{\mathcal{N}_{s}^{m}}\lesssim 
\|h\|_{H^{s+s_0+1}}\|u\|^{p-1}_{H^{s+s_0+1}}\,.
\end{equation*}
The same holds for $\pa_{\bar{u}}a$. Moreover if the symbol $a$ does not depend on $\nabla u$, then the same results are true with $s_0+1\rightsquigarrow s_0$.
\end{lemma}
\begin{proof}
It follows from Lemma \ref{CompfuncLemma}.
\end{proof}

\vspace{0.2em}
\noindent
{\bf Trilinear operators.}
Along the paper we shall deal with
trilinear operators 
on the Sobolev spaces.
We shall adopt a combination of notation introduced in \cite{BD} and 
\cite{IPtori}.
In particular we are interested in studying properties of operators of the form
\begin{equation}\label{trilinearop}
\begin{aligned}
&Q=Q[u_1,u_2,u_3] : (C^{\infty}(\mathbb{T}^{d};\mathbb{C}))^3\to 
C^{\infty}(\mathbb{T}^{d};\mathbb{C})\,,\\
&\widehat{Q}(\x)=\frac{1}{(2\pi)^{d}}\sum_{\eta,\zeta\in \mathbb{Z}^{d}}q(\x,\eta,\zeta)
\hat{u}_1(\x-\eta-\zeta)\hat{u}_2(\eta)\hat{u}_3(\zeta)\,,
\qquad \forall\, \x\in \mathbb{Z}^{d}\,,
\end{aligned}
\end{equation}
where the coefficients $q(\x,\eta,\zeta)\in \mathbb{C}$ for any 
$\x,\eta,\zeta\in \mathbb{Z}^{d}$.
We now prove that, under certain conditions on the coefficients, 
the operators of the form \eqref{trilinearop} extend as 
continuous maps on the Sobolev spaces.

\begin{lemma}\label{lem:trilineare}
Let $\mu\geq 0$ and $m\in \mathbb{R}$. Assume that for any
$\x,\eta,\zeta\in \mathbb{Z}^{d}$ one has
\begin{equation}\label{ipoMMM}
|q(\x,\eta,\zeta)|\lesssim 
\frac{\max_{2}\{\langle \x-\eta-\zeta\rangle, \langle \eta\rangle,\langle \zeta\rangle\}^{\mu}}
{\max_{1}\{\langle \x-\eta-\zeta\rangle, \langle \eta\rangle,\langle \zeta\rangle\}^{m}}\,.
\end{equation}
Then, for $s\geq s_0>d/2+\mu$, the map $Q$ in \eqref{trilinearop}
with coefficients satisfying \eqref{ipoMMM} extends as a 
continuous map form $(H^{s}(\mathbb{T}^{d};\mathbb{C}))^{3}$ to 
$H^{s+m}(\mathbb{T}^{d};\mathbb{C})$. Moreover one has
\begin{equation}\label{stimaMMM}
\|Q(u_1,u_2,u_3)\|_{H^{s+m}}\lesssim
\sum_{i=1}^{3} \|u_i\|_{H^{s}}\prod_{i\neq k}\|u_k\|_{H^{s_0}}\,.
\end{equation}
\end{lemma}

\begin{proof}
By \eqref{Sobnorm} we have
\begin{equation}\label{stimaMMM2}
\begin{aligned}
\|Q(u_1,u_2&,u_3)\|^{2}_{H^{s+m}}\leq
\sum_{\x\in \mathbb{Z}^{d}}\langle\x\rangle^{2(s+m)}
\left(
\sum_{\eta,\zeta\in \mathbb{Z}^{d}}|q(\x,\eta,\zeta)|
|\hat{u}_1(\x-\eta-\zeta)||\hat{u}_2(\eta)||\hat{u}_3(\zeta)|
\right)^{2}\\
&\stackrel{\eqref{ipoMMM}}{\lesssim}
\sum_{\x\in \mathbb{Z}^{d}}
\left(
\sum_{\eta,\zeta\in \mathbb{Z}^{d}}\langle\x\rangle^{s}
\max_{2}\{\langle \x-\eta-\zeta\rangle, \langle \eta\rangle,\langle \zeta\rangle\}^{\mu}
|\hat{u}_1(\x-\eta-\zeta)||\hat{u}_2(\eta)||\hat{u}_3(\zeta)|
\right)^{2}\\
&:=I + II+ III\,,
\end{aligned}
\end{equation}
where $I,II,III$ are the terms in \eqref{stimaMMM2} which are supported 
respectively 
on indexes such that $\max_{1}\{\langle \x-\eta-\zeta\rangle, \langle \eta\rangle,\langle \zeta\rangle\}=\langle \x-\eta-\zeta\rangle$,
$\max_{1}\{\langle \x-\eta-\zeta\rangle, \langle \eta\rangle,\langle \zeta\rangle\}
=\langle \eta\rangle$ and 
$\max_{1}\{\langle \x-\eta-\zeta\rangle, \langle \eta\rangle,\langle \zeta\rangle\}
=\langle \zeta\rangle$.
Consider for instance the term $III$.
By using the Young inequality for sequences we deduce
\[
III\lesssim \|(
\langle p \rangle^{\mu}\hat{u}_1(p)
)*(\langle\eta\rangle^{\mu}\hat{u}_2(\eta))*(\langle\zeta\rangle^{s}\hat{u}_3(\zeta))\|_{\ell^{2}}
\lesssim \|u_1\|_{H^{s_0}}\|u_2\|_{H^{s_0}}\|u_3\|_{H^{s}}\,,
\]
which is the \eqref{stimaMMM}. The bounds of $I$ and $II$ are similar.
\end{proof}

In the following lemma we shall prove that
a class of ``para-differential'' trilinear operators, having some decay on the coefficients,
satisfies the hypothesis of the previous lemma.

\begin{lemma}\label{lem:paratri}
Let $\mu\geq 0$ and $m\in\mathbb{R}$, $m\geq 0$. Consider a trilinear map $Q$
as in \eqref{trilinearop} with coefficients
satisfying 
\begin{equation}\label{gotham10}
q(\x,\eta,\zeta)=f(\x,\eta,\zeta)\chi_{\epsilon}
\Big(\frac{|\x-\zeta|}{\langle\x+\zeta\rangle}\Big)\,,
\qquad |f(\x,\eta,\zeta)|\lesssim\frac{|\x-\zeta|^{\mu}}{\langle\zeta\rangle^{m}}
\end{equation}
for any $\x,\eta,\zeta\in \mathbb{Z}^{d}$ and $0<\epsilon\ll1$. Then
the coefficients $q(\x,\eta,\zeta)$ satisfy the \eqref{ipoMMM}
with $\mu\rightsquigarrow \mu+m$.
\end{lemma}

\begin{proof}
First of all we write $q(\x,\eta,\zeta)=q_1(\x,\eta,\zeta)+q_2(\x,\eta,\zeta)$ with
\begin{align}
q_1(\x,\eta,\zeta)&=f(\x,\eta,\zeta)\chi_{\epsilon}\Big(\frac{|\xi-\zeta|}{\langle\xi+\zeta\rangle}\Big)\chi_{\epsilon}\Big(\frac{|\x-\eta-\zeta|+|\eta|}{\langle \zeta\rangle}\Big)\,,
\label{gotham11}\\
q_2(\x,\eta,\zeta)&=f(\x,\eta,\zeta)
\chi_{\epsilon}\Big(\frac{|\x-\zeta|}{\langle \xi+\zeta\rangle}\Big)
\Big[\chi_{\epsilon}\Big(\frac{|\x-\eta-\zeta|+|\zeta|}{\langle \eta\rangle}\Big)+\chi_{\epsilon}\Big(\frac{|\eta|+|\zeta|}{\langle \xi-\eta-\zeta\rangle}\Big)+\Theta_{\epsilon}(\xi,\eta,\zeta)
\Big]\,,\label{gotham12}
\end{align}
where  $\Theta_{\epsilon}(\xi,\eta,\zeta)$ is defined in \eqref{partUnity}.
%
Recalling \eqref{cutofffunct} one can check that if
$
q_1(\x,\eta,\zeta)\neq0$ then $
|\x-\eta-\zeta|+|\eta|\ll |\zeta|\sim|\x|\,.
$
Together with the bound on $f(\x,\eta,\zeta)$ in \eqref{gotham10}
we deduce that the coefficients in \eqref{gotham11} 
satisfy the \eqref{ipoMMM}.
The coefficients in 
\eqref{gotham12}
satisfy the \eqref{ipoMMM} because of the support of the cut off function in \eqref{cutofffunct}.
\end{proof}

 \noindent
 {\bf Hamiltonian formalism in complex variables.} Given a Hamiltonian function
 $H : H^{1}(\mathbb{T}^{d};\mathbb{C}^{2})\to \mathbb{R} $, its Hamiltonian 
 vector field has the form
 \begin{equation}\label{VecfieldHam}
 X_{H}(U):=-\ii J\nabla H(U)=-\ii \left(
 \begin{matrix}
 \nabla_{\bar{u}}H(U)\\
- \nabla_{u}H(U)
 \end{matrix}
 \right)\,,\quad J=\sm{0}{1}{-1}{0}\,,\quad U=\vect{u}{\bar{u}}\,.
 \end{equation}
 Indeed one has
 \begin{equation}\label{def:vecHam}
 dH(U)[V]=-\Omega(X_{H}(U), V)\,,\qquad \forall 
 U=\vect{u}{\bar{u}}\,,\; V=\vect{v}{\bar{v}}\,,
 \end{equation}
 where $\Omega$ is the non-degenerate symplectic form
 \begin{equation}\label{symform}
 \Omega(U,V)=-\int_{\mathbb{T}^{d}}U\cdot \ii JV dx
 =-\int_{\mathbb{T}^{d}}
 \ii (u\bar{v}-\bar{u}v)dx\,.
 \end{equation}
 The Poisson brackets between two Hamiltonians $H,G$ are defined as
\begin{equation}\label{Poisson}
\{G,H\}:=\Omega(X_{G},X_{H})
\stackrel{\eqref{symform}}{=}
-\int_{\mathbb{T}^d}  \ii J\nabla G\cdot\nabla H dx=-
\ii \int_{\mathbb{T}^d} \nabla_{u}H\nabla_{\bar{u}}G-  \nabla_{\bar{u}}H\nabla_{{u}}G dx\,.
\end{equation}
The nonlinear commutator between two Hamiltonian vector fields is given by
\begin{equation}\label{nonlinCommu}
[X_{G},X_{H}](U)=
dX_{G}(U)\big[X_{H}(U)\big]-dX_{H}(U)\big[X_{G}(U)\big]=
-X_{\{G,H\}}(U)\,.
\end{equation}
 \noindent
 {\bf Hamiltonian formalism in real variables.} Given a Hamiltonian function
 $H_{\mathbb{R}} : H^{1}(\mathbb{T}^{d};\mathbb{R}^{2})\to \mathbb{R} $, its hamiltonian 
 vector field has the form
 \begin{equation}\label{realvecfield}
 X_{H_{\mathbb{R}}}(\psi,\phi):= J\nabla H_{\mathbb{R}}(\psi,\phi)=\left(
 \begin{matrix}
 \nabla_{\phi}H_{\mathbb{R}}(\psi,\phi)\\
- \nabla_{\psi}H_{\mathbb{R}}(\psi,\phi)
 \end{matrix}
 \right)\,,
 \end{equation}
 where $J$ is in \eqref{VecfieldHam}. 
 Indeed one has
 \begin{equation}\label{def:vecHamKGreal}
 dH_{\mathbb{R}}(\psi,\phi)[h]=-\tilde{\Omega}(X_{H_{\mathbb{R}}}(\psi,\phi), h)\,,
 \qquad \forall 
 \vect{\psi}{\phi}\,,\; h=\vect{\hat{\psi}}{\hat{\phi}}\,,
 \end{equation}
 where $\tilde{\Omega}$ is the non-degenerate symplectic form
 \begin{equation}\label{symReal}
\widetilde{\Omega}(\vect{\psi_1}{\phi_1},\vect{\psi_2}{\phi_2})
:=\int_{\mathbb{T}^d}\vect{\psi_1}{\phi_1}
\cdot J^{-1}\vect{\psi_2}{\phi_2}dx=
\int_{\mathbb{T}^d}-(\psi_1\phi_2-\phi_1\psi_2)dx\,,
\end{equation}
We introduce the complex symplectic variables 
\begin{equation}\label{CVWW}
\!\!\!\!\!\!\left(\begin{matrix}
u \\
\overline{u}
\end{matrix} 
\right) = 
\mathcal{C}
\left(\begin{matrix}
\psi \\ \phi\end{matrix}\right)
 : =\frac{1}{\sqrt{2}} 
\left(\begin{matrix}
\Lambda_{\KG}^{\frac{1}{2}}\psi+ \ii \Lambda_{\KG}^{-\frac{1}{2}}\phi   \\
\Lambda_{\KG}^{\frac{1}{2}}\psi - \ii \Lambda_{\KG}^{-\frac{1}{2}}\phi  
\end{matrix} 
\right) \, , 
\qquad 
\left(\begin{matrix}\psi \\ \phi\end{matrix}\right) 
= \mathcal{C}^{-1} 
\left(\begin{matrix}
u \\ \bar{u} \end{matrix}\right) = 
\frac{1}{\sqrt{2}} 
\left(\begin{matrix}
\Lambda_{\KG}^{-\frac12}( u + \bar{u} ) \\
 - \ii  \Lambda_{\KG}^{\frac12}(  u - \bar{u} )
\end{matrix}\right)\,,
\end{equation}
where $\Lambda_{\KG}$ is in \eqref{def:Lambda}.
The symplectic form in \eqref{symReal} transforms, for 
$U=\vect{u}{\bar{u}}$, $V=\vect{v}{\bar{v}}$, into
$\Omega(U,V)$ where $\Omega$ is in \eqref{symform}.
In these coordinates the vector field $X_{H_{\mathbb{R}}}$ in \eqref{realvecfield}  
assumes the form $X_{H}$ as in \eqref{VecfieldHam}
with $H:=H_{\mathbb{R}}\circ\mathcal{C}^{-1}$.

We now study some algebraic properties enjoyed by the Hamiltonian functions previously defined.
Let us consider a homogeneous Hamiltonian 
$H : H^{1}(\mathbb{T}^{d};\mathbb{C}^{2})\to \mathbb{R}$
of degree four of the form
 \begin{equation}\label{frate6}
 H(U)=(2\pi)^{-d}\sum_{\x,\eta,\zeta\in\mathbb{Z}^{d}}\mathtt{h}_{4}(\x,\eta,\zeta)
 \hat{u}(\x-\eta-\zeta)\hat{\bar{u}}(\eta)\hat{u}(\zeta)\hat{\bar{u}}(-\x)\,,
 \qquad U=\vect{u}{\bar{u}}\,,
 \end{equation}
 for some coefficients $\mathtt{h}_{4}(\x,\eta,\zeta)\in \mathbb{C}$ such that
  \begin{equation}\label{frate3}
\begin{aligned}
&\mathtt{h}_{4}(\x,\eta,\zeta)=\mathtt{h}_{4}(-\eta,-\x,\zeta)
=\mathtt{h}_{4}(\x,\eta,\x-\eta-\zeta)\,,\\
&\mathtt{h}_{4}(\x,\eta,\zeta)=\ov{\mathtt{h}_{4}(\zeta,\eta+\zeta-\x,\x)}\,,\qquad \forall\, 
\x,\eta,\zeta\in \mathbb{Z}^{d}\,.
\end{aligned}
\end{equation}
By \eqref{frate3} one can check that 
the Hamiltonian $H$ is real valued and symmetric in its entries. 
Recalling \eqref{VecfieldHam} we have that its Hamiltonian vector field
can be written as
\begin{equation}\label{frate33}
X_{H}(U)= \left(
\begin{matrix}
-\ii\nabla_{\bar{u}}H(U) \\
\ii\nabla_{u}H(U)
\end{matrix}
\right)=
\left(\begin{matrix}
X_{H}^{+}(U) \vspace{0.2em}\\
\ov{X_{H}^{+}(U)}
\end{matrix}
\right)
\end{equation}
\begin{equation}\label{frate331}
\widehat{X_{H}^{+}(U)}(\x)=(2\pi)^{-d}\sum_{\eta,\zeta\in \mathbb{Z}^{d}}f(\x,\eta,\zeta)
 \hat{u}(\x-\eta-\zeta)\hat{\bar{u}}(\eta)\hat{u}(\zeta)\,,
\end{equation}
where the coefficients $f(\x,\eta,\zeta)$
have the form
\begin{equation}\label{frate44}
f(\x,\eta,\zeta)=-2\ii \mathtt{h}_{4}(\x,\eta,\zeta)\,,\qquad \x,\eta,\zeta\in \mathbb{Z}^{d}\,.
\end{equation}
We need the following definition.
 \begin{definition}{\bf (Resonant set).}\label{def:resonantSet}
We define the following set of \emph{resonant} indexes:
\begin{equation}\label{resonantSet}
\begin{aligned}
\mathcal{R}&:=\big\{(\x,\eta,\zeta)\in \mathbb{Z}^{3d} \,:\, |\x|=|\zeta|\,, |\eta|=|\x-\eta-\zeta|\big\}
\\&\qquad\qquad\cup
\big\{(\x,\eta,\zeta)\in\mathbb{Z}^{3d} \,:\, |\x|=|\x-\eta-\zeta|\,, |\eta|=|\zeta|\big\}\,.
\end{aligned}
\end{equation}
Consider the vector field in \eqref{frate331} with Hamiltonian $H$ defined in \eqref{frate6}. We define the field
$X_{H}^{+,{\rm res}}(U)$ by
\begin{equation}\label{alien7}
\hat{X_{H}^{+,{\rm res}}}(\x)=(2\pi)^{-d}
\sum_{\eta,\zeta\in \mathbb{Z}^{d}}f^{({\rm res})}(\x,\eta,\zeta)
\hat{u}(\x-\eta-\zeta)\hat{\bar{u}}(\eta)\hat{u}(\zeta)\,,
\end{equation}
where
\begin{equation}\label{alien8}
f^{({\rm res})}(\x,\eta,\zeta):=f(\x,\eta,\zeta)
\mathtt{1}_{\mathcal{R}}(\x,\eta,\zeta)\,,
\end{equation}
where $\mathtt{1}_{\mathcal{R}}$ is the characteristic function of the set $\mathcal{R}$ and $f$ is defined in \eqref{frate44}.
\end{definition}
\noindent
In the next lemma we prove a fundamental cancellation.
\begin{lemma}\label{cancellazioneRes}
For $n\geq0$ one has (recall \eqref{Sobnorm})
\begin{equation}\label{alienalien88}
{\rm Re}(\langle D\rangle^{n}X_{H}^{+,{\rm res}}(U), \langle D\rangle^{n}u)_{L^{2}}
\equiv0\,.
\end{equation}
\end{lemma}

\begin{proof}
Using
 \eqref{resonantSet}-\eqref{alien8} one can check that
\[
\hat{X_{H}^{+,{\rm res}}}(\x)=(2\pi)^{-d}
\sum_{(\eta,\zeta)\in \mathcal{R}({\x})}
\mathcal{F}(\x,\eta,\zeta)
\hat{u}(\x-\eta-\zeta)\hat{\bar{u}}(\eta)\hat{u}(\zeta)\,,
\]
 with $\mathcal{R}(\x):=\{(\eta,\zeta)\in \mathbb{Z}^{2d}\,:\,|\x|=|\zeta| \,,
 |\eta|=|\x-\eta-\zeta|\}$, for $\x\in \mathbb{Z}^{d}$, and
 \begin{equation}\label{frate4}
\mathcal{F}(\x,\eta,\zeta):=
f(\x,\eta,\zeta)+f(\x,\eta,\x-\eta-\zeta)\,.
 \end{equation}
  By an explicit computation
 we have
 \[
 \begin{aligned}
  {\rm Re}(\langle D\rangle^{s} &X_{H}^{+,{\rm res}}(U),\langle D\rangle^{s}u)_{L^{2}}=\\
& = (2\pi)^{-d}
\sum_{\x\in \mathbb{Z}^{d}, (\eta,\zeta)\in\mathcal{R}(\x)}
  \langle\x\rangle^{2s}\Big[ \mathcal{F}(\x,\eta,\zeta)+\ov{\mathcal{F}(\zeta,\zeta+\eta-\x,\x)}\Big]
  \hat{u}(\x-\eta-\zeta)\hat{\bar{u}}(\eta)\hat{u}(\zeta)\hat{\bar{u}}(-\x)\,.
  \end{aligned}
 \]
 By \eqref{frate4}, \eqref{frate44} and using the symmetries \eqref{frate3}
 we have $
 \mathcal{F}(\x,\eta,\zeta)+\ov{\mathcal{F}(\zeta,\zeta+\eta-\x,\x)}=0\,$.
\end{proof}

\begin{remark}\label{rmkalg}
We remark that along the paper we shall deal with general Hamiltonian functions of the
form
\[
 H(W)=(2\pi)^{-d}\sum_{\substack{\s_1,\s_2,\s_3,\s_4\in \{\pm\}\\ 
 \x,\eta,\zeta\in\mathbb{Z}^{d}}}\mathtt{h}^{\s_1,\s_2,\s_3,\s_4}(\x,\eta,\zeta)
 \hat{u^{\s_1}}(\x-\eta-\zeta)\hat{u^{\s_2}}(\eta)\hat{u^{\s_3}}(\zeta)\hat{u^{\s_4}}(-\x)\,,
\]
where we used the notation
\begin{equation}\label{gordon1}
\hat{u^{\s}}(\cdot)=\hat{u}(\cdot)\,,\;\;{\rm if}\;\;\s=+\,,\quad {\rm and }\quad
\hat{u^{\s}}(\cdot)=\hat{\bar{u}}(\cdot)\,,\;\;{\rm if}\;\;\s=-\,.
\end{equation}
However, by the definition of the resonant set \eqref{resonantSet},
we can note that the resonant vector field has still the form
\eqref{alien7} and it depends only on the monomials in the Hamiltonian $H(U)$
which are \emph{gauge} invariant, i.e. of the form \eqref{frate6}.
\end{remark}

\section{Para-differential formulation of the problems}
In this section we rewrite the equations in a para-differential form
by means of the para-linearization formula (\emph{\`a la} Bony see \cite{bony}).
In subsection \ref{sec:3} we deal with the problem
\eqref{NLS} and in the \ref{sec:3KG} we deal with
\eqref{KG}.

\subsection{Para-linearization of the NLS}\label{sec:3}
In the following we para-linearize \eqref{NLS}, 
with respect to the variables $(u,\bar{u})$. We recall that \eqref{NLS} may be rewritten as \eqref{hamiltoniana} and we define $\tilde{P}(u):=P(u,\nabla u)-\frac12 |u|^4=\frac12|\nabla h(|u|^2)|^2$. We set
\begin{equation}\label{tildep}
\quad \tilde{p}(u):=(\pa_{\bar{u}}\tilde{P})(u,\nabla u)-\sum_{j=1}^{d}\pa_{x_j}
  \big(\pa_{\bar{u}_{x_{j}}}\tilde{P}\big)(u,\nabla u).
\end{equation}
We have the following.
\begin{lemma}\label{product}
Fix $s_0>d/2$ and $0\leq\rho<s-s_0$, $s\geq s_0$.
Consider $u\in H^{s}(\mathbb{T}^{d};\mathbb{C})$.
Then we have that 
\begin{align}
\tilde{p}(u)&=T_{\pa_{u\bar{u}}\tilde{P}}[u]+T_{\pa_{\bar{u}\,\bar{u}}\tilde{P}}[\bar{u}]
\label{paralin1}
\\&+\sum_{j=1}^{d}\Big(  T_{\pa_{\bar{u}u_{x_{j}}}\tilde{P}}[u_{x_j}]
+T_{\pa_{\bar{u}\,\ov{u_{x_{j}}}}\tilde{P}}[\ov{u_{x_j}}] \Big)
-\sum_{j=1}^{d}\pa_{x_j}\Big(  T_{\pa_{{u}\ov{u_{x_{j}}}}\tilde{P}}[u]
+T_{\pa_{\bar{u}\,\ov{u_{x_{j}}}}\tilde{P}}[\bar{u}] \Big)\label{paralin2}\\
&-\sum_{j=1}^{d}\pa_{x_{j}}
\Big( 
 T_{\pa_{\ov{u_{x_{j}}} \,{u_{x_j}}}\tilde{P}}[u_{x_j}]
+T_{\pa_{\ov{u_{x_{j}}}\, \ov{u_{x_{j}}}}\tilde{P}}[\bar{u_{x_{j}}}]
\Big)+R(u)\,,\label{paralin3}
\end{align}
where  $R(u)$ is a remainder satisfying 
\begin{equation}\label{stimarestopara}
\|R(u)\|_{H^{s+\rho}}\lesssim C\|u\|_{H^{s}}^{7}\,,
\end{equation}
for some constant $C>0$ depending on  $s,s_0$.
\end{lemma}
\begin{proof}
By using  the Bony para-linearization formula, 
see  \cite{bony,Metivier,Tay-Para}, and passing to the Weyl quantization we obtain 
  \begin{align}
\tilde{p}(u)&=T_{\pa_{u\bar{u}}\tilde{P}}[u]+T_{\pa_{\bar{u}\,\bar{u}}\tilde{P}}[\bar{u}]
\\&+\sum_{j=1}^{d}\Big(  T_{\pa_{\bar{u}u_{x_{j}}}\tilde{P}}[u_{x_j}]
+T_{\pa_{\bar{u}\,\ov{u_{x_{j}}}}\tilde{P}}[\ov{u_{x_j}}] \Big)
-\sum_{j=1}^{d}\pa_{x_j}\Big(  T_{\pa_{{u}\ov{u_{x_{j}}}}\tilde{P}}[u]
+T_{\pa_{\bar{u}\,\ov{u_{x_{j}}}}\tilde{P}}[\bar{u}] \Big)\\
&-\sum_{j=1}^{d}\pa_{x_{j}}
\sum_{k=1}^{d}\Big( 
 T_{\pa_{\ov{u_{x_{j}}} \,{u_{x_k}}}\tilde{P}}[u_{x_k}]
+T_{\pa_{\ov{u_{x_{j}}}\, \ov{u_{x_{k}}}}\tilde{P}}[\bar{u_{x_{k}}}]
\Big)+R(u)\,,\label{uru}
\end{align}
where $R(u)$ satisfies the estimate \eqref{stimarestopara} since $h(x)\sim x^2$ for $x\sim 0$. The first  term in \eqref{uru} is equal to the first in \eqref{paralin3} because $\partial_{\ov{u_{x_j}}u_{x_k}}\tilde{P}=\frac12\partial_{\bar{u}_{x_j}u_{x_k}}|\nabla h(|u|^2)|^2=0$ if $j\neq k$. \end{proof}

\bigskip We shall use the following 
notation throughout  the rest of the  paper
\begin{equation}\label{matriciozze}
U:=\vect{u}{\bar{u}}\,,\quad  E:=\sm{1}{0}{0}{-1}\,,\quad
\uno:=\sm{1}{0}{0}{1}\,,\quad \diag(b):=b\uno\,,\,\;\; b\in\mathbb{C}\,.
\end{equation}
Define the following \emph{real} symbols
\begin{equation}\label{simboa2}
\begin{aligned}
a_2(x):=&\, \left[h'(|u|^2)\right]^2|u|^2\,,
\quad b_2(x):=
\left[h'(|u|^2)\right]^2u^2,\\
\vec{a}_1(x)\cdot\xi:=&\,\left[h'(|u|^2)\right]^2
\sum_{j=1}^d\Im(u\bar{u}_{x_j})\xi_j\,, \quad \xi=(\xi_1,\ldots,\xi_d)\,. 
\end{aligned}
\end{equation}
We define also the matrix of functions
\begin{equation}\label{matriceA2}
A_2(x):=A_2(U;x):=\sm{a_2(U;x)}{b_2(U;x)}{\ov{b_2(U;x)}}{a_2(U;x)}=
\sm{a_2(x)}{b_2(x)}{\ov{b_2(x)}}{a_2(x)}
\end{equation}
with $a_2(x)$ and $b_2(x)$ defined in \eqref{simboa2}.
We have the following.

\begin{proposition}{\bf (Paralinearization of NLS).}\label{NLSparapara}
The equation \eqref{NLS}
is equivalent to the following system: 
\begin{equation}\label{QLNLS444}
\dot{U}=-\ii E\opbw\big((\uno +A_{2}(x))|\xi|^2\big)U-\ii EV*U
- \ii\opbw\big(\diag(\vec{a}_1(x)\cdot\xi)\big)U+
X_{\mathcal{H}^{(4)}_{\NLS}}(U)+R(U)\,,
\end{equation}
where 
$V$ is the convolution 
potential in \eqref{insPot},
the matrix $A_2(x)$ is the one in \eqref{matriceA2}, 
the symbol $\vec{a}_1(x)\cdot\x$
is in \eqref{simboa2}
and the vector field $X_{\mathcal{H}^{(4)}_{\NLS}}(U)$ is defined as follows
\begin{equation}\label{X_H}
X_{\mathcal{H}^{(4)}_{\NLS}}(U)=-\ii E\Big[
\opbw\Big(\sm{2|u|^{2}}{u^{2}}{\bar{u}^{2}}{2|u|^{2}}\Big)
U+Q_3(U)\Big].
\end{equation}
The semi-norms of the symbols satisfy the following estimates
\begin{equation}\label{realtaAAA2}
\begin{aligned}
|a_{2}|_{\mathcal{N}_p^{0}}+
|b_{2}|_{\mathcal{N}_p^{0}}&\lesssim  \|u\|^{6}_{H^{p+s_0}}\,, 
&\forall \, p+s_0\leq s\,,\,\quad p\in \mathbb{N},\\
|\vec{a}_{1}\cdot\xi|_{\mathcal{N}_p^{1}}&\lesssim \|u\|^{6}_{H^{p+s_0+1}}\,,
 &\forall \, p+s_0+1\leq s\,,\,\quad p\in \mathbb{N}\,,
\end{aligned}
\end{equation} 
where we have chosen $s_0>d$.
The remainder $Q_{3}(U)$
has the form $\big(Q^+_{3}(U), \ov{Q^+_{3}(U)}\big)^{T}$ and 
\begin{equation}\label{restoQ3KGNLS}
\widehat{Q^+_{3}}(\x)=(2\pi)^{-d}
\sum_{
\eta,\zeta\in\mathbb{Z}^{d}}\!\!
\mathtt{q}(\x,\eta,\zeta)
\hat{u}(\x-\eta-\zeta)\hat{\bar{u}}(\eta)\hat{u}(\zeta)\,,
\end{equation}
for some $\mathtt{q}(\x,\eta,\zeta)\in \mathbb{C}$.
The coefficients of $Q_{3}^{+}$ satisfy
\begin{equation}\label{restoQ32KGNLS}
|\mathtt{q}(\x,\eta,\zeta)|
\lesssim\frac{\max_2\{\langle\x-\eta-\zeta\rangle,\langle \eta\rangle,\langle\zeta\rangle\}^{\rho}}{\max\{\langle\x-\eta-\zeta\rangle,\langle \eta\rangle,\langle\zeta\rangle\}^{\rho}}\,,
\quad \forall \, \rho\geq 0\,.
\end{equation}
The remainder $R(U)$ has the form $(R^+(U),\ov{R^+(U)})^T$.
Moreover, for any $s>2d+2$,
we have the estimates 
\begin{equation}\label{stimaRRR}
\begin{aligned}
\|R(U)\|_{H^{s}}\lesssim& \,\|U\|_{H^{s}}^{7}\,,\qquad
\|Q_{3}(U)\|_{H^{s+2}}\lesssim\, \|U\|_{H^s}^{3}\,.
\end{aligned}
\end{equation}
\end{proposition}
\begin{proof}
By  Lemma \ref{lem:paraproduct} the cubic term $|u|^2u$ in \eqref{NLS} is equal to $2T_{|u|^2}u+T_{u^2}\bar{u}+\mathcal{R}(u,u,\bar{u})$. Setting $Q_3^+(U)=\mathcal{R}(u,u,\bar{u}),$ we get the \eqref{restoQ3KGNLS} by  the \eqref{eq:paraproduct2}. The second estimate in \eqref{stimaRRR} is a consequence of Lemma \ref{lem:trilineare} applied with $\rho=\mu=m=2$.

We now deal with the remaining quasi-linear term $\tilde{p}(u)$ defined in \eqref{tildep}.
We start by noting that
\begin{equation}\label{facile}
\pa_{x_{j}}:=\opbw(\ii\x_{j})\,,\;\; j=1,\ldots \,d\,,
\end{equation}
and that the quantization of  a symbol $a(x)$ is given by $\opbw(a(x))$.
We also remark that
the symbols 
appearing in \eqref{paralin1}, \eqref{paralin2} and \eqref{paralin3}
can be estimated  (in the norm $|\cdot|_{\mathcal{N}_{s}^{m}}$)
by using Lemma \ref{lem:nonomosimbo}.
Consider now the first para-differential term in \eqref{paralin3}. 
We have, for any $j=1,\ldots,d$,
\[
\pa_{x_{j}}
 T_{\pa_{\ov{u_{x_{j}}} \,{u_{x_j}}}\tilde{P}}\pa_{x_j}u=
 \opbw(\ii\x_{j})\circ\opbw(\pa_{\ov{u_{x_{j}}} \,{u_{x_j}}}\tilde{P})\circ\opbw(\ii \x_j)u\,.
\]
By applying 
Proposition \ref{prop:compo} 
and recalling the 
Poisson  bracket in \eqref{PoissonBra},
we deduce 
\begin{align}
 \opbw(\ii\x_{j})\,\circ&\,\opbw(\pa_{\ov{u_{x_{j}}} \,{u_{x_j}}}\tilde{P})\circ\opbw(\ii \x_j)
=
\opbw\big( -\x_{j}^2\pa_{\ov{u_{x_{j}}} \,{u_{x_j}}}\tilde{P} \big)
\\&
 +\opbw\Big(
  \frac{\ii}{2 } \x_{j}\pa_{x_{j}}(\pa_{\ov{u_{x_{j}}} \,{u_{x_j}}}\tilde{P})
 -\frac{\ii\x_{j}}{2}\pa_{x_{j}}(\pa_{\ov{u_{x_{j}}} \,{u_{x_j}}}\tilde{P})\Big)
    \\&+\widetilde{R}^{(1)}_{j}(u) +\widetilde{R}^{(2)}_{j}(u)\,,
\end{align}
where $\widetilde{R}^{(1)}_{j}(u):=
\opbw\big( 
 -\frac{1}{4}\pa_{{x_{j}}{x_{j}}}(\pa_{\ov{u_{x_{j}}} \,{u_{x_j}}}\tilde{P})\big)$
 and $\widetilde{R}^{(2)}_{j}(u)$ is some bounded operator.
 More precisely, using \eqref{composit2}, \eqref{actionSob} 
 and the estimates given by Lemma \ref{lem:nonomosimbo},
 we have, $\forall \, h\in H^{s}(\mathbb{T}^{d};\mathbb{C})$,
\begin{equation}\label{crociate}
\|\widetilde{R}^{(2)}_{j}(u)h\|_{H^{s}}\leq C\|h\|_{H^{s}}\|u\|_{H^{s}}^6\,,\qquad 
\|\widetilde{R}^{(1)}_{j}(u)h\|_{H^{s}}\leq C\|h\|_{H^{s}}\|u\|_{H^{2s_0+3}}^6\,,
\end{equation}
for some constant $C>0$ and  $s_0\geq d+1$, 
$s_0\in\mathbb{N}$.
We set
\[
\widetilde{R}(u):=\sum_{j=1}^{d}\Big(\widetilde{R}^{(1)}_{j}(u)
+\widetilde{R}^{(2)}_{j}(u)\Big)\,.
\]
Then 
\[
\begin{aligned}
-\sum_{j=1}^{d}\pa_{x_j} &T_{\pa_{\ov{u_{x_{j}}} \,{u_{x_j}}}\tilde{P}}\pa_{x_{j}}u=
\opbw\Big(\sum_{j=1}^{d}\x_{j}^2\pa_{\ov{u_{x_{j}}} \,{u_{x_j}}}\tilde{P}\Big)+
\widetilde{R}(u)
\\&
-\frac{\ii}{2}\opbw\Big(
\sum_{j=1}^{d}\Big( -\x_{j}\pa_{x_{j}}(\pa_{\ov{u_{x_{j}}} \,{u_{x_j}}}\tilde{P})
+ \x_{j}\pa_{x_{j}}(\pa_{\ov{u_{x_{j}}} \,{u_{x_j}}}\tilde{P})
\Big)\Big)
\\&
\stackrel{\eqref{simboa2}}{=}
\opbw(a_2(x)|\x|^2)+\widetilde{R}(u)
+\frac{\ii}{2}\opbw\Big(\sum_{j=1}^{d} \x_{j}\pa_{x_{j}}\Big(
(\pa_{\ov{u_{x_{j}}} \,{u_{x_j}}}\tilde{P})-(\pa_{\ov{u_{x_{j}}} \,{u_{x_j}}}\tilde{P})\Big)
\Big)\\
&
=\opbw(a_2(x)|\x|^2)+\widetilde{R}(u)\,,
\end{aligned}
\]
where we used the symmetry of the matrix $\pa_{\ov{\nabla u}\, \nabla u}\tilde{P}$ (recall $\tilde{P}$ is real) and that 
\begin{equation*}
\partial_{\ov{u_{x_j}}u_{x_j}}\tilde{P}(u)=\tfrac12{\partial_{\ov{u_{x_j}}u_{x_j}}}|\nabla h(|u|^2)|^2\stackrel{\eqref{simboa2}}{=} a_2(x).
\end{equation*}
By performing similar explicit computations on the other summands in \eqref{paralin1}-\eqref{paralin3}
we get the 
\eqref{QLNLS444}, \eqref{matriceA2} with symbols in 
\eqref{simboa2}.
By the discussion above we deduced that the remainder $R(U)$ in \eqref{QLNLS444}
satisfies the bound \eqref{stimaRRR}.
%
%
\end{proof}

\begin{remark}\label{cubica} 

\noindent 
$\bullet$ The cubic term $X_{\mathcal{H}^{(4)}_{\NLS}}(U)$ in \eqref{X_H} 
is the Hamiltonian vector field
of the  Hamiltonian function
\begin{equation}\label{hamiltonianaS}
{\mathcal{H}^{(4)}_{\NLS}}(U):=\frac12\int_{\T^d}|u|^4dx\,,\qquad 
X_{\mathcal{H}^{(4)}_{\NLS}}(U)=
-\ii |u|^{2}\vect{u}{\bar{u}}
\end{equation}

\noindent
$\bullet$ The operators $\opbw\big((\uno +A_{2}(x))|\xi|^2\big)$, 
$\opbw\big(\diag(\vec{a}_1(x)\cdot\xi)\big)$ and 
$\opbw\Big(\sm{2|u|^{2}}{u^{2}}{\bar{u}^{2}}{2|u|^{2}}\Big)$
are self-adjoint thanks to \eqref{simboAggiunto2} 
and \eqref{simboa2}.
\end{remark}

\subsection{Para-linearization of the KG}\label{sec:3KG}

In this section we rewrite the equation \eqref{KG} as a paradifferential system.
This is the content of Proposition \ref{KGparaparaKG}. Before stating 
this result we need some preliminaries. In particular in Lemma \ref{calGvecKG}
below we analyze some properties of the cubic terms in the equation \eqref{KG}.
Define the following \emph{real} symbols
\begin{equation}\label{simboa2KG}
\begin{aligned}
a_2(x,\x)&:=a_{2}(u;x,\x)
:=\sum_{j,k=1}^{d}\big(\pa_{\psi_{x_j}\psi_{x_{k}}} F\big)(\psi,\nabla \psi)\x_{j}\x_{k}\,,\qquad
\psi=\tfrac{\Lambda_{\KG}^{-\frac{1}{2}}}{\sqrt{2}}(u+\bar{u})\,,
\\
a_0(x,\x)&:=a_{0}(u;x,\x):=\tfrac{1}{2}(\pa_{y_1y_1}G)(\psi,\Lambda_{\KG}^{\frac{1}{2}}\psi)+
(\pa_{y_1y_0}G)(\psi,\Lambda_{\KG}^{\frac{1}{2}}\psi)\Lambda_{\KG}^{-\frac{1}{2}}(\x)\,.
\end{aligned}
\end{equation}
We define also the matrices of symbols
\begin{align}
\mathcal{A}_1(x,\x)&:=\mathcal{A}_1(u;x,\x):=\tfrac{1}{2}\sm{1}{1}{1}{1}\Lambda_{\KG}^{-2}(\x)a_2(u;x,\x)\,,
\label{calAA1KG}\\
\mathcal{A}_0(x,\x)&:=\mathcal{A}_0(u;x,\x)
:=\sm{1}{1}{1}{1}a_0(u;x,\x)\,,\label{calAA0KG}
\end{align}
and the Hamiltonian function
\begin{equation}\label{calGGGKG}
\mathcal{H}^{(4)}_{\KG}(U):=\int_{\mathbb{T}^{d}}
G\big({\psi},\Lambda_{\KG}^{\frac{1}{2}}\psi\big)dx\,,
\end{equation}
with $G$ the function appearing in \eqref{exHamKG}.
First of all we study some properties of the vector field of the Hamiltonian 
$\mathcal{H}^{(4)}_{\KG}$.
\begin{lemma}\label{calGvecKG}
We have that
\begin{equation}\label{calGGG1KG}
X_{\mathcal{H}^{(4)}_{\KG}}(U)=-\ii J\nabla\mathcal{H}^{(4)}_{\KG}(U)=
-\ii E\opbw(\mathcal{A}_0(x,\x))U+Q_{3}(u)\,,
\end{equation}
with $\mathcal{A}_0$ in \eqref{calAA0KG}. The remainder $Q_{3}(u)$
has the form $\big(Q^+_{3}(u), \ov{Q^+_{3}(u)}\big)^{T}$ and (recall \eqref{gordon1})
\begin{equation}\label{restoQ3KG}
\widehat{Q^+_{3}}(\x)=\frac{1}{(2\pi)^{d}}\sum_{\substack{\s_1,\s_2,\s_{3}\in\{\pm\}\\
\eta,\zeta\in\mathbb{Z}^{d}}}\!\!
\mathtt{q}^{\s_1,\s_2,\s_3}(\x,\eta,\zeta)
\hat{u^{\s_1}}(\x-\eta-\zeta)\hat{u^{\s_2}}(\eta)\hat{u^{\s_3}}(\zeta)\,,
\end{equation}
for some $\mathtt{q}^{\s_1,\s_2,\s_3}(\x,\eta,\zeta)\in \mathbb{C}$.
The coefficients of $Q_{3}^{+}$ satisfy
\begin{equation}\label{restoQ32KG}
|\mathtt{q}^{\s_1,\s_2,\s_3}(\x,\eta,\zeta)|
\lesssim\frac{\max_2\{\langle\x-\eta-\zeta\rangle,\langle \eta\rangle,\langle\zeta\rangle\}}{\max\{\langle\x-\eta-\zeta\rangle,\langle \eta\rangle,\langle\zeta\rangle\}}
\end{equation}
for any $\s_1,\s_2,\s_3\in\{\pm\}$. Finally, for $s>2d+1$, we have
\begin{equation}\label{calA00KG}
|a_0|_{\mathcal{N}_p^{0}}\lesssim \|u\|_{H^{p+s_0}}^{2}\,,\quad
p+s_0\leq s\,,\quad s_0>d\,,
\end{equation}
\begin{equation}\label{stimaGGKG}
\|X_{\mathcal{H}^{(4)}_{\KG}}(U)\|_{H^{s}}\lesssim \|u\|_{H^{s}}^{3}\,,\qquad
\|Q_{3}(u)\|_{H^{s+1}}\lesssim\, \|u\|_{H^s}^{3}\,,
\end{equation}
\begin{equation}\label{stimaGG2KG}
\|d_{U}X_{\mathcal{H}^{(4)}_{\KG}}(U)[h]\|_{H^{s}}\lesssim \|u\|_{H^{s}}^{2}\|h\|_{H^{s}}\,,\qquad\forall
h\in H^{s}(\mathbb{T}^{d};\mathbb{C}^{2})\,.
\end{equation}
\end{lemma}
\begin{proof}
By an explicit computation and using \eqref{KGnon2}
we get
\[
X_{\mathcal{H}^{(4)}_{\KG}}(U)=\left(
X^{+}_{\mathcal{H}^{(4)}_{\KG}}(U),
\ov{X^{+}_{\mathcal{H}^{(4)}_{\KG}}(U)}
\right)^T\,,\qquad X_{\mathcal{H}^{(4)}_{\KG}}^{+}(U)=
-\ii\tfrac{\Lambda_{\KG}^{-\frac{1}{2}}}{\sqrt{2}}
g\big(\psi\big)\,.
\]
The function $g$ is a homogeneous polynomial of degree three. Hence,
by using Lemma \ref{lem:paraproduct},
we obtain 
\begin{equation}\label{oriente100KG}
\begin{aligned}
\ii X_{\mathcal{H}^{(4)}_{\KG}}^{+}(U)&=A_0+A_{-\frac{1}{2}}+A_{-1}
+Q^{-\rho}(u)
\end{aligned}
\end{equation}
where
\begin{align}
A_0&:=\frac{1}{2}\opbw(\pa_{y_1y_1}G(\psi,\Lambda_{\KG}^{1/2} \psi))[u+\bar{u}]\,,\label{oriente20KG}\\
A_{-\frac{1}{2}}&:=\frac{1}{2}\opbw(\pa_{y_1y_0}G(\psi,\Lambda_{\KG}^{1/2} \psi))[\Lambda_{\KG}^{-\frac{1}{2}}(u+\bar{u})]+
\frac{\Lambda_{\KG}^{-\frac{1}{2}}}{2}\opbw(\pa_{y_1y_0}G(\psi,\Lambda_{\KG}^{1/2} \psi))[u+\bar{u}]\,,\label{oriente21KG}\\
A_{-1}&:=
\frac{\Lambda_{\KG}^{-\frac{1}{2}}}{2}\opbw(\pa_{y_0y_0}G(\psi,\Lambda_{\KG}^{1/2} \psi))[\Lambda_{\KG}^{-\frac{1}{2}}(u+\bar{u})]
\label{oriente22KG}\,,
\end{align}
and $Q^{-\rho}$ is a cubic smoothing remainder of the form
\eqref{eq:paraproduct2}
whose coefficients satisfy the bound \eqref{restoQ32KG}.
The symbols of the the paradifferential operators have the form
(using that $G$ is a polynomial)
\begin{equation}\label{oriente101KG}
(\pa_{kj}G)\Big(\tfrac{\Lambda_{\KG}^{-\frac{1}{2}}(u+\bar{u})}{\sqrt{2}},
\tfrac{u+\bar{u}}{\sqrt{2}}
\Big)=(2\pi)^{-d}\sum_{\sigma_1,\sigma_2\in\{\pm\}}\sum_{\x\in \mathbb{Z}^{d}}e^{\ii \x\cdot x}\sum_{\eta\in \mathbb{Z}^{d}}
\mathtt{g}_{k,j}^{\s_1,\s_2}(\x,\eta)\hat{u^{\s_1}}(\x-\eta)\hat{u^{\s_2}}(\eta)
\end{equation}
where $k,j\in\{y_0,y_1\}$ and where the coefficients $\mathtt{g}_{k,j}^{\s_1,\s_2}(\x,\eta)\in \mathbb{C}$
satisfy $|\mathtt{g}_{k,j}^{\s_1,\s_2}(\x,\eta)|\lesssim1$.

\noindent
We claim that  the term in \eqref{oriente22KG}
is a cubic remainder of the form \eqref{restoQ3KG}
with coefficients satisfying \eqref{restoQ32KG}.
 By \eqref{quantiWeyl} we have
\[
\begin{aligned}
\widehat{A_{-1}}(\x)&=\frac{1}{2(2\pi)^{d}}\sum_{\zeta\in \mathbb{Z}^{d}, \s\in\{\pm\}}
\widehat{\pa_{y_0y_0}G}(\x-\zeta)\Lambda_{\KG}^{-\frac{1}{2}}(\x)\Lambda_{\KG}^{-\frac{1}{2}}(\zeta)
\chi_{\epsilon}\Big(\frac{|\x-\zeta|}{\langle \x+\zeta\rangle}\Big)
\hat{u^{\s}}(\x)\\
&\stackrel{\mathclap{\eqref{oriente101KG}}}{=}\;\;\frac{1}{2(2\pi)^{d}}
\sum_{\substack{\s_1,\s_2,\s\in\{\pm\} \\
\eta,\zeta\in\mathbb{Z}^{d}}}
\mathtt{g}_{y_0,y_0}^{\s_1,\s_2}(\x-\zeta,\eta)
\Lambda_{\KG}^{-\frac{1}{2}}(\x)\Lambda_{\KG}^{-\frac{1}{2}}(\zeta)
\chi_{\epsilon}\Big(\frac{|\x-\zeta|}{\langle \x+\zeta\rangle}\Big)
\hat{u^{\s_1}}(\x-\eta-\zeta)\hat{u^{\s_2}}(\eta)\hat{u^{\s}}(\zeta)\,,
\end{aligned}
\]
which implies that $A_{-1}$ has the form
\eqref{restoQ3KG} with coefficients
\begin{align}
\mathtt{a}_{-1}^{\s_1,\s_2,\s_3}(\x,\eta,\zeta)
&=\frac{1}{2}
\mathtt{g}_{y_0,y_0}^{\s_1,\s_2}(\x-\zeta,\eta)
\Lambda_{\KG}^{-\frac{1}{2}}(\x)\Lambda_{\KG}^{-\frac{1}{2}}(\zeta)
\chi_{\epsilon}\Big(\tfrac{|\x-\zeta|}{\langle \x+\zeta\rangle}\Big)\,.\label{gotham2KG}
\end{align}
By Lemma \ref{lem:paratri} 
we have that the coefficients in \eqref{gotham2KG} satisfy \eqref{restoQ32KG}.
This prove the claim for the operator $A_{-1}$. 
We now study the term in \eqref{oriente21KG}.
We remark that, by Proposition \ref{prop:compo}
(see the composition formula \eqref{composit}), we have
that $A_{-1/2}=\opbw(\Lambda_{\KG}^{-\frac{1}{2}}(\x)\pa_{y_0y_1}G)$ up to a 
smoothing operator of order $-3/2$. 
Actually to prove that such a remainder has the form \eqref{restoQ3KG} with coefficients
\eqref{restoQ32KG} it is more convenient to compute the composition operator 
explicitly. 
In particular,  recalling \eqref{quantiWeyl},
we get
\begin{equation}\label{gotham3KG}
A_{-\frac{1}{2}}=\opbw(\Lambda_{\KG}^{-\frac{1}{2}}(\x)\pa_{y_0y_1}G)+R_{-1}\,,
\end{equation}
where
\[
\begin{aligned}
\widehat{R_{-1}}(\x)&=(2\pi)^{-d}
\sum_{\substack{\s_1,\s_2,\s\in\{\pm\} \\
\eta,\zeta\in\mathbb{Z}^{d}}}
\mathtt{r}^{\s_1,\s_2,\s}(\x-\eta-\zeta,\eta,\zeta)
\hat{u^{\s_1}}(\x-\eta-\zeta)\hat{u^{\s_2}}(\eta)\hat{u^{\s}}(\zeta)\,,\\
\mathtt{r}^{\s_1,\s_2,\s}(\x-\eta-\zeta,\eta,\zeta)&=\tfrac{1}{2}
\mathtt{g}_{y_0,y_1}^{\s_1,\s_2}(\x-\zeta,\eta)
\chi_{\epsilon}\Big(\tfrac{|\x-\zeta|}{\langle \x+\zeta\rangle}\Big)
\big[ \Lambda_{\KG}^{-\frac{1}{2}}(\x)+
\Lambda_{\KG}^{-\frac{1}{2}}(\zeta)-2\Lambda_{\KG}^{-\frac{1}{2}}(\tfrac{\x+\zeta}{2})
\big]\,.
\end{aligned}
\]
We note that
\[
\Lambda_{\KG}^{-\frac{1}{2}}(\x)=\Lambda_{\KG}^{-\frac{1}{2}}(\tfrac{\x+\zeta}{2})
-\frac{1}{2}\int_{0}^{1}\Lambda_{\KG}^{-\frac{3}{2}}(\tfrac{\x+\zeta}{2}+\tau \tfrac{\x-\zeta}{2})
d\tau\,.
\]
Then we deduce
\[
\Big|  \Lambda_{\KG}^{-\frac{1}{2}}(\x)+
\Lambda_{\KG}^{-\frac{1}{2}}(\zeta)-2\Lambda_{\KG}^{-\frac{1}{2}}(\tfrac{\x+\zeta}{2})\Big|
\lesssim |\x|^{-\frac{3}{2}}+|\zeta|^{-\frac{3}{2}}\,.
\]
Again by Lemma \ref{lem:paratri}
one can conclude that $\mathtt{r}^{\s_1,\s_2,\s}(\x-\eta-\zeta,\eta,\zeta)$ 
satisfies the \eqref{restoQ32KG}.
By  \eqref{gotham3KG}, \eqref{oriente20KG}, \eqref{oriente22KG} 
and recalling the definition 
of $a_{0}(x,\x)$ in \eqref{simboa2KG}, we obtain the \eqref{calGGG1KG}.
%
The bound 
\eqref{stimaGGKG} for $Q_{3}$
 follows by \eqref{restoQ32KG} and Lemma \ref{lem:trilineare}.
 Moreover the bound \eqref{calA00KG} follows
 by Lemma \ref{lem:nonomosimbo}
recalling that $G(\psi,\Lambda_{\KG}^{\frac{1}{2}} \psi)\sim O(u^{4})$.
Then the bound \eqref{stimaGGKG} for $X_{\mathcal{H}^{(4)}_{\KG}}$ follows
by Lemma \ref{azioneSimboo}. Let us prove the \eqref{stimaGG2KG}.
By differentiating \eqref{calGGG1KG}
we get
\begin{equation}\label{gordon2KG}
d_{U}X_{\mathcal{H}^{(4)}_{\KG}}(U)[h]=-\ii E\opbw(\mathcal{A}_0(x,\x))h
-\ii E\opbw(d_{U}\mathcal{A}_0(x,\x)h)U+d_{U}Q_{3}(u)[h]\,.
\end{equation}
The first summand in \eqref{gordon2KG} satisfies \eqref{stimaGG2KG} by Lemma \ref{azioneSimboo} and \eqref{calA00KG}. Moreover using \eqref{oriente101KG}
and \eqref{simboa2KG}
one can check that 
\[
|d_{U}\mathcal{A}_0(x,\x)h|_{\mathcal{N}_{p}^{0}}\lesssim 
\|u\|_{H^{p+s_0}}\|h\|_{H^{p+s_0}}\,,\quad p+s_0\leq s\,.
\]
Then the second summand in \eqref{gordon2KG} verify the bound \eqref{stimaGG2KG}
again by Lemma \ref{azioneSimboo}. The estimate on the third summand in \eqref{gordon2KG}
follows by \eqref{restoQ3KG}, \eqref{restoQ32KG} 
and Lemma \ref{lem:trilineare}.
 \end{proof}
 
\begin{remark}\label{strutturaA0KG}
We remark that the symbol $a_0(x,\x)$
 in \eqref{simboa2KG} is homogenenous of degree two in the variables $u,\bar{u}$.
 In particular, by \eqref{oriente101KG}, we have
 \begin{equation}\label{strutturaA01KG}
 \begin{aligned}
&a_0(x,\x)=(2\pi)^{-\frac{d}{2}}\sum_{p\in\mathbb{Z}^{d}}e^{\ii p\cdot x}\widehat{a_0}(p,\x)\,,
\qquad
\widehat{a_0}(p,\x)=(2\pi)^{-d}
\sum_{\substack{\s_1,\s_2\in\{\pm\}\\\eta\in\mathbb{Z}^{d}}}
a_0^{\s_1,\s_2}(p,\eta,\x)\widehat{u^{\s_1}}(p-\eta)\widehat{u^{\s_2}}(\eta)\\
&a_0^{\s_1,\s_2}(p,\eta,\x):=
\frac{1}{2}\mathtt{g}_{y_1,y_1}^{\s_1,\s_2}(p,\eta)+\mathtt{g}_{y_0,y_1}^{\s_1,\s_2}(p,\eta)
\Lambda_{\KG}^{-\frac{1}{2}}(\x)\,.
 \end{aligned}
 \end{equation}
 Moreover one has $|a_0^{\s_1,\s_2}(p,\eta,\x)|\lesssim1$. Since the symbol $a_0(x,\x)$
 is real-valued one can check that
 \begin{equation}\label{mazzo2KG}
 a_0^{\s_1,\s_2}(p,\eta,\x)=\ov{a_0^{-\s_1,-\s_2}(-p,-\eta,\x)}\,,\qquad \forall\; \x,p,\eta\in \mathbb{Z}^{d}\,,\;\s_1,\s_2\in\{\pm\}\,.
 \end{equation}
 \end{remark}

 \begin{remark}\label{semilin1}
 Consider the special case when the function $G$ in \eqref{KGnon2} 
 is independent of $y_1$.
 Following the proof of Lemma \ref{calGvecKG}
 one could obtain the formula \eqref{calGGG1KG}
 with symbol $a_{0}(x,\x)$ of order $-1$ given by (see \eqref{oriente22KG})
 \[
 a_{0}(x,\x):=\tfrac{1}{2}\pa_{y_0y_0}G(\psi)\Lambda_{\KG}^{-1}(\x)\,.
 \]
 The remainder $Q_3$ would satisfy the \eqref{restoQ32KG} with better 
 denominator $\max\{\langle\x-\eta-\zeta\rangle,\langle \eta\rangle,\langle\zeta\rangle\}^{2}$.
 
 \end{remark}

The main result of this section is  the following.

\begin{proposition}{\bf (Paralinearization of KG).}\label{KGparaparaKG}
The system  \eqref{exHameq}
is equivalent to
\begin{equation}\label{QLNLS444KG}
\dot{U}=-\ii E\opbw\big((\uno +\mathcal{A}_{1}(x,\x))\Lambda_{\KG}(\x)\big)U+
X_{\mathcal{H}^{(4)}_{\KG}}(U)+{R}(u)\,,
\end{equation}
where $U:=\vect{u}{\bar{u}}:=\mathcal{C}\vect{\psi}{\phi}$ (see \eqref{CVWW}),
$\mathcal{A}_1(x,\x)$ is in \eqref{calAA1KG},
$X_{\mathcal{H}^{(4)}_{\KG}}(U)$ is the Hamiltonian vector field 
of \eqref{calGGGKG}.
The operator ${R}(u)$
has the form
$({R}^+(u), \ov{{R}^+(u)})^{T}$.
Moreover we have that
\begin{align}
|\mathcal{A}_1|_{\mathcal{N}_p^{0}}
+|a_{2}|_{\mathcal{N}_p^{2}}+&\lesssim \|u\|^{3}_{H^{p+s_0+1}}\,,
\quad \forall \, p+s_0+1\leq s\,,\,\quad p\in \mathbb{N}\,,\label{realtaAAA2KG}
\end{align}
where we have chosen $s_0>d$.
Finally there is $\mu>0$ such that, for any $s>2d+\mu$,
the remainder ${R}(u)$ satisfies
\begin{equation}\label{stimaRRRKG}
\begin{aligned}
\|R(u)\|_{H^{s}}\lesssim& \,\|u\|_{H^{s}}^{4}\,.
\end{aligned}
\end{equation}
\end{proposition}

\begin{proof}
First of all we note that system \eqref{exHameq} 
in the complex coordinates \eqref{CVWW}
reads
\begin{equation}\label{oriente10KG}
 \pa_{t}u=-\ii \Lambda_{\KG} u-\ii\tfrac{\Lambda_{\KG}^{-\frac{1}{2}}}{\sqrt{2}}(f(\psi)+g(\psi))\,,
\qquad 
\psi=\tfrac{\Lambda_{\KG}^{-\frac{1}{2}}(u+\bar{u})}{\sqrt{2}}\,,
\end{equation}
with $f(\psi)$, $g(\psi)$ in \eqref{KGnon}, \eqref{KGnon2}.
The term $-\ii/\sqrt{2}\Lambda_{\KG}^{-1/2}g(\psi)$ is the first component 
of the vector field $X_{\mathcal{H}^{(4)}_{\KG}}(U)$ which has been studied in Lemma \ref{calGvecKG}.
By using the Bony para-linearization formula (see \cite{bony,Metivier, Tay-Para}), passing to the Weyl quantization 
 and \eqref{KGnon}
we get
\begin{align}
f(\psi)&=-\sum_{j,k=1}^{d}\pa_{x_j}\circ
\opbw\Big(\big(\pa_{\psi_{x_j}\psi_{x_{k}}} F\big)(\psi,\nabla \psi) \Big)\circ\pa_{x_{k}}\psi
\label{house1KG}\\
&+\sum_{j=1}^{d}\Big[\opbw\big( \big(\pa_{\psi \psi_{x_{j}}} F\big)(\psi,\nabla \psi)\big) , 
\pa_{x_j} \Big]\psi
+\opbw\big( \big(\pa_{\psi \psi} F\big)(\psi,\nabla \psi)\big)\psi+R^{-\rho}(\psi)\,,
\label{house2KG}
\end{align}
where $R^{-\rho}(\psi)$ satisfies $\|R^{-\rho}(\psi)\|_{H^{s+\rho}}\lesssim 
\|\psi\|_{H^{s}}^{4}$
for any $s\geq s_0>d+\rho$.
By Lemma \ref{lem:nonomosimbo}, 
and recalling that $F(\psi,\nabla \psi)\sim O(\psi^{5})$, we 
have that
\begin{equation}\label{orienteKG}
|\pa_{\psi_{x_k} \psi_{x_{j}}} F|_{\mathcal{N}_p^{0}}
+|\pa_{\psi \psi_{x_{j}}} F|_{\mathcal{N}_p^{0}}+
|\pa_{\psi \psi} F|_{\mathcal{N}_p^{0}}\lesssim \|\psi\|_{H^{p+s_0+1}}^{3}\,,\quad
p+s_0+1\leq s\,,
\end{equation}
where $s_0>d$.
Recall that $\pa_{x_j}=\opbw(\x_{j})$. Then, by Proposition \ref{prop:compo},
we have
\[
\Big[\opbw\big( \pa_{\psi \psi_{x_{j}}} F\big) , \pa_{x_j} \Big]\psi=
\opbw(-\ii \{\pa_{\psi \psi_{x_{j}}} F,\x_{j}\})\psi+Q(\psi)
\]
with (see \eqref{composit2})
$\|Q(\psi)\|_{H^{s+1}}\lesssim | \pa_{\psi \psi_{x_{j}}} F|_{\mathcal{N}_{s_0+2}^{0}}\|\psi\|_{H^{s}}$.
Then by \eqref{prodSimboli}, \eqref{orienteKG} and \eqref{actionSob}
(see Lemma \ref{azioneSimboo} and Proposition \ref{prop:compo})
we deduce that the terms in \eqref{house2KG}
can be absorbed in a  remainder satisfying \eqref{stimaRRRKG} 
with $s\gg2d$ large enough.
We now consider  the r.h.s. of \eqref{house1KG}.
We have
\[
-\pa_{x_j}\circ
\opbw\Big(\big(\pa_{\psi_{x_j}\psi_{x_{k}}} F\big)(\psi,\nabla \psi) \Big)\circ\pa_{x_{k}}=
\opbw(\x_{j})\opbw\Big(\big(\pa_{\psi_{x_j}\psi_{x_{k}}} F\big)(\psi,\nabla \psi) \Big)\opbw(\x_{k})\,.
\]
By using again Lemma \ref{azioneSimboo} 
and Proposition \ref{prop:compo}
we get that
\begin{equation}\label{oriente3KG}
f(\psi)=\opbw(a_{2}(x,\x))\psi+\tilde{{R}}(\psi)\,,
\end{equation}
where $a_2$ is in \eqref{simboa2KG} and $\tilde{R}(\psi)$ is a remainder 
satisfying \eqref{stimaRRRKG}. The symbol $a_{2}(x,\x)$ satisfies 
\eqref{realtaAAA2KG} by \eqref{orienteKG}.
Moreover 
\begin{equation}\label{oriente2KG}
\tfrac{1}{\sqrt{2}}\Lambda_{\KG}^{-\frac{1}{2}}f(\psi)=
\tfrac{1}{\sqrt{2}}\Lambda_{\KG}^{-\frac{1}{2}}f(\tfrac{\Lambda_{\KG}^{-\frac{1}{2}}(u+\bar{u})}{\sqrt{2}})
\stackrel{\eqref{oriente3KG}}{=}\tfrac{1}{2}\opbw(a_{2}(x,\x)\Lambda_{\KG}^{-1}(\x))[u+\bar{u}]
\end{equation}
up to  remainders satisfying \eqref{stimaRRRKG}. 
Here we used Proposition \ref{prop:compo} to study the composition operator
$
\Lambda_{\KG}^{-\frac{1}{2}}\opbw(a_2(x,\x))\Lambda_{\KG}^{-\frac{1}{2}}\,.
$
By the discussion above and formula \eqref{oriente10KG}
we deduce the \eqref{QLNLS444KG}.
\end{proof}
%

\begin{remark}\label{semilin2}
In the semi-linear case, i.e. when $f=0$ and $g$
does not depend on $y_1$  (see \eqref{KGnon}, \eqref{KGnon2}) , 
the equation \eqref{QLNLS444KG} reads
\[
\dot{U}=-\ii E\opbw\big(\uno\Lambda_{\KG}(\x)\big)U+
X_{\mathcal{H}^{(4)}_{\KG}}(U)\,,
\]
and where the vector field $X_{\mathcal{H}^{(4)}_{\KG}}$ has the particular structure described in Remark \ref{semilin1}.
\end{remark}

\section{Approximately symplectic maps}\label{flow-ham}
\subsection{Para-differential Hamiltonian vector fields}
In this section we shall construct some approximatively symplectic changes of coordinates which will be important for the diagonalization procedure of Section \ref{diago}.

Define
the following frequency  localization:
\begin{equation}\label{innercutoff}
S_{\xi} w:= \sum_{k\in \mathbb{Z}^{d}} \hat{w}(k)
\chi_{\epsilon}\Big(\frac{|k|}{\langle \x\rangle}\Big)e^{\ii k\cdot x}\,,
\qquad \x\in \mathbb{Z}^{d}\,,
\end{equation}
for some $0<\epsilon<1$, where $\chi_{\epsilon}$ is defined in \eqref{cutofffunct}.
Consider the matrix of  symbols
\begin{equation}\label{gene1}
B_{\NLS}(W;x,\x):=B_{\NLS}(x,\x):=\left(
\begin{matrix} 0 & b_{\NLS}(x,\x)\\
\ov{b_{\NLS}(x,-\x)} & 0
\end{matrix}
\right)\,, \quad
b_{\NLS}(x,\x)=\tilde{\chi}(\xi)w^{2}\frac{1}{2|\x|^{2}}\,,
\end{equation}
where $\tilde{\chi}(\xi)$ is a $C^{\infty}(\mathbb{R};\mathbb{R}^+)$ function equal to $0$ if $|\xi|\leq 1/4$ and $1$ if $|\xi|\geq 1/2$.
Define also   the Hamiltonian function 
\begin{equation}\label{HamFUNC}
\mathcal{B}_{\NLS}(W):=\frac{1}{2}\int_{\mathbb{T}^{d}}\ii E
\opbw(B_{\NLS}(S_{\x}W;x,\x))W\cdot \ov{W}dx\,,
\end{equation}
where $S_{\x}W:=({S_{\x}w}, {S_{\x}\ov{w}})^{T}$\,. The presence of truncation on the high modes ($S_{\x}$) will be decisive in obtaining Lemma \ref{HamvectorFieldG} (see comments in the proof of this lemma).\\
Analogously we define the following.
Consider the matrix of symbols
\begin{equation}\label{gene1KG}
B_{\KG}(W;x,\x):=B_{\KG}(x,\x):=\left(
\begin{matrix} 0 & b_{\KG}(x,\x)\\
\ov{b_{\KG}(x,-\x)} & 0
\end{matrix}
\right)\,, \quad
b_{\KG}(W;x,\x)=\frac{a_0(x,\x)}{2\Lambda_{\KG}(\x)}\,,
\end{equation}
with $a_0(x,\xi)$ in \eqref{simboa2KG} and $\Lambda_{\KG}$ in \eqref{def:Lambda2}, and define the Hamiltonian function
\begin{equation}\label{HamFUNCKG}
\mathcal{B}_{\KG}(W):=\frac{1}{2}\int_{\mathbb{T}^{d}}\ii E
\opbw(B_{\KG}(S_{\x}W;x,\x))W\cdot \ov{W}dx\,,
\end{equation}
where $S_{\x}W:=({S_{\x}w}, {S_{\x}\ov{w}})^{T}$ where $S_{\x}$ is in \eqref{innercutoff}.

In this section we study some properties of the maps  
generated by the Hamiltonians $\mathcal{B}_{\NLS}(W)$ in \eqref{HamFUNC}
and  $\mathcal{B}_{\KG}(W)$ in \eqref{HamFUNCKG}.
In the next lemma 
we show that their Hamiltonian vector fields 
are given by $\opbw(B_{\NLS}(W;x,\x))W$
and $\opbw(B_{\KG}(W;x,\x))W$ respectively, 
modulo smoothing remainders.
More precisely we have the following.

\begin{lemma}\label{HamvectorFieldG}
Consider the Hamiltonian function $\mathcal{B}(W)$ equal to $\mathcal{B}_{\NLS}$ in \eqref{HamFUNC} or $\mathcal{B}_{\KG}$ in \eqref{HamFUNCKG}.
One has that the Hamiltonian vector field  of $\mathcal{B}(W)$ has the form 
\begin{equation}\label{HamvecField}
X_{\mathcal{B}}(W)=-\ii J \nabla \mathcal{B}(W)=\opbw(B(W;x,\x))W+Q_{\mathcal{B}}(W)\,,
\end{equation}
where $Q_{\mathcal{B}}(W)$ is a smoothing remainder 
of the form $(Q^{+}_{\mathcal{B}}(W),\ov{Q^{+}_{\mathcal{B}}(W)})^{T}$
and the symbol $B(W;x,\x)$
is respectively equal to $B_{\NLS}(W;x,\x)$ in \eqref{gene1}
or $B_{\KG}(W;x,\x)$ in \eqref{gene1KG}.
In particular the cubic remainder $Q_{\mathcal{B}}(W)$ has the form
\begin{equation}\label{formaQG}
\begin{aligned}
&\widehat{(Q_{\mathcal{B}}^{+}(W))}(\x)=
\frac{1}{(2\pi)^{d}}\sum_{\substack{ \s_1,\s_2,\s_3\in\{\pm\}\\ 
\eta,\zeta\in\mathbb{Z}^{d}}}
\mathtt{q}_{\mathcal{B}}^{\s_1,\s_2,\s_3}(\x,\eta,\zeta)\hat{w^{\s_1}}(\x-\eta-\zeta)\hat{w^{\s_2}}(\eta)\hat{w^{\s_3}}(\zeta)\,,\quad 
\x\in \mathbb{Z}^{d} \,,
\end{aligned} 
\end{equation}
where $\mathtt{q}_{\mathcal{B}}^{\s_1,\s_2,\s_3}(\x,\eta,\zeta)\in \mathbb{C}$ satisfy, 
for any $\x,\eta,\zeta\in \mathbb{Z}^{d}$, 
a bound like \eqref{eq:paraproduct2}.
In the case that $\mathcal{B}=\mathcal{B}_{\NLS}$
we have that $\s_1=+,\s_2=-,\s_3=+$.
Moreover, for  $s> d/2+\rho$, we have the following 
 \begin{equation}\label{stimaQG}
\begin{aligned}
\|d^{k}_{W}Q_{\mathcal{B}}(W)[h_1,\ldots,h_k]\|_{H^{s+\rho}}&
\lesssim\|w\|^{3-k}_{H^{s}}\prod_{i=1}^{k}
\|h_i\|_{H^{s}}\,,\qquad
\forall\, h_i\in H^{s}(\mathbb{T}^{d};\mathbb{C}^{2})\,, \;\; i=1,2,3\,,
\end{aligned}
\end{equation}
for $k=0,1,2,3$.
Moreover, for any $s> 2d+2$, one has
\begin{align}
\|d^{k}_{W}X_{\mathcal{B}_{\NLS}}(W)[h_1,\ldots,h_k]\|_{H^{s+2}}&\lesssim
\|w\|^{3-k}_{H^{s}}
\prod_{i=1}^{k}
\|h_i\|_{H^{s}}\,,\qquad
\forall\, h_i\in H^{s}(\mathbb{T}^{d};\mathbb{C}^{2})\,,
\; i=1,2,3\,,
\label{stimaXG4}\\
\|d^{k}_{W}X_{\mathcal{B}_{\KG}}(W)[h_1,\ldots,h_k]\|_{H^{s+1}}&\lesssim
\|w\|^{3-k}_{H^{s}}
\prod_{i=1}^{k}
\|h_i\|_{H^{s}}\,,\qquad
\forall\, h_i\in H^{s}(\mathbb{T}^{d};\mathbb{C}^{2})\,,
\; i=1,2,3\,,
\label{stimaXG4KG}
\end{align}
with $k=0,1,2,3$.
\end{lemma}

\begin{proof}
We prove the statement in the case $\mathcal{B}=\mathcal{B}_{\NLS}$,
the other case is similar.
Using the  formul\ae\, \eqref{gene1}, \eqref{HamFUNC} we obtain 
$\mathcal{B}_{\NLS}(W)=-G_1(W)-G_2(W)$ with
\[
G_1(W):=-\frac{\ii}{2}\int_{\T^d}\opbw(b_{\NLS}(S_{\xi}w))\bar{w}\,\bar{w}dx\,,
\qquad 
G_2(W):=\frac{\ii}{2}\int_{\T^d}\opbw(\ov{b_{\NLS}(S_{\xi}w})wwdx\,,
\]
where we recall \eqref{innercutoff}.
By \eqref{gene1} we obtain that 
$\nabla_{\bar{w}}G_1(W)=-\ii\opbw(b_{\NLS}(S_{\xi}w))\bar{w}$. 
We compute the gradient with respect $\bar{w}$ of the term $G_2(W)$. We have 
\begin{align*}
d_{\bar{w}}G_2(W)(\bar{h})&=
\tfrac{\ii}{2}\int_{\T^d}\opbw(S_{\xi}(\bar{w})S_{\xi}(\bar{h})\tfrac{1}{|\xi|^2}\tilde{\chi}(\xi))w\,wdx\\
&\stackrel{\mathclap{\eqref{quantiWeyl}}}{=}
\tfrac{\ii}{2}\tfrac{1}{(2\pi)^d}\sum_{\xi,\eta,\zeta\in\Z^d}
\hat{S_{\frac{\xi+\zeta}{2}}(\bar{w})}(\xi-\eta-\zeta)
\hat{S_{\frac{\xi+\zeta}{2}}(\bar{h})}(\eta)\hat{w}(\zeta)
\tfrac{4}{|\zeta+\xi|^2}\tilde{\chi}(\tfrac{\zeta+\xi}{2})
\chi_{\epsilon}\Big(\tfrac{|\xi-\zeta|}{\langle\xi+\zeta\rangle}\Big)\hat{w}(-\xi)\\
&\stackrel{\mathclap{\eqref{innercutoff}}}{=}
2\ii\tfrac{1}{(2\pi)^d}\sum_{\xi,\eta,\zeta\in\Z^d}\tilde{\chi}(\tfrac{\zeta+\xi}{2})\tfrac{1}{|\zeta+\xi|^2}
\chi_{\epsilon}\Big(\tfrac{|\xi-\zeta|}{\langle\xi+\zeta\rangle}\Big)
\chi_{\epsilon}\Big(\tfrac{2|\xi-\eta-\zeta|}{\langle\xi+\zeta\rangle}\Big)\chi_{\epsilon}\Big(\tfrac{2|\eta|}{\langle\xi+\zeta\rangle}\Big) 
\hat{\bar{w}}(\xi-\eta-\zeta)\hat{\bar{h}}(\eta)\hat{w}(\zeta)\hat{w}(-\xi)
\\&
=2\ii\tfrac{1}{(2\pi)^d} \sum_{\eta\in \Z^d} 
\hat{\bar{h}}(-\eta)\sum_{\xi,\zeta\in\Z^d}\tilde{\chi}(\tfrac{\zeta+\xi}{2}) \tfrac{1}{|\zeta+\xi|^2}
\chi_{\epsilon}\Big(\tfrac{|\xi-\zeta|}{\langle\xi+\zeta\rangle}\Big)
\\&\quad 
\times\chi_{\epsilon}\Big(\tfrac{2|\xi+\eta-\zeta|}{\langle\xi+\zeta\rangle}\Big)\chi_{\epsilon}\Big(\tfrac{2|\eta|}{\langle \xi+\zeta\rangle}\Big) 
\hat{\bar{w}}(\xi+\eta-\zeta)\hat{w}(\zeta)\hat{w}(-\xi)\,.
\end{align*}
Recalling \eqref{VecfieldHam} and the computations above, 
after some changes of variables in the summations, we obtain
\[
X_{\mathcal{B}_{\NLS}}(W)=\opbw(B_{\NLS}(S_{\x}W;x,\x))W+R_1(W)
\]
where the remainder $R_1(W)$ has the form  
$(R_{1}^{+}(W), \ov{R_{1}^{+}(W)})^{T}$ where
(recall \eqref{cutofffunct})
\[
\begin{aligned}
\widehat{(R_1^{+}(W))}(\x)&=\tfrac{1}{(2\pi)^d}\sum_{\eta,\zeta\in \mathbb{Z}^{d}}
r_1(\x,\eta,\zeta)
\hat{w}(\x-\eta-\zeta)\hat{\bar{w}}(\eta)\hat{w}(\zeta)\,, 
\qquad \x\in\mathbb{Z}^{d}\,,\\
r_1(\x,\eta,\zeta)&=-\tfrac{2}{|2\zeta-\x+\eta|^{2}}\tilde{\chi}\Big(\tfrac{2\zeta-\xi+\eta}{2}\Big)
\chi_{\epsilon}\Big(\tfrac{|\eta-\x|}{\langle2\zeta-\x+\eta\rangle}\Big)
\chi_{\epsilon}\Big(\tfrac{2|\x|}{\langle\x-\eta-2\zeta\rangle}\Big)
\chi_{\epsilon}\Big(\tfrac{2|\eta|}{\langle\x-\eta-2\zeta\rangle}\Big)\,.
\end{aligned}
\]
One can check, for $0<\epsilon<1$ small enough, 
$|\x|+|\eta|\ll |\x-\eta-\zeta|\sim |\zeta|$.
Therefore the coefficients $r_1(\x,\eta,\zeta)$ satisfies the 
\eqref{eq:paraproduct2}. Here we really need the truncation operator $S_\xi$: if you don't insert it in the definition of  $\mathcal{B}_{\NLS}$ (see \eqref{HamFUNC}) then $R_1$ is not a regularizing operator. Furthermore  this truncation does not affect the leading term: 
define the operator 
\[
R_{2}(W)=\left(\begin{matrix} R_{2}^{+}(W)  \vspace{0.2em} \\
\ov{R_{2}^{+}(W) }
\end{matrix}\right):=\opbw\Big(B_{\NLS}(S_{\x}W;x,\x)-B_{\NLS}(W;x,\x)\Big)W\,,
\]
we are going to prove that $R_2$ is also a regularizing operator.
By an explicit computation using \eqref{quantiWeyl}, \eqref{innercutoff}
and \eqref{gene1} one can check that
\[
\begin{aligned}
\widehat{(R_2^{+}(W))}(\x)&=\tfrac{1}{(2\pi)^d}\sum_{\eta,\zeta\in \mathbb{Z}^{d}}
r_2(\x,\eta,\zeta)
\hat{w}(\x-\eta-\zeta)\hat{\bar{w}}(\eta)\hat{w}(\zeta)\,,
\qquad \x\in\mathbb{Z}^{d}\,,\\
r_2(\x,\eta,\zeta)&=-
\tfrac{1}{|\x+\zeta|^{2}}\tilde{\chi}\Big(\tfrac{\xi+\zeta}{2}\Big)
\chi_{\epsilon}\Big(\tfrac{|\x-\zeta|}{\langle\x+\zeta\rangle}\Big)
\left(
1-\chi_{\epsilon}\Big(\tfrac{|\x-\eta-\zeta|}{\langle\x+\zeta\rangle}\Big)
\chi_{\epsilon}\Big(\tfrac{|\eta|}{\langle\x+\zeta\rangle}\Big)
\right)\,.
\end{aligned}
\]
We write $1\cdot r_{2}(\x,\eta,\zeta)$ and we use the partition of the 
 unity in \eqref{partUnity}. Hence 
using the \eqref{cutofffunct} one can check that each summand 
satisfies the bound in \eqref{eq:paraproduct2}.
Therefore the operator $Q_{G}:=R_1+R_2$ 
has the form \eqref{formaQG} 
and 
 \eqref{HamvecField}  is proved.
 The estimates \eqref{stimaQG} follow
by 
Lemma \ref{lem:trilineare}.
We note that 
\[
d_{W}\Big(\opbw(B_{\NLS}(W;x,\x))W\Big)[h]=
\opbw(B_{\NLS}(W;x,\x))h+\opbw(d_{W}B_{\NLS}(W;x,\x)[h])W\,.
\]
Then the estimates \eqref{stimaXG4} with $k=0,1$,  follow
by using \eqref{stimaQG}, the explicit formula of $B(W;x,\x)$ in 
\eqref{gene1} and Lemma \ref{azioneSimboo}.
Reasoning similarly one can prove the 
\eqref{stimaXG4} with $k=2,3$.
\end{proof}
 In the next proposition we define the changes of coordinates generated by the Hamiltonian vector fields $X_{\mathcal{B}_{\NLS}}$ and $X_{\mathcal{B}_{\KG}}$ and we study their properties as maps on Sobolev spaces.

\begin{proposition}\label{flussononlin}
For any $s\geq s_0>2d+2$ there is $r_0>0$ 
such that for $0\leq r\leq r_0$ and 
$W=\vect{w}{\bar{w}}\in 
B_{r}(H^{s}(\mathbb{T}^{d};\mathbb{C}^{2}))$, 
the following holds. 
Define
\begin{equation}\label{def:Zeta}
Z:=\Phi_{\mathcal{B}_{\star}}(W):=W+X_{\mathcal{B}_{\star}}(W)\,,
\end{equation}
where $\star\in\{ {\rm NLS}, {\rm KG}  \}$ (recall \eqref{HamFUNC}, \eqref{HamFUNCKG}).
Then one has
\begin{align}
\|Z\|_{H^{s}}&\leq \|w\|_{H^{s}}(1+C\|w\|^{2}_{H^{s}})\,,
\label{stimaflusso1}
\end{align}
for some $C>0$ depending on $s$,
and 
\begin{equation}\label{approxinvKG}
W=Z-X_{\mathcal{B}_{\star}}(Z)+r(w)\,,
\end{equation}
where 
\begin{equation}\label{approxinvKG2}
\|r(w)\|_{H^{s}}\lesssim \|w\|_{H^{s}}^{5}\,.
\end{equation}
\end{proposition}
\begin{proof}
By \eqref{def:Zeta} we can write
\[
W=Z-X_{\mathcal{B}_{\star}}(W)=Z-X_{\mathcal{B}_{\star}}(Z)
+\big[X_{\mathcal{B}_{\star}}(W)-X_{\mathcal{B}_{\star}}(Z)\big]\,.
\]
By using estimates \eqref{stimaXG4} or \eqref{stimaXG4KG} one can deduce that
$X_{\mathcal{B}_{\star}}(W)-X_{\mathcal{B}_{\star}}(Z)$
satisfies the bound \eqref{approxinvKG2}.
The bound \eqref{stimaflusso1} follows by Lemma \ref{HamvectorFieldG}.
\end{proof} 

\subsection{Conjugations}
\noindent Recalling \eqref{pollini1} and \eqref{hamiltonianaS}  we
set 
\begin{equation}\label{Hamprova}
\mathcal{H}_{\NLS}^{(\leq4)}(W):=\mathcal{H}_{\NLS}^{(2)}(W)
+\mathcal{H}_{\NLS}^{(4)}(W)\,, \quad \mathcal{H}_{\NLS}^{(2)}(Z):=\int_{\mathbb{T}^{d}}\Lambda_{\NLS} z\cdot\bar{z}dx.
\end{equation}
Analogously, recalling \eqref{calGGGKG} and \eqref{def:Lambda2},
 we set
  \begin{equation}\label{HamprovaKG}
 \mathcal{H}^{(\leq4)}_{\KG}(W)
 :=\mathcal{H}^{(2)}_{\KG}(W)+\mathcal{H}^{(4)}_{\KG}(W)\,,\quad \mathcal{H}^{(2)}_{\KG}(Z):=\int_{\mathbb{T}^{d}}\Lambda_{\KG} z\cdot\bar{z}dx
 \end{equation}
In the following lemma we study how the Hamiltonian vector fields $X_{\mathcal{H}^{(\leq4)}_{\NLS}}(W)$ in \eqref{Hamprova}, and $X_{\mathcal{H}^{(\leq4)}_{\KG}}(W)$ in \eqref{HamprovaKG}, transform under the change of variables given by the previous lemma.

\begin{lemma}\label{conjconjconj}
Let $s_0>2d+4$. Then for any $s\geq s_0$ there is $r_0>0$
such that for all $0<r\leq r_0$ and  $Z=\vect{z}{\bar{z}}\in 
B_{r}(H^{s}(\mathbb{T}^{d};\mathbb{C}^{2}))$
the following holds. Consider the Hamiltonian $\mathcal{B}_{\star}$ 
with $\star\in\{ {\rm NLS}, {\rm KG}  \}$ 
(recall \eqref{HamFUNC}, \eqref{HamFUNCKG})
and the Hamiltonian $\mathcal{H}_{\star}^{(\leq4)}$ 
(see \eqref{Hamprova}, \eqref{HamprovaKG}).
Then
\begin{equation}\label{PPPPtau2}
d_{W}\Phi_{\mathcal{B}_{\star}}(W)\big[X_{\mathcal{H}_{\star}^{(\leq4)}}(W)
\big]=
X_{\mathcal{H}_{\star}^{(\leq4)}}(Z)+\big[X_{\mathcal{B}_{\star}}(Z),
X_{\mathcal{H}_{\star}^{(2)}}(Z)\big]+R_{5}(Z)\,,
\end{equation}
%
where
the remainder $R_{5}$ satisfies
\begin{equation}\label{PPPPtau3}
\|R_{5}(Z)\|_{H^{s}}\lesssim \|z\|_{H^{s}}^{5}\,,
\end{equation}
and $[\cdot,\cdot]$ is the nonlinear commutator defined 
in \eqref{nonlinCommu}.
\end{lemma}

\begin{proof}
We prove the statement in the case $\mathcal{B}_{\star}=\mathcal{B}_{\NLS}$ and 
$\mathcal{H}_{\star}^{(\leq4)}=\mathcal{H}^{(\leq4)}_{\NLS}$, the KG-case is similar.
One can check that \eqref{PPPPtau2} follows by setting 
\begin{align}
R_{5}&:=d_{W}X_{\mathcal{B}_{\NLS}}(W)\big[X_{\mathcal{H}_{\NLS}^{(\leq4)}}(W)
-X_{\mathcal{H}_{\NLS}^{(\leq4)}}(Z)\big]
\label{gordon10NLS}\\
&+\big(d_{W}X_{\mathcal{B}_{\NLS}}(W)
-d_{W}X_{\mathcal{B}_{\NLS}}(Z)\big)
\big[X_{\mathcal{H}_{\NLS}^{(\leq4)}}(Z)\big]\label{gordon11NLS}\\
&+X_{\mathcal{H}_{\NLS}^{(\leq4)}}(W)
-X_{\mathcal{H}_{\NLS}^{(\leq4)}}(Z)+d_{W}X_{\mathcal{H}_{\NLS}^{(\leq4)}}(Z)
\big[X_{\mathcal{B}_{\NLS}}(Z)\big]\,\label{gordon12NLS}\\
&+\big[
X_{\mathcal{B}_{\NLS}}(Z)
,X_{\mathcal{H}^{(4)}_{\NLS}}(Z)
\big]\,.
\label{gordon12bisNLS}
\end{align}
We are left to prove that 
 $R_5$ satisfies  \eqref{PPPPtau3}.
We start from the term in \eqref{gordon10NLS}.
First of all we note that
\[
X_{\mathcal{H}_{\NLS}^{(\leq4)}}(W)-X_{\mathcal{H}_{\NLS}^{(\leq4)}}(Z)
=-\ii E\Lambda_{\NLS}(W-Z)+X_{\mathcal{H}_{\NLS}^{(4)}}(W)-
X_{\mathcal{H}_{\NLS}^{(\leq4)}}(Z)\,,
\]
where we used that $X_{\mathcal{H}_{\NLS}^{(2)}}(W)=-\ii E\Lambda_{\NLS} W$.
By Proposition  \ref{flussononlin}, the \eqref{hamiltonianaS}
and \eqref{stimaXG4}
we deduce that
\[
\|X_{\mathcal{H}_{\NLS}^{(\leq4)}}(W)
-X_{\mathcal{H}_{\NLS}^{(\leq4)}}(Z)\|_{H^{s}}\lesssim\|w\|_{H^{s}}^{3}\,.
\]
Hence using again the bounds \eqref{stimaXG4} we obtain
\[
\|d_{W}X_{\mathcal{B}_{\NLS}}(W)\big[X_{\mathcal{H}_{\NLS}^{(\leq4)}}(W)
-X_{\mathcal{H}_{\NLS}^{(\leq)}}(Z)\big]\|_{H^{s}}
\lesssim \|w\|^{5}_{H^{s}}\,.
\]
Reasoning in the same way, using also \eqref{approxinvKG},  one can check that the terms in 
\eqref{gordon11NLS}, \eqref{gordon12NLS}, \eqref{gordon12bisNLS}
satisfies the same quintic estimates. 
\end{proof}

In the next lemma we study the structure of the the cubic terms in the vector field in \eqref{PPPPtau2} in the NLS case.

\begin{lemma}\label{conj:hamfield}
Consider the Hamiltonian
$\mathcal{B}_{\NLS}(W)$ in \eqref{HamFUNC} 
and recall \eqref{hamiltonianaS}, \eqref{Hamprova}. 
Then we have that
\begin{equation}\label{cubici}
X_{\mathcal{H}^{(4)}_{\NLS}}(Z)
+\big[X_{\mathcal{B}_{\NLS}}(Z),X_{\mathcal{H}^{(2)}_{\NLS}}(Z)\big]
=-\ii E\opbw\left(\begin{matrix}
2|z|^{2} & 0 \\0 & 2|z|^{2}
\end{matrix}\right)Z+Q_{\mathtt{H}^{(4)}_{\NLS}}(Z)\,,
\end{equation}
where the remainder $Q_{\mathtt{H}^{(4)}_{\NLS}}$ has the form 
$Q_{\mathtt{H}^{(4)}_{\NLS}}(Z)=
(Q_{\mathtt{H}^{(4)}_{\NLS}}^{+}(Z),
\ov{Q_{\mathtt{H}^{(4)}_{\NLS}}^{+}(Z)})^T$ and 
\begin{equation}\label{formaRRRi}
 \widehat{(Q_{\mathtt{H}^{(4)}_{\NLS}}^{+}(Z))}(\x)
 =\tfrac{1}{(2\pi)^d}\sum_{\eta,\zeta\in\mathbb{Z}^{d}}
\mathtt{q}_{\mathtt{H}^{(4)}_{\NLS}}(\x,\eta,\zeta)
\hat{z}(\x-\eta-\zeta)\hat{\bar{z}}(\eta)\hat{z}(\zeta)\,,
\qquad \x\in \mathbb{Z}^{d} \,,
\end{equation}
with symbol satisfying
\begin{align}
|\mathtt{q}_{\mathtt{H}^{(4)}_{\NLS}}(\x,\eta,\zeta)|&\lesssim 
\frac{\max_{2}\{\langle\x-\eta-\zeta\rangle,\langle\eta\rangle,\langle\zeta\rangle\}^{4}}
{\max_1\{\langle\x-\eta-\zeta\rangle,\langle\eta\rangle,\langle\zeta\rangle\}^{2}}\,.
\label{pollo2}
\end{align}
\end{lemma}
\begin{proof}
We start by considering the commutator between 
$X_{\mathcal{B}_{\NLS}}$ and $X_{\mathcal{H}^{(2)}_{\NLS}}$.
First of all notice that 
(see \eqref{HamvecField}, \eqref{gene1})
\[
X_{\mathcal{B}_{\NLS}}(Z)=\left(
\begin{matrix}
X_{\mathcal{B}_{\NLS}}^{+}(Z) \vspace{0.2em}\\
\ov{X_{\mathcal{B}_{\NLS}}^{+}(Z)}
\end{matrix}\right)\,,
\qquad X_{\mathcal{B}_{\NLS}}^{+}(Z):=
\opbw\Big(\tfrac{z^2}{2|\x|^{2}}\tilde{\chi}(\xi)\Big)[\bar{z}]+Q_{\mathcal{B}_{\NLS}}^{+}(Z)\,,
\]
and hence (recall \eqref{quantiWeyl}), for $\x\in \mathbb{Z}^{d}$
\begin{equation}\label{pollini3}
\widehat{(X_{\mathcal{B}_{\NLS}}^{+}(Z))}(\x)=\tfrac{1}{(2\pi)^d}
\sum_{\eta,\zeta\in\mathbb{Z}^{d}}
\hat{z}(\x-\eta-\zeta)\hat{\bar{z}}(\eta)\hat{z}(\zeta)
\Big[\tfrac{2}{|\x+\eta|^{2}}
\tilde{\chi}\Big(\tfrac{\xi+\eta}{2}\Big)\chi_{\epsilon}\Big(\tfrac{|\x-\eta|}{\langle\x+\eta\rangle}\Big)
+q_{\mathcal{B}_{\NLS}}(\x,\eta,\zeta)\Big]
\end{equation}
where $q_{\mathcal{B}_{\NLS}}(\z,\eta,\zeta)$ 
satisfies the bound in \eqref{eq:paraproduct2}. 
Hence, by using formul\ae\, \eqref{pollini1}, 
 \eqref{pollini3},
\eqref{nonlinCommu}, one obtains 
\begin{equation*}
\begin{aligned}
&X_{\mathcal{H}^{(4)}_{\NLS}}(Z)+\big[X_{\mathcal{B}_{\NLS}}(Z),X_{\mathcal{H}^{(2)}_{\NLS}}(Z)\big]=
\left(
\begin{matrix}
\mathcal{C}^{+}(Z) \vspace{0.2em}\\
\ov{\mathcal{C}^{+}(Z)}
\end{matrix}\right)\,,\\
&\widehat{(\mathcal{C}^{+}(Z))}(\x)=
\tfrac{-1}{(2\pi)^d}\sum_{\eta,\zeta\in\mathbb{Z}^{d}}
\ii \mathtt{c}(\x,\eta,\zeta)
\hat{z}(\x-\eta-\zeta)\hat{\bar{z}}(\eta)\hat{z}(\zeta)
\end{aligned}
\end{equation*}
where
\begin{equation}\label{pollini4}
\mathtt{c}(\x,\eta,\zeta)=1+
\Big[\tfrac{2}{|\x+\eta|^{2}}\tilde{\chi}\Big(\tfrac{\xi+\eta}{2}\Big)
\chi_{\epsilon}\Big(\tfrac{|\x-\eta|}{\langle\x+\eta\rangle}\Big)
+q_{\mathcal{B}_{\NLS}}(\x,\eta,\zeta)\Big]
\Big[\Lambda_{\NLS}(\x-\eta-\zeta)-\Lambda_{\NLS}(\eta)
+\Lambda_{\NLS}(\zeta)-\Lambda_{\NLS}(\x)\Big]\,.
\end{equation}
We need to prove that this can be written 
as the r.h.s. of \eqref{cubici}.
First we note that the term in \eqref{pollini4} 
\begin{equation}\label{pollini5}
q_{\mathcal{B}_{\NLS}}(\x,\eta,\zeta)\Big[\Lambda_{\NLS}(\x-\eta-\zeta)-\Lambda_{\NLS}(\eta)
+\Lambda_{\NLS}(\zeta)-\Lambda_{\NLS}(\x)\Big]
\end{equation}
can be absorbed in $\mathtt{R}_1$ since 
the \eqref{pollini5} satisfy the same bound as in \eqref{pollo2}.
Moreover, using the \eqref{pollini1} and the \eqref{insPot}, we have that
the coefficients
\[
\tfrac{2}{|\x+\eta|^{2}}\tilde{\chi}\Big(\tfrac{\xi+\eta}{2}\Big)\chi_{\epsilon}\Big(\tfrac{|\x-\eta|}{\langle\x+\eta\rangle}\Big)
\Big[\hat{V}(\x-\eta-\zeta)-\hat{V}(\eta)+\hat{V}(\zeta)-\hat{V}(\x)\Big]
\]
satisfy the bound in \eqref{pollo2} by using also Lemma \ref{lem:paratri}.
 Therefore the corresponding 
operator contributes to
$\mathtt{R}_{1}$.
The same holds for the operator corresponding to the coefficients
\[
\tfrac{2}{|\x+\eta|^{2}}\tilde{\chi}\Big(\tfrac{\xi+\eta}{2}\Big)\chi_{\epsilon}\Big(\tfrac{|\x-\eta|}{\langle\x+\eta\rangle}\Big)
\Big[|\x-\eta-\zeta|^{2}+|\zeta|^{2}\Big]\,.
\] 
We are left with the 
most relevant terms in \eqref{pollini4} containing 
the highest frequencies $\eta$ and $\x$.
We have that
\[
\tfrac{-2(|\x|^{2}+|\eta|^{2})}{|\x+\eta|^{2}}
\chi_{\epsilon}\Big(\tfrac{|\x-\eta|}{\langle\x+\eta\rangle}\Big)\tilde{\chi}\Big(\tfrac{\xi+\eta}{2}\Big)
=
-\chi_{\epsilon}\Big(\tfrac{|\x-\eta|}{\langle\x+\eta\rangle}\Big)
-r_{1}(\x,\eta,\zeta)\,,
\]
where
\[
r_{1}(\x,\eta,\zeta)=
\Big(\tilde{\chi}\Big(\tfrac{\xi+\eta}{2}\Big)-1\Big)\chi_{\epsilon}\Big(\tfrac{|\x-\eta|}{\langle\x+\eta\rangle}\Big) +
\tfrac{|\x-\eta|^{2}}{|\x+\eta|^{2}}\tilde{\chi}\Big(\tfrac{\xi+\eta}{2}\Big)\chi_{\epsilon}
\Big(\tfrac{|\x-\eta|}{\langle\x+\eta\rangle}\Big)\,.
\]
Again we note that the coefficients $r_{1}(\x,\eta,\zeta)$, using Lemma \ref{lem:paratri}
and the definition of $\tilde{\chi}$ below \eqref{gene1}, 
satisfy \eqref{pollo2}.
Then it remains to 
study
the operator
$\mathcal{R}^{+}(Z)$ with
\[\quad
\widehat{(\mathcal{R}^{+}(Z))}(\x):=\tfrac{-1}{(2\pi)^d}
\sum_{\eta,\zeta\in\mathbb{Z}^{d}}
\ii \Big(1-\chi_{\epsilon}\Big(\tfrac{|\x-\eta|}{\langle\x+\eta\rangle}\Big)\Big)
\hat{z}(\x-\eta-\zeta)\hat{\bar{z}}(\eta)\hat{z}(\zeta)\,.
\]
By formula \eqref{X_H} and \eqref{quantiWeyl} we get
$
\mathcal{R}^{+}(Z)=-\ii\opbw(2|z|^{2})z+Q^{+}_{3}(U)\,,
$
where $Q_{3}$ satisfies \eqref{restoQ3KGNLS}, \eqref{restoQ32KGNLS}.
This concludes the proof.
\end{proof}

In the next lemma we study the structure of the the cubic terms in the vector field in \eqref{PPPPtau2} in the KG case.

\begin{lemma}\label{conj:hamfieldKG}
Consider the Hamiltonian
$\mathcal{B}_{\KG}(W)$ in \eqref{HamFUNCKG} 
and recall \eqref{calGGGKG}, \eqref{HamprovaKG}. 
Then we have that

\begin{equation}\label{cubiciKG}
X_{\mathcal{H}^{(4)}_{\KG}}(Z)
+\big[X_{\mathcal{B}_{\KG}}(Z),X_{\mathcal{H}^{(2)}_{\KG}}(Z)\big]=
-\ii E\opbw\big(
\diag(a_0(x,\x))\big)Z+Q_{\mathtt{H}^{(4)}_{\KG}}(Z)
\end{equation}
the symbol $a_0(x,\x)=a_0(u,x,\x)$ is in  \eqref{simboa2KG}, the remainder
$Q_{\mathtt{H}^{(4)}_{\KG}}(Z)$ has the form
$(Q^{+}_{\mathtt{H}^{(4)}_{\KG}}(Z),\ov{Q^{+}_{\mathtt{H}^{(4)}_{\KG}}(Z)})^{T}$
with
\begin{equation}\label{restoQ3GGGKG}
\widehat{Q^+_{\mathtt{H}^{(4)}_{\KG}}}(\x)=(2\pi)^{-d}\sum_{\substack{\s_1,\s_2,\s_{3}\in\{\pm\}\\
\eta,\zeta\in\mathbb{Z}^{d}}}\!\!
\mathtt{q}_{\mathtt{H}^{(4)}_{\KG}}^{\s_1,\s_2,\s_3}(\x,\eta,\zeta)
\hat{z^{\s_1}}(\x-\eta-\zeta)\hat{z^{\s_2}}(\eta)\hat{z^{\s_3}}(\zeta)\,,
\end{equation}
for some 
$\mathtt{q}_{\mathtt{H}^{(4)}_{\KG}}^{\s_1,\s_2,\s_3}(\x,\eta,\zeta)\in \mathbb{C}$
satisfying 
\begin{equation}\label{restoQ32GGGKG}
|\mathtt{q}_{\mathtt{H}^{(4)}_{\KG}}^{\s_1,\s_2,\s_3}(\x,\eta,\zeta)|
\lesssim\frac{\max_2\{\langle\x-\eta-\zeta\rangle,\langle \eta\rangle,\langle\zeta\rangle\}^{\mu}}{\max\{\langle\x-\eta-\zeta\rangle,\langle \eta\rangle,\langle\zeta\rangle\}}
\end{equation}
for some $\mu>1$.
\end{lemma}

\begin{proof}
Using \eqref{HamvecField}  (with $\mathcal{B}=\mathcal{B}_{\KG}$)
we can note that 
\begin{equation}\label{gothamKG55}
\big[X_{\mathcal{B}_{\KG}}(Z),X_{\mathcal{H}^{(2)}_{\KG}}(Z)\big]=\big[
\opbw(B_{\KG}(Z;x,\x)),
X_{\mathcal{H}^{(2)}_{\KG}}(Z)
\big]+R_2(Z)
\end{equation}
where $R_2(Z)=(R^{+}_2(Z),\ov{R^{+}_2(Z)})^{T}$ with 
 
\begin{equation}\label{formaRrestoKG}
\begin{aligned}
\widehat{(R_2^{+}(Z))}(\x)&=(2\pi)^{-d}\sum_{\substack{ \s_1,\s_2,\s_3\in\{\pm\}\\ 
\eta,\zeta\in\mathbb{Z}^{d}}}
\mathtt{r}_2^{\s_1,\s_2,\s_3}(\x,\eta,\zeta)\hat{z^{\s_1}}(\x-\eta-\zeta)\hat{z^{\s_2}}(\eta)\hat{z^{\s_3}}(\zeta)\,,\quad 
\x\in \mathbb{Z}^{d} \,,\\
\mathtt{r}_2^{\s_1,\s_2,\s_3}(\x,\eta,\zeta)&:=
\mathtt{q}_{\mathcal{B}_{\KG}}^{\s_1,\s_2,\s_3}(\x,\eta,\zeta)\Big[\s_1\Lambda_{\KG}(\x-\eta-\zeta)
+\s_2\Lambda_{\KG}(\eta)+\s_3\Lambda_{\KG}(\zeta)-\Lambda_{\KG}(\x)\Big],
\end{aligned} 
\end{equation}
where the coefficients are defined in \eqref{formaQG}.
The remainder $R_2$ has the form \eqref{restoQ3GGGKG} and
we have that the coefficients
$\mathtt{r}_2^{\s_1,\s_2,\s_3}(\x,\eta,\zeta)$ satisfy the bound \eqref{restoQ32GGGKG}.
On the other hand, recalling \eqref{gene1KG}, \eqref{nonlinCommu}, we have
\begin{equation}\label{gothamKG555}
\big[
\opbw(B_{\KG}(Z;x,\x)),
X_{\mathcal{H}^{(2)}_{\KG}}(Z)
\big]=R_3(Z)+R_4(Z)\,,\qquad R_{j}(Z)=
\left(\begin{matrix}
R^{+}_{j}(Z)\vspace{0.2em}\\
\ov{R^{+}_{j}(Z)}
\end{matrix}\right)\,,\;\; j=3,4\,,
\end{equation}
where
\begin{align}
R^+_3(Z)&:=\opbw(b_{\KG}(Z;x,\x))[\ii \Lambda_{\KG} \bar{z}]
+\ii \Lambda_{\KG}\opbw(b_{\KG}(Z;x,\x))[\bar{z}]\,,
\label{RRR333}\\
R^{+}_{4}(Z)&:=\opbw\big((d_{Z}b_{\KG})(z;x,\x)[X_{\mathcal{H}^{(2)}_{\KG}}(Z)]\big)[\bar{z}]\,.\label{RRR444}
\end{align}
By Remark \ref{strutturaA0KG} and \eqref{quantiWeyl} we get
\[
\begin{aligned}
\widehat{R_{4}^{+}}(\x)=(2\pi)^{-d}\sum_{\substack{\s_1,\s_2\in\{\pm\}\\
\eta,\zeta\in \mathbb{Z}^{d}}}
&a_{0}^{\s_1,\s_2}(\x-\zeta,\eta,\tfrac{\x+\zeta}{2})\frac{1}{2\Lambda_{\KG}(\frac{\x+\zeta}{2})}
\chi_{\epsilon}\Big(\tfrac{|\x-\z|}{\langle \x+\zeta\rangle}\Big)\Big[ 
-\ii\s_1\Lambda_{\KG}(\x-\eta-\zeta)-\ii \s_2\Lambda_{\KG}(\eta)
\Big]\times\\
&\qquad\qquad \times
\hat{z^{\s_1}}(\x-\eta-\zeta)\hat{z^{\s_2}}(\eta)\hat{\bar{z}}(\zeta)\,.
\end{aligned}
\]
Using the explicit form of the coefficients of $R_{4}^{+}$
and Lemma \ref{lem:paratri} one can conclude that
the operator $R_{4}^{+}$ has the form \eqref{restoQ3GGGKG}
with coefficients satisfying \eqref{restoQ32GGGKG}.
To summarize, by 
\eqref{gothamKG55} and \eqref{gothamKG555},
we have obtained
(recall also \eqref{calGGG1KG}, \eqref{calAA0KG})
\begin{equation}\label{orangeKG2}
{\rm l.h.s.\; of}\; \eqref{cubiciKG}
=\opbw\left(
\begin{matrix}
-\ii a_{0}(x,\x) & 0\\0&
\ii a_0(x,\x)
\end{matrix}
\right)Z+F_{3}(Z)+Q_{3}(Z)+R_{2}(Z)+R_{4}(Z)
\end{equation}
where $R_4$ is in \eqref{RRR444}, $R_2$ is in \eqref{formaRrestoKG},
$Q_{3}(Z)$ is in \eqref{calGGG1KG}
and
\begin{equation}\label{finalFFFKG3}
F_{3}(Z)=\left(
\begin{matrix}
F_{3}^{+}(Z) \vspace{0.2em}\\
\ov{F_{3}^{+}(Z)}
\end{matrix}
\right)\,,\qquad
F_{3}^{+}(Z)=-\ii\opbw(a_{0}(x,\x))[\bar{z}]+R_{3}^{+}(Z)
\end{equation}
where $R_{3}^{+}$ is in \eqref{RRR333}. By the discussion above
and by Lemma \ref{calGvecKG}
we have that the remainders $R_2$, $R_{4}$ and $Q_{3}$ have the form
\eqref{restoQ3GGGKG} with coefficients satisfying \eqref{restoQ32GGGKG}.
To conclude the prove we need to show that $F_{3}$ has the same property.
This will be a consequence of the choice of the symbol 
$b_{\KG}(W;x,\x)$ in \eqref{gene1KG}.
Indeed, by \eqref{gene1KG}, Remark \ref{strutturaA0KG}, \eqref{finalFFFKG3}, 
\eqref{RRR333}, 
we have
\[
\widehat{F_{3}^{+}}(\x)=(2\pi)^{-d}
\sum_{\substack{\s_1,\s_2\in\{\pm\}\\
\eta,\zeta\in \mathbb{Z}^{d}}}
\mathtt{f}_{3}^{\s_1,\s_2,-}(\x,\eta,\zeta)\hat{z^{\s_1}}(\x-\eta-\zeta)\hat{z^{\s_2}}(\eta)
\hat{\bar{z}}(\zeta)
\]
where
\begin{equation}\label{orangeKG1}
\mathtt{f}_{3}^{\s_1,\s_2,-}(\x,\eta,\zeta):=
a_{0}^{\s_1,\s_2}(\x-\zeta,\eta,\tfrac{\x+\zeta}{2})\ii \Big[\tfrac{\Lambda_{\KG}(\x)+\Lambda_{\KG}(\zeta)}{2\Lambda_{\KG}(\frac{\x+\zeta}{2})}-1\Big]\chi_{\epsilon}
\Big(\tfrac{|\x-\zeta|}{\langle\z+\zeta\rangle}\Big)\,.
\end{equation}
By Taylor expanding the symbol $\Lambda_{\KG}(\x)$ in \eqref{def:Lambda2} 
(see also Remark \ref{strutturaA0KG}) one deduces 
that
\[
\left|
a_{0}^{\s_1,\s_2}(\x-\zeta,\eta,\tfrac{\x+\zeta}{2})\ii \Big[\tfrac{\Lambda_{\KG}(\x)+\Lambda_{\KG}(\zeta)}{2\Lambda_{\KG}(\frac{\x+\zeta}{2})}-1\Big]
\right|\lesssim\tfrac{|\x-\zeta|}{(\langle\x\rangle+\langle\zeta\rangle)^{3/2}}\,.
\]
Therefore, using Lemma \ref{lem:paratri}, we have that the coefficients
$\mathtt{f}_{3}^{\s_1,\s_2,-}(\x,\eta,\zeta)$ in \eqref{orangeKG1}
satisfy the \eqref{restoQ32GGGKG}. 
This implies the \eqref{cubiciKG}.\end{proof}

\section{Diagonalization}\label{diago}

\subsection{Diagonalization of the NLS}
In this section we   diagonalize 
the system \eqref{QLNLS444}. We first diagonalize 
the matrix $E(\uno+A_2(x))$ in \eqref{QLNLS444} 
by means of a change of coordinates as the 
ones made in the papers 
\cite{Feola-Iandoli-local-tori, Feola-Iandoli-Long}. 
After that we diagonalize the matrix of symbols of 
order $0$ at homogeneity $3$, by means of an approximatively  \emph{symplectic} 
change of coordinates. 
Throughout the rest of the section we shall assume the following.
\begin{hypo}\label{2epsilon}
We restrict the solution of \eqref{NLS} on the interval of times $[0,T)$, 
with $T$ such that 
\begin{equation*}
\sup_{t\in[0,T)}\|u(t,x)\|_{H^s}\leq \e\,, \quad \|u_0(x)\|_{H^s}\leq\tfrac{1}{4} \e\,.
\end{equation*}
\end{hypo}
\noindent Note that such a time $T>0$ exists thanks to the 
local existence theorem in \cite{Feola-Iandoli-local-tori}.

\subsubsection{Diagonalization at order 2}\label{blockdiago2}
We consider the matrix $E(\uno+A_2(x))$ in \eqref{QLNLS444}. 
We define
\begin{equation}\label{autovalori}
\lambda_{\NLS}(x):=\lambda_{\NLS}(U;x):=\sqrt{1+2|u|^2[h'(|u|^2)]^2}\,, 
\qquad a_2^{(1)}(x):=\lambda_{\NLS}(x)-1\,,
\end{equation}
and we note that  $\pm\lambda_{\NLS}(x)$ are the eigenvalues 
of the matrix $E(\uno+{A}_2(x))$. 
We denote by $S$  matrix of the eigenvectors of $E(\uno+{A}_2(x))$, more explicitly
\begin{equation}\label{matriceS}
\begin{aligned}
&S=\left(\begin{matrix}
s_1 &s_2\\
\bar{s}_2 & s_1
\end{matrix}\right)\,, 
\qquad S^{-1}=\left(\begin{matrix}
s_1 &-s_2\\
-\bar{s}_2 & s_1
\end{matrix}\right)\,, \\
&s_1(x):=\frac{1+|u|^2[h'(|u|^2)]^2+\lambda_{\NLS}(x)}
{\sqrt{2\lambda_{\NLS}(x)(1+[h'(|u|^2)]^2|u|^2+\lambda_{\NLS}(x))}}\,, \\
&s_2(x):=\frac{-u^2[h'(|u|^2)]^2}{\sqrt{2\lambda_{\NLS}(x)(1+[h'(|u|^2)]^2|u|^2+\lambda_{\NLS}(x))}}\,.
\end{aligned}
\end{equation}
Since $\pm\lambda_{\NLS}(x)$ are the eigenvalues 
and $S(x)$ is the matrix  eigenvectors 
of $E(\uno+{A}_2(x))$ we have that 
\begin{equation}\label{diago_algebra}
S^{-1}E(\uno+{A}_2(x))S=E\diag(\lambda_{\NLS} (x))\,,\quad s_1^2-|s_2|^2=1\,,
\end{equation}
where we have used the notation \eqref{matriciozze}.
In the following lemma we estimate the semi-norms 
of the symbols defined above.
\begin{lemma}\label{stime_diago_max}
Let $\N\ni s_0> d$. 
The symbols $a_2^{(1)}$ defined in \eqref{autovalori}, 
$s_1-1$ and $s_2$ defined in \eqref{matriceS} 
satisfy the following estimate
\begin{equation*}
|a^{(1)}_2|_{\mathcal{N}_p^0}
+|s_1-1|_{\mathcal{N}_p^0}
+|s_2|_{\mathcal{N}_p^0}\lesssim \|u\|^{6}_{H^{p+s_0}}\,, 
\qquad p+s_0\leq s\,,\quad p\in \mathbb{N}\,.
\end{equation*}
\end{lemma}

\begin{proof}
The proof follows by using the estimate \eqref{realtaAAA2} 
on the symbols in \eqref{simboa2}, 
the fact that $h'(s)\sim s$ when $s\sim 0$, $\|u\|_s\ll 1$, 
and the explicit expression \eqref{autovalori}, \eqref{matriceS}.
\end{proof}

\noindent
We now study how the system \eqref{QLNLS444} transforms under the maps 
\begin{equation}\label{diago_para}
\begin{aligned}
\Phi_{\NLS}:=&\,\Phi_{\NLS}(U):=\opbw(S^{-1}(U;x))\,,\qquad 
\Psi_{\NLS}:=\,\Psi_{\NLS}(U):=\opbw(S(U;x))\,.
\end{aligned}
\end{equation}

\begin{lemma}\label{propMAPPA}
Let $U=\vect{u}{\bar{u}}$ be a solution of \eqref{QLNLS444} 
and assume Hyp. \ref{2epsilon}. 
Then for any $s\geq 2s_0+2$, $\N\ni s_0>d$,
we have the following.

\noindent
$(i)$ One has the upper bound 
\begin{equation}\label{stime-descentTOTA}
\begin{aligned}
\|\Phi_{\NLS}(U)W\|_{H^{s}}+\|\Psi_{\NLS}(U)W\|_{H^{s}}&\leq 
\|{W}\|_{H^{s}}\big(1+C\|u\|^{6}_{H^{2s_0}}\big)\,,\\
\|(\Phi_{\NLS}(U)-\uno)W\|_{H^{s}}+\|(\Psi_{\NLS}(U)-\uno)W\|_{H^{s}}&\lesssim
\|{W}\|_{H^{s}}\|u\|^{6}_{H^{2s_0}}\,,
\qquad \forall\, W\in H^{s}(\T^d;\mathbb{C})\,,
\end{aligned}
\end{equation} 
where the constant $C$ depends on $s$;

\noindent
$(ii)$ one has $\Psi_{\NLS}(U)\circ\Phi_{\NLS}(U)=\uno+R(u)$
where $R$ is a real-to-real remainder of the form \eqref{barrato4}
satisfying 
\begin{equation}\label{achille10}
\|R(u)W\|_{H^{s+2}}\lesssim  \|W\|_{H^{s}}\|u\|^{6}_{H^{2s_0+2}}\,.
\end{equation}
The map $\uno+R(u)$ is invertible with inverse 
$(\uno+R(u))^{-1}:=(1+\tilde{R}(u))$ with 
$\tilde{R}(u)$ of the form \eqref{barrato4} and
\begin{equation}\label{stima-R-tilde}
\|\tilde{R}(u)W\|_{H^{s+2}}\lesssim \|W\|_{H^s}\|u\|_{H^{2s_0+2}}^{6}\,,
\end{equation}
as a consequence the map $\Phi_{\NLS}$ 
is invertible and $\Phi_{\NLS}^{-1}=(1+\tilde{R})\Psi_{\NLS}$ with estimates
\begin{equation}\label{inv-phi}
\|\Phi_{\NLS}^{-1}(U)W\|_{H^s}\leq \|W\|_{H^s}(1+C\|u\|_{H^{2s_0+2}}^{6})\,,
\end{equation}
where the constant $C$ depends on $s$;

\noindent
$(iii)$ for any $t\in [0,T)$,
one has 
$\pa_{t}\Phi_{\NLS}(U)[\cdot]=\opbw(\pa_{t}S^{-1}(U;x))$
and
\begin{equation}\label{achille11}
|\pa_{t}S^{-1}(U;x)|_{\mathcal{N}_{s_0}^{0}}\lesssim
 \|u\|^{6}_{H^{2s_0+2}}\,,
\qquad
\|\pa_{t}\Phi_{\NLS}(U)V\|_{H^{s}}\lesssim  \|W\|_{H^{s}}\|u\|^{6}_{H^{2s_0+2}}\,.
\end{equation}
\end{lemma}

\begin{proof}
$(i)$ The bounds \eqref{stime-descentTOTA} follow by 
\eqref{actionSob} and Lemma \ref{stime_diago_max}.

\noindent
$(ii)$ We apply Proposition \ref{prop:compo} 
to the maps in  \eqref{diago_para},
in particular the first part of the item follows 
by using the expansion \eqref{composit3} 
and recalling that symbols $s_1(x)$ and $s_2(x)$  
do not depend on $\xi$.  
The \eqref{stima-R-tilde} is obtained by Neumann series by using that  
(see Hyp. \ref{2epsilon}) $\|u\|_{H^s}\ll 1$.

\noindent
$(iii)$ We note that
 $
 \pa_{t}s_1(x,\x)=(\pa_{u}s_1)(u;x,\x)[\dot{u}]+
 (\pa_{\bar{u}}s_1)(u;x,\x)[\dot{\bar{u}}]\,.
 $
 Since $u$ solves \eqref{QLNLS444} and satisfies Hypothesis \ref{2epsilon}, 
 then using Lemma \ref{azioneSimboo} and \eqref{stimaRRR} 
 we deduce that $\|\dot{u}\|_{H^{s}}\lesssim \|u\|_{H^{s+2}} $.
 Hence the estimates \eqref{achille11} follow by direct inspection 
by using the explicit structure of the 
symbols $s_1,$ $s_2$ in \eqref{matriceS}, Lemma \ref{lem:nonomosimbo}
and \eqref{actionSob}.
\end{proof}

We are now in position to state the following proposition.

\begin{proposition}[{\bf Diagonalization at order $2$}]\label{diago2max}
Consider the system \eqref{QLNLS444} and set 
\begin{equation}\label{cambio1}
W=\Phi_{\NLS}(U)U\,,
\end{equation}
with $\Phi_{\NLS}$ defined in \eqref{diago_para}. 
Then $W$ solves the equation
\begin{equation}\label{QLNLS444KK}
\begin{aligned}
\dot{W}=&-\ii E\opbw\big(\diag\big(1+a_2^{(1)}(U;x)\big)|\x|^{2})W-\ii EV*W\\
&\quad-\ii\opbw\big(\diag\big({\vec{a}_1^{(1)}(U;x)\cdot\xi}\big)\big)W+
X_{\mathcal{H}^{(4)}_{\NLS}}(W)+R^{(1)}(U)\,,
\end{aligned}\end{equation}
where  the vector field $X_{\mathcal{H}^{(4)}_{\NLS}}$ is defined in \eqref{X_H}.
The symbols $a_2^{(1)}$ and $\vec{a}_1^{(1)}\cdot\xi$ are \emph{real} 
valued and satisfy the following estimates
\begin{equation}\label{realtaAAA2KK}
\begin{aligned}
|a_{2}^{(1)}|_{\mathcal{N}_p^{0}}& \lesssim \|u\|^{6}_{H^{p+s_0}}\,, 
\quad &\forall \, p+s_0\leq s\,,\,\quad p\in \mathbb{N}\,,\\
|\vec{a}_{1}^{(1)}\cdot\xi|_{\mathcal{N}_p^{1}}&\lesssim \|u\|^{6}_{H^{p+s_0+1}}\,,
\quad &\forall \, p+s_0+1\leq s\,,\,\quad p\in \mathbb{N}\,,
\end{aligned}
\end{equation} 
where we have chosen $s_0>d$.
The remainder $R^{(1)}$ has the form 
$(R^{(1,+)},\ov{R^{(1,+)}})^T$.
Moreover, for any $s> 2d+2$, it satisfies  the estimate
\begin{equation}\label{stimaRRRKK}
\begin{aligned}
\|R^{(1)}(U)\|_{H^{s}}\lesssim& \,\|U\|_{H^{s}}^{7}\,.
\end{aligned}
\end{equation}
\end{proposition}
\begin{proof}
The function $W$ defined in \eqref{cambio1} satisfies 
\begin{align}
\dot{W}&=\,[\partial_t\Phi_{\NLS}(U)]U+\Phi_{\NLS}(U)\dot{U}\nonumber\\ 
&=\,-\Phi_{\NLS}(U)\ii E\opbw\big((\uno+A_2(U))|\xi|^2\big)\Psi_{\NLS}(U) W 
-\Phi_{\NLS}(U)\ii EV*\Psi_{\NLS}(U) W\label{1}\\ 
&-\ii \Phi_{\NLS}(U)\opbw\big(\diag(\vec{a}_1(U)\cdot\xi)\big)\Psi_{\NLS}(U) W\label{2}\\ 
&+ \Phi_{\NLS}(U)X_{\mathcal{H}^{(4)}_{\NLS}}(U)\label{3}\\ 
&+\Phi_{\NLS}(U)R(U)+\opbw({\partial_{t}S^{-1}(U)})U\label{4}\\ 
&-\Phi_{\NLS}(U)\ii \Big[E\opbw\big((\uno+A_2(U))|\xi|^2\big) 
+\opbw\big(\diag(\vec{a}_1\cdot\xi)\big) 
+ EV*\Big]\tilde{R}(U)\Psi_{\NLS}(U)W\,, \label{5}
\end{align}
here we have used items 
$(ii)$ and $(iii)$ of Lemma \ref{propMAPPA}. 

\noindent 
We are going to analyze each term in the r.h.s. 
of the equation above.  
Because of estimates \eqref{stima-R-tilde}, \eqref{stime-descentTOTA} 
(applied for the map $\Phi_{\NLS}$), Lemma \ref{stime_diago_max} 
(applied for the symbols $a_2$, $b_2$ and $\vec{a}_1\cdot\xi$) and finally item $(ii)$ 
of Lemma \ref{azioneSimboo} we may absorb term \eqref{5} 
in the remainder $R^{(1)}(U)$ verifying \eqref{stimaRRRKK}. 
The term in \eqref{4} may be absorbed in $R^{(1)}(U)$ 
as well because of \eqref{stimaRRR} and 
\eqref{stime-descentTOTA} for the first term,  
because of \eqref{achille11} and item $(ii)$ 
of Lemma \ref{azioneSimboo} for the second one.

\noindent
We study the first term in \eqref{1}. 
We recall \eqref{diago_para} and \eqref{matriceS}, 
we apply Proposition \ref{prop:compo} and 
we get, by direct inspection, that the new term, 
modulo contribution that may be absorbed in 
$R^{(1)}(U)$, is given by
\begin{equation*}
-\ii E\opbw\Big(\diag({\lambda_{\NLS}})\Big)W
-2\ii\opbw\Big(\diag\big(\Im\big\{(s_2\bar{b}_2)\nabla s_1
+(s_1b_2+s_2(1+a_2))\nabla \bar{s}_2\big\}\cdot\xi\big)\Big)W\,,
\end{equation*}
where by $\Im\{\vec{b}\}$, with $\vec{b}=({b}_1,\ldots,{b}_d)$, 
we denoted the vector $(\Im({b}_1),\ldots,\Im({b}_d))$. 
The second term in \eqref{1} 
is equal to $-\ii EV*W$ modulo contributions to 
$R^{(1)}(U)$ thanks to \eqref{insPot} and \eqref{stime-descentTOTA}.

\noindent
Reasoning analogously one can prove that 
the term in \eqref{2} equals to
$
-\ii\opbw\big(\diag(\vec{a}_1(U)\cdot\xi)\big)W\,,
$
modulo contributions to $R^{(1)}(U)$. 
We are left with studying \eqref{3}. 
First of all we note that $X_{\mathcal{H}^{(4)}_{\NLS}}(U)=-\ii E |u|^2U$, 
then we write
\begin{equation*}
X_{\mathcal{H}^{(4)}_{\NLS}}(U)=X_{\mathcal{H}^{(4)}_{\NLS}}(W)
+X_{\mathcal{H}^{(4)}_{\NLS}}(U)-X_{\mathcal{H}^{(4)}_{\NLS}}(W)\,.
\end{equation*}
 Lemma \ref{stime_diago_max} and item $(ii)$ of 
Lemma \ref{azioneSimboo} (recall also \eqref{matriceS}), imply
$\|\Phi_{\NLS}(U)U-U\|_{H^s}\lesssim\|U\|^{7}_{H^s}\,$,
therefore it is a contribution to $R^{(1)}(U)$. 
We have obtained  
$\Phi_{\NLS}(U)X_{\mathcal{H}^{(4)}_{\NLS}}(U)
=X_{\mathcal{H}^{(4)}_{\NLS}}(W)$ modulo $R^{(1)}(U)$.

\noindent 
Summarizing we obtained the \eqref{QLNLS444KK} 
with symbols $a_2^{(1)}$ defined in \eqref{autovalori} and 
\begin{equation}\label{piccolaiena}
\vec{a}^{(1)}_1=\vec{a}_1+2\Im\big\{(s_2\bar{b}_2)\nabla s_1
+(s_1b_2+s_2(1+a_2))\nabla \bar{s}_2\big\}\in\R\,,
\end{equation}
with $\vec{a}_1$ in \eqref{simboa2}.
\end{proof}

\subsubsection{Diagonalization of cubic terms at order 0}\label{cubo}
The aim of this section 
is to diagonalize the
cubic vector field $X_{\mathcal{H}^{(4)}_{\NLS}}$ 
in \eqref{QLNLS444KK} (see also \eqref{X_H})
up to smoothing remainder. In order to do this we will consider 
a change of coordinates which is \emph{symplectic} up to high degree of homogeneity.
We reason as follows.

Let 
\begin{equation}\label{novavarV} 
Z:=\vect{z}{\bar{z}}:=\Phi_{\mathcal{B}_{\NLS}}(W):=W+X_{\mathcal{B}_{\NLS}}(W)
\end{equation}
where $X_{\mathcal{B}_{\NLS}}$ is the Hamiltonian vector field of \eqref{HamFUNC}.  We note that $\Phi_{\mathcal{B}_{\NLS}}$ is not symplectic, nevertheless it is close to the flow of $\mathcal{B}_{\NLS}(W)$ which is symplectic.
The properties of $X_{\mathcal{B}_{\NLS}}$ and the estimates of 
$\Phi_{\mathcal{B}_{\NLS}}$ have been discussed in 
Lemma \ref{HamvectorFieldG} and 
in Proposition \ref{flussononlin}. 
\begin{remark}\label{equivLemmaRMK}
Recall \eqref{cambio1} and \eqref{novavarV}.
One can note that, owing to Hypothesis \ref{2epsilon}, 
for $s>2d+2$, we have
\begin{equation}\label{alien1}
(1-\tfrac{1}{100})\|U\|_{H^{s}}\leq \|W\|_{H^{s}}\leq (1+\tfrac{1}{100})\|U\|_{H^{s}}\,,
\qquad
(1-\tfrac{1}{100})\|W\|_{H^{s}}\leq \|Z\|_{H^{s}}\leq (1+\tfrac{1}{100})\|W\|_{H^{s}}
\end{equation}
This is a consequence of the estimates 
\eqref{stime-descentTOTA}, \eqref{inv-phi}, 
\eqref{stimaflusso1}, \eqref{stimaXG4}, \eqref{approxinvKG2}
tanking $\e$ small enough depending on $s$.
\end{remark}
 We prove the following. 
\begin{proposition}[{\bf Diagonalization at order 0}]\label{prop:blockdiag00}
Let $U=(u,\bar{u})$ be a solution 
of \eqref{QLNLS444} and assume Hyp. \ref{2epsilon}. 
Define $W:=\Phi_{\NLS}(U)U$
where $\Phi_{\NLS}(U)$ is the map in \eqref{diago_para} given 
in Lemma \ref{propMAPPA}.
Then the function $Z=\vect{z}{\bar{z}}$ 
defined in \eqref{novavarV}
satisfies (recall \eqref{pollini1})
\begin{equation}\label{eq:ZZZ}
\begin{aligned}
\pa_{t}Z=-\ii E \Lambda_{\NLS} Z-\ii E\opbw\Big(& {\rm diag} 
\big(a_{2}^{(1)}(x)|\x|^2\big)
\Big)Z\\
&-\ii \opbw\Big({\rm diag}\big(\vec{a}_{1}^{(1)}(x)\cdot\x\big)\Big)Z+
X_{\mathtt{H}_{\NLS}^{(4)}}(Z)+
R_{5}^{(2)}(U)\,,
\end{aligned}
\end{equation}
where $a_2^{(1)}(x)$, $\vec{a}_{1}^{(1)}(x)
$
are the real valued symbols appearing in Proposition 
\ref{diago2max}, the cubic  vector field 
$X_{\mathtt{H}_{\NLS}^{(4)}}(Z)$ has the form
(see \eqref{cubici})
\begin{equation}\label{cubicifine}
X_{\mathtt{H}_{\NLS}^{(4)}}(Z):=
-\ii E\opbw\left(\begin{matrix}
2|z|^{2} & 0 \\0 & 2|z|^{2}
\end{matrix}\right)Z+Q_{\mathtt{H}^{(4)}_{\NLS}}(Z)\,,
\end{equation}
the remainder $Q_{\mathtt{H}^{(4)}_{\NLS}}$ is given by 
Lemma \ref{conj:hamfield}
and satisfies \eqref{formaRRRi}-\eqref{pollo2}.
The remainder $R_{5}^{(2)}(U)$
has the form
$({R}_{5}^{(2,+)},\ov{{R}_{5}^{(2,+)}})^{T}$. Moreover, for any $s> 2d+4$, 
\begin{equation}\label{stimaRRRfina}
\begin{aligned}
\|R_5^{(2)}(U)\|_{H^{s}}\lesssim& \,\|U\|_{H^{s}}^{5}\,.
\end{aligned}
\end{equation}
The vector field $X_{\mathtt{H}_{\NLS}^{(4)}}(Z)$ in \eqref{cubicifine}
is Hamiltonian, i.e. (see \eqref{VecfieldHam}, \eqref{Poisson}) 
$X_{\mathtt{H}_{\NLS}^{(4)}}(Z):=-\ii J\nabla\mathtt{H}_{\NLS}^{(4)}(Z)$
with
\begin{equation}\label{Hamfinale}
\qquad \mathtt{H}_{\NLS}^{(4)}(Z):=
\mathcal{H}_{\NLS}^{(4)}(Z)-\{\mathcal{B}_{\NLS}(Z),\mathcal{H}_{\NLS}^{(2)}(Z)\}\,,\quad
\mathcal{H}_{\NLS}^{(2)}(Z)=\int_{\mathbb{T}^{d}}\Lambda_{\NLS} z\cdot\bar{z}dx
\end{equation}
where $\mathcal{H}_{\NLS}^{(4)}$ is  in \eqref{hamiltonianaS},  
and $\mathcal{B}_{\NLS}$ is in \eqref{HamFUNC},
\eqref{gene1}.
\end{proposition}
\begin{proof}

 Recall \eqref{Hamprova}.
  We have that the equation 
 \eqref{QLNLS444KK} reads
 \[
 \pa_{t}W=X_{\mathcal{H}_{\NLS}^{(\leq4)}}(W)-\ii \opbw\big(A(U;x,\x)\big)W
 +R^{(1)}(U)
 \]
 where we set
 \begin{equation}\label{simbopollo}
 A(U;x,\x):=E{\rm diag}(a^{(1)}_{2}(U;x)|\x|^{2})
 +{\rm diag}(\vec{a}^{(1)}_{1}(U;x)\cdot\x)\,.
 \end{equation}
Hence by \eqref{novavarV}  we get
\begin{equation}\label{nuovaeq}
\begin{aligned}
\pa_{t}Z&=(d_{W}\Phi_{\mathcal{B}_{\NLS}}(W))\Big[-\ii \opbw\big( A(U;x,\x)\big)W\Big]+
(d_{W}\Phi_{\mathcal{B}_{\NLS}})(W)\big[X_{\mathcal{H}_{\NLS}^{(\leq4)}}(W)\big]\\&+
(d_{W}\Phi_{\mathcal{B}_{\NLS}})(W)\big[R^{(1)}(U)\big]\,.
\end{aligned}
\end{equation}
We study each summand separately.
First of all we have that 
\begin{equation}\label{breakdance5}
\|d_{W}\Phi_{\mathcal{B}_{\NLS}}(W)\big[R^{(1)}(U)\big]\|_{H^{s}}
\stackrel{\eqref{stimaXG4}, \eqref{stimaRRRKK}}{\lesssim}
\|u\|_{H^{s}}^{7}(1+\|w\|_{H^{s}}^{2})
\stackrel{\eqref{alien1}}{\lesssim}\|u\|_{H^{s}}^{7}.
\end{equation}
Let us now analyze the first summand in the 
r.h.s. of \eqref{nuovaeq}.
We write
\begin{equation}\label{alien2}
\begin{aligned}
&(d_{W}\Phi_{\mathcal{B}_{\NLS}}(W))\Big[\ii \opbw\big( A(U;x,\x)\big)W\Big]=
\ii \opbw\big( A(U;x,\x)\big)Z+P_1+P_2\,,\\
&P_1:=\ii \opbw\big( A(U;x,\x)\big)\big[W-Z\big]\,,\\
&P_2:=\Big((d_{W}\Phi_{\mathcal{B}_{\NLS}}(W))-\uno \Big)
\big[\ii \opbw\big( A(U;x,\x)\big)W\big]\,.
\end{aligned}
\end{equation}
Fix $s_0>d$, we have that, for $s\geq 2s_0+4$, 
\begin{equation}\label{breakdance}
\|P_2\|_{H^{s}}\stackrel{\eqref{stimaXG4}}{\lesssim}
\|w\|_{H^{s}}^{2}\| \opbw\big( A(U;x,\x)\big)W\|_{H^{s-2}}
\stackrel{\eqref{realtaAAA2KK}, \eqref{actionSob}, 
\eqref{alien1}}{\lesssim}\|u\|_{H^{s}}^{9}\,.
\end{equation}
By \eqref{novavarV}, \eqref{stimaXG4}
we get
$\|W-Z\|_{H^{s}}\lesssim
\|w\|_{H^{s-2}}^{3}\,$.
Therefore, by \eqref{alien2}, \eqref{simbopollo}, \eqref{realtaAAA2KK}, 
\eqref{actionSob} and \eqref{alien1}
we get
\begin{equation}\label{breakdance2}
\|P_1\|_{H^{s}}\lesssim\|u\|_{H^{2s_0+1}}^{6}\|W-Z\|_{H^{s+2}}\stackrel{}{\lesssim}
\|u\|_{H^{2s_0+1}}^{6}\|w\|_{H^{s}}^{3}\lesssim\|u\|_{H^{s}}^{9}\,.
\end{equation}
The estimates \eqref{breakdance5}, \eqref{breakdance}, \eqref{breakdance2}
imply that the term $P_1$, $P_2$ and 
$d_{W}\Phi_{\mathcal{B}_{\NLS}}(W)\big[R^{(1)}(U)\big]$ 
can be absorbed in a remainder
satisfying \eqref{stimaRRRfina}. 
Finally we consider the second summand in
\eqref{nuovaeq}.
By Lemma \ref{conjconjconj} we deduce
\[
d_{W}\Phi_{\mathcal{B}_{\NLS}}(W)\big[X_{\mathcal{H}_{\NLS}^{(\leq4)}}(W)\big]
=X_{\mathcal{H}_{\NLS}^{(\leq4)}}(Z)+\big[X_{\mathcal{B}_{\NLS}}(Z),
X_{\mathcal{H}_{\NLS}^{(2)}}(Z)\big]+R_{5}(Z)
\]
where $R_{5}$ is a remainder satisfying 
the quintic estimate \eqref{PPPPtau3}. 
By Lemma \ref{conj:hamfield} we also have 
 that
\[
X_{\mathcal{H}_{\NLS}^{(\leq4)}}(Z)+
\big[X_{\mathcal{B}_{\NLS}}(Z),X_{\mathcal{H}_{\NLS}^{(2)}}(Z)\big]=
-\ii E\Lambda_{\NLS} Z
+X_{\mathtt{H}_{\NLS}^{(4)}}(Z)\,,
\]
with $X_{\mathtt{H}_{\NLS}^{(4)}}$ as in \eqref{cubicifine}.
Moreover it is Hamiltonian with Hamiltonian as in \eqref{Hamfinale}
by formul\ae\, \eqref{cubici} and \eqref{nonlinCommu}.
This concludes the proof.
\end{proof}

\begin{remark}\label{campototale}
The Hamiltonian function  in \eqref{Hamfinale} may be rewritten, up to symmetrizations, 
as in \eqref{frate6} with coefficients $\mathtt{h}_4(\x,\eta,\zeta)$
satisfying \eqref{frate3}.
The coefficients of its Hamiltonian vector field have the form 
\eqref{frate44} (see also \eqref{frate331}). 
Moreover, by \eqref{cubicifine}, \eqref{quantiWeyl}, \eqref{cubici}, \eqref{formaRRRi},
we deduce that
\begin{equation}\label{alien6}
-2\ii\mathtt{h}_{4}(\x,\eta,\zeta)=-2\ii\chi_{\epsilon}\Big(\tfrac{|\x-\zeta|}{\langle\x+\zeta\rangle} \Big)+
\mathtt{q}_{\mathtt{H}^{(4)}_{\NLS}}(\x,\eta,\zeta)\,.
\end{equation}
\end{remark}

\subsection{Diagonalization of the KG}\label{DiagonaleKG}
In this section we   diagonalize 
the system \eqref{QLNLS444KG}
up to a smoothing remainder.
This will be done into two steps. 
We first diagonalize 
the matrix $E(\uno+\mathcal{A}_1(x,\x))$ in \eqref{QLNLS444KG} 
by means of a change of coordinates similar to the one made in the previous section for the \eqref{NLS} case.
After that we diagonalize the matrix of symbols of 
order $0$ at homogeneity $3$, by means of an \emph{approximatively symplectic} 
change of coordinates. Consider  the Cauchy problem associated to  \eqref{KG}.  
Throughout the rest of the section we shall assume the following.
\begin{hypo}\label{2epsilonKG}
We restrict the solution of \eqref{KG} on the interval of times $[0,T)$, 
with $T$ such that 
\begin{equation*}
\sup_{t\in[0,T)}\big(\|\psi(t,\cdot)\|_{H^{s+\frac{1}{2}}}
+\|\pa_t{\psi}(t,\cdot)\|_{H^{s-\frac{1}{2}}}\big)\leq \e\,, 
\quad \|\psi_0(\cdot)\|_{H^{s+\frac{1}{2}}}
+\|{\psi}_1(\cdot)\|_{H^{s-\frac{1}{2}}}\leq 1/32 \e\,,
\end{equation*}
with $\psi(0,x)=\psi_0(x)$
 and $(\pa_{t}\psi)(0,x)=\psi_1(x)$\,.
 \end{hypo}
Note that such a $T$ exists thanks to the local well-posedness proved in \cite{kato}.

\begin{remark}\label{WWUUUKG}
Recall the \eqref{CVWW}. Then one can note that 
\[
\frac{1}{4}\big(\|\psi(t,\cdot)\|_{H^{s+\frac{1}{2}}}
+\|\pa_t{\psi}(t,\cdot)\|_{H^{s-\frac{1}{2}}}\big)
\leq \|u\|_{H^{s}}\leq 
2
\big(\|\psi(t,\cdot)\|_{H^{s+\frac{1}{2}}}
+\|\pa_t{\psi}(t,\cdot)\|_{H^{s-\frac{1}{2}}}\big)\,.
\]
\end{remark}

\subsubsection{Diagonalization at order 1}
Consider the matrix of symbols (see  \eqref{simboa2KG}, \eqref{calAA1KG})
\begin{equation}\label{A2tildeKG}
\begin{aligned}
&E(\uno+\mathcal{A}_{1}(x,\x))\,,\qquad 
\mathcal{A}_{1}(x,\x)
:=\sm{1}{1}{1}{1}\widetilde{a}_{2}(x,\x)\,,\quad\widetilde{a}_{2}(x,\x):=\tfrac{1}{2}\Lambda_{\KG}^{-2}(\x)a_2(x,\x)\,.
\end{aligned}
\end{equation}
Define 
\begin{equation}\label{nuovadiagKG}
\begin{aligned}
\lambda_{\KG}(x,\x)&:= \sqrt{(1+\widetilde{a}_{2}(x,\x))^{2}
-(\widetilde{a}_{2}(x,\x))^{2}},
\qquad \widetilde{a}_{2}^{+}(x,\x):=\lambda_{\KG}(x,\x)-1.
\end{aligned}
\end{equation}
Notice that the symbol $\lambda_{\KG}(x,\x)$ is well-defined by 
taking $\|u\|_{H^{s}}\ll1$ small enough.
The matrix of  eigenvectors associated to the eigenvalues of 
$E(\uno+\mathcal{A}_{1}(x,\x))$
is 
\begin{equation}\label{transCKG}
\begin{aligned}
S(x,\x)&:=\left(\begin{matrix} {s}_1(x,\x) & {s}_2(x,\x)\vspace{0.2em}\\
{{s_2(x,\x)}} & {{s_1(x,\x)}}
\end{matrix}
\right)\,,
\qquad
S^{-1}(x,\x):=\left(\begin{matrix} {s}_1(x,\x) & -{s}_2(x,\x)\vspace{0.2em}\\
-{{s_2(x,\x)}} & {{s_1(x,\x)}}
\end{matrix}
\right)\,,
\\
s_{1}&:=\frac{1+\widetilde{a}_{2}
+\lambda_{\KG}}{\sqrt{2\lambda_{\KG}\big(1+\widetilde{a}_{2}+
\lambda_{\KG}\big) }},
\qquad s_{2}:=\frac{-\widetilde{a}_{2}}{\sqrt{2\lambda_{\KG}\big(1
+\widetilde{a}_{2}
+\lambda_{\KG}\big) }}\,.
\end{aligned}
\end{equation}
By a direct computation one can check that
\begin{equation}\label{diagodiagoKG}
S^{-1}(x,\x)E(\uno+\mathcal{A}_{1}(x,\x))S(x,\x)= E \diag(\lambda_{\KG}(x,\xi)),
\qquad s_1^{2}-|s_2|^{2}=1\,.
\end{equation}
We shall study how the system \eqref{QLNLS444KG}
transforms under the maps
\begin{equation}\label{Mappa1KG}
\begin{aligned}
\Phi_{\KG}&=\Phi_{\KG}(U)[\cdot]:=\opbw(S^{-1}(x,\x))\,,\qquad
\Psi_{\KG}=\Psi_{\KG}(U)[\cdot]:=\opbw(S(x,\x))\,.
\end{aligned}
\end{equation}
We prove the following result.
\begin{lemma}\label{propMAPPAKG}
Assume Hypothesis \ref{2epsilonKG}.
We have the following:
\begin{enumerate}
\item[$(i)$] if $s_0>d$,
then
\begin{equation}\label{mission1KG}
|\widetilde{a}^{+}_{2}|_{\mathcal{N}^{0}_{p}}
+|\widetilde{a}_{2}|_{\mathcal{N}^{0}_{p}}
+|s_1-1|_{\mathcal{N}^{0}_{p}}
+|s_2|_{\mathcal{N}^{0}_{p}}\lesssim
\|u\|_{H^{p+s_0+1}}^{3}\,,\quad p+s_0+1\leq s\,;
\end{equation}

\item[$(ii)$] for any $s\in \mathbb{R}$ one has
\begin{equation}\label{stime-descentTOTAKG}
\begin{aligned}
\|\Phi_{\KG}(U)V-V\|_{H^{s}}+\|\Psi_{\KG}(U)V-V\|_{H^{s}}&\lesssim
\|{V}\|_{H^{s}}\|u\|^{3}_{H^{2s_0+1}}\,,
\qquad \forall\, V\in H^{s}(\mathbb{T}^{d};\mathbb{C}^2)\,;
\end{aligned}
\end{equation} 

\item[$(iii)$] one has $\Psi_{\KG}(U)\circ\Phi_{\KG}(U)=\uno+Q(U)$
where $Q$ is a real-to-real remainder 
satisfying 
\begin{equation}\label{achille10KG}
\|Q(U)V\|_{H^{s+1}}\lesssim  \|V\|_{H^{s}}\|u\|^{3}_{H^{2s_0+3}}\,;
\end{equation}

\smallskip
\item[$(iv)$] for  any $t\in [0,T)$,
one has 
$\pa_{t}\Phi_{\KG}(U)[\cdot]=\opbw(\pa_{t}S^{-1}(x,\x))$
and
\begin{equation}\label{achille11KG}
|\pa_{t}S^{-1}(x,\x)|_{\mathcal{N}_{s_0}^{0}}\lesssim
 \|u\|_{H^{2s_0+3}}^{3}\,,
\qquad
\|\pa_{t}\Phi_{\KG}(U)V\|_{H^{s}}\lesssim \|V\|_{H^{s}}\|u\|^{3}_{H^{2s_0+3}}\,.
\end{equation}
\end{enumerate}
\end{lemma}
\begin{proof}
$(i)$ The \eqref{mission1KG} follows by \eqref{realtaAAA2KG}
using the explicit formul\ae\, \eqref{transCKG}, \eqref{nuovadiagKG}.

\noindent
$(ii)$ It follows by using \eqref{mission1KG} and item $(ii)$ in  Lemma \ref{azioneSimboo}.

\noindent
$(iii)$ By formula \eqref{composit} in Proposition \ref{prop:compo} one gets
\[
\Psi_{\KG}(U)\circ\Phi_{\KG}(U)=\uno+\opbw\left(
\begin{matrix}
0 & \ii \{s_1,s_2\} \\ -\ii \{s_1,s_2\} & 0
\end{matrix}
\right)+R(s_1,s_2)\,,
\]
for some remainder satisfying \eqref{composit2} with $a\rightsquigarrow s_1$
and $b\rightsquigarrow s_2$. Therefore the \eqref{achille10KG}
follows by using \eqref{prodSimboli}, \eqref{actionSob}
and \eqref{mission1KG}.

\noindent
$(iv)$ It is similar to the proof of item $(iii)$ of Lemma \ref{propMAPPA}.
\end{proof}


\begin{proposition}[{\bf Diagonalization at order $1$}]\label{diagoOrd1}
Consider the system \eqref{QLNLS444KG} and set 
\begin{equation}\label{cambio1KG}
W=\Phi_{\KG}(U)U\,,
\end{equation}
with $\Phi_{\KG}$ defined in \eqref{Mappa1KG}. 
Then $W$ solves the equation (recall \eqref{matriciozze})
\begin{equation}\label{QLNLS444KKKG}
\begin{aligned}
\partial_t{W}=&-\ii E\opbw\big(\diag\big(1+\widetilde{a}_{2}^{+}(x,\x)\big)
\Lambda_{\KG}(\x))W+
X_{\mathcal{H}^{(4)}_{\KG}}(W)+{R}^{(1)}(u)\,,
\end{aligned}
\end{equation}
where  the vector field $X_{\mathcal{H}^{(4)}_{\KG}}$ 
is defined in \eqref{calGGG1KG}.
The symbol $\tilde{a}_2^{+}$ is 
defined in \eqref{nuovadiagKG}.
The remainder ${R}^{(1)}$ has the form 
$({R}^{(1,+)},\ov{{R}^{(1,+)}})^T$.
Moreover, for any $s> 2d+\mu$, for some $\mu>0$, it satisfies  the estimate
\begin{equation}\label{stimaRRRKKKG}
\begin{aligned}
\|{R}^{(1)}(u)\|_{H^{s}}\lesssim& \,\|u\|_{H^{s}}^{4}\,.
\end{aligned}
\end{equation}
\end{proposition}

\begin{proof}
By \eqref{cambio1KG}
 and \eqref{QLNLS444KG} we get
\begin{equation} \label{jap1KG}\begin{aligned}
 \pa_{t}W&=\Phi_{\KG}(U)\dot{U}+(\pa_{t}\Phi_{\KG}(U))[U]\\
 &=-\ii\Phi_{\KG}(U)\opbw\big(E(\uno+\mathcal{A}_1(x,\x))\Lambda_{\KG}(\x)\big)
 \Psi_{\KG}(U)W
 \\
 &+\Phi_{\KG}(U)X_{\mathcal{H}^{(4)}_{\KG}}(U)\\
 &+\Phi_{\KG}(U){R}(u)+(\pa_{t}\Phi_{\KG}(U))[U]\\
 &+\ii\Phi_{\KG}(U)\opbw\big(E(\uno+\mathcal{A}_1(x,\x))(\x)\big)Q(U)U\,,
  \end{aligned}
  \end{equation}
where we used items $(ii)$, $(iii)$ in Lemma \ref{propMAPPAKG}.
We study the first summand in the r.h.s of  \eqref{jap1KG}. By direct inspection,
using Lemma \ref{azioneSimboo} and Proposition 
\ref{prop:compo} we get
\[
\begin{aligned}
-\ii\Phi_{\KG}(U)\opbw\big(E(\uno+\mathcal{A}_1(x,\x))\Lambda_{\KG}(\x)\big)
\Psi_{\KG}(U)
&=-\ii \opbw\Big(S^{-1}E(\uno+\mathcal{A}_1(x,\x))S\Big)+R(u)\\
&\stackrel{\mathclap{\eqref{diagodiagoKG}}}{=}-\ii E\opbw(\diag(\lambda_{\KG}(x,\xi)))+R(u)
\end{aligned}
\]
where $R(u)$ is a remainder satisfying \eqref{stimaRRRKKKG}.
Thanks to the discussion above and \eqref{nuovadiagKG} we obtain the highest order term in \eqref{QLNLS444KKKG}. All the other summands in the r.h.s. of \eqref{jap1KG} may be analyzed as done in the proof of Prop. \ref{diago2max} by using Lemma \ref{propMAPPAKG}.
\end{proof}

\subsubsection{Diagonalization of cubic terms at order 0}
In the previous section we showed that if the function $U$ 
solves \eqref{QLNLS444KG} then $W$ in \eqref{cambio1KG}
solves \eqref{QLNLS444KKKG}.
The cubic terms in the system \eqref{QLNLS444KKKG}
are the same appearing in \eqref{QLNLS444KG} and have the form
\eqref{calGGG1KG}.
The aim of this section is to  diagonalize the matrix of symbols
of order zero $\mathcal{A}_0(x,\x)$.

Let us define
\begin{equation}\label{varZZZKG}
Z:=\vect{z}{\bar{z}}:=\Phi_{\mathcal{B}_{\KG}}(W):=W+X_{\mathcal{B}_{\KG}}(W)
\end{equation}
where $X_{\mathcal{B}_{\KG}}$ is the Hamiltonian vector field
of \eqref{HamFUNCKG} and $W$ is
the function in \eqref{cambio1KG}.
The properties of $X_{\mathcal{B}_{\KG}}$ and the estimates of 
$\Phi_{\mathcal{B}_{\KG}}$ have been discussed in 
Lemma \ref{HamvectorFieldG} and 
in Proposition \ref{flussononlin}.

\begin{remark}\label{equivLemmaRMKKG}
Recall \eqref{cambio1KG} and \eqref{varZZZKG}.
One can note that, owing to Hypothesis \ref{2epsilonKG}, 
for $s>2d+3$, we have
\begin{equation}\label{alienKG1}
(1-\tfrac{1}{100})\|U\|_{H^{s}}\leq \|W\|_{H^{s}}\leq (1+\tfrac{1}{100})\|U\|_{H^{s}}\,,
\qquad
(1-\tfrac{1}{100})\|W\|_{H^{s}}\leq \|Z\|_{H^{s}}\leq (1+\tfrac{1}{100})\|W\|_{H^{s}}
\end{equation}
This is a consequence of the estimates 
\eqref{stime-descentTOTAKG}, \eqref{achille10KG}
\eqref{stimaflusso1}, \eqref{stimaXG4KG}, \eqref{approxinvKG2} taking $\e$ small enough.
\end{remark}

%
%
%


\begin{proposition}[{\bf Diagonalization at order 0}]\label{prop:blockdiag00KG}
Let $U$ be a solution of \eqref{QLNLS444KG} and assume
 Hyp. \ref{2epsilonKG} (see also 
Remark \ref{WWUUUKG}).
Then the function $Z$ defined in \eqref{varZZZKG}, with $W$ 
given in \eqref{cambio1KG},
satisfies
\begin{equation}\label{eq:ZZZKG}
\begin{aligned}
\pa_{t}Z&=-\ii E\opbw\big(\diag\big(1+\widetilde{a}_{2}^{+}(x,\x)\big)\Lambda_{\KG}(\x))Z+
X_{\mathtt{H}^{(4)}_{\KG}}(Z)+{R}^{(2)}_{4}(u)\,,
\end{aligned}
\end{equation}
where $\widetilde{a}_2^{+}(x,\x)$ 
is the real valued symbol 
in \eqref{nuovadiagKG}, the cubic  vector field 
$X_{\mathtt{H}^{(4)}_{\KG}}(Z)$ has the form
\begin{equation}\label{cubicifineKG}
X_{\mathtt{H}^{(4)}_{\KG}}(Z):=
-\ii E\opbw\big(\diag(a_0(x,\xi))\big)Z+Q_{\mathtt{H}^{(4)}_{\KG}}(Z)
\end{equation}
the symbol $a_0(x,\x)$ is in \eqref{simboa2KG}, the remainder
$Q_{\mathtt{H}^{(4)}_{\KG}}(Z)$ is the cubic remainder 
given in Lemma \ref{conj:hamfieldKG}.
The remainder  $R_{4}^{(2)}(u)$
has the form
$({R}_{4}^{(2,+)}(u),\ov{{R}_{4}^{(2,+)}}(u))^{T}$. Moreover, 
for any $s> 2d+\mu$, for some $\mu>0$,  we have the estimate
\begin{equation}\label{stimaRRRfinaKG}
\begin{aligned}
\|R_{4}^{(2)}(u)\|_{H^{s}}\lesssim& \,\|u\|_{H^{s}}^{4}\,.
\end{aligned}
\end{equation}
Finally the vector field $X_{\mathtt{H}^{(4)}_{\KG}}(Z)$ in \eqref{cubicifineKG}
is Hamiltonian, i.e. $X_{\mathtt{H}^{(4)}_{\KG}}(Z):=-\ii J\nabla \mathtt{H}^{(4)}_{\KG}(Z)$
with 
\begin{equation}\label{HamfinaleKG}
\qquad \mathtt{H}^{(4)}_{\KG}(Z):=
\mathcal{H}^{(4)}_{\KG}(Z)-\{\mathcal{B}_{\KG}(Z),\mathcal{H}^{(2)}_{\KG}(Z)\}\,,\quad
\mathcal{H}^{(2)}_{\KG}(Z)=\int_{\mathbb{T}^{d}}\Lambda_{\KG} z\cdot\bar{z}dx
\end{equation}
where $\mathcal{H}^{(4)}_{\KG}$ is  in \eqref{calGGGKG},  
and $\mathcal{B}_{\KG}$ is in \eqref{HamFUNCKG},
\eqref{gene1KG}.
\end{proposition}

\begin{proof}
We recall \eqref{HamprovaKG} and we rewrite the 
equation \eqref{QLNLS444KKKG} as
\[
\pa_{t}W=X_{ \mathcal{H}^{(\leq4)}_{\KG}}(W)
-\ii E\opbw\big( \widetilde{a}_2^{+}(x,\x)\Lambda_{\KG}(\x)\big)W
+{R}^{(1)}(u).
\]
Then, using \eqref{varZZZKG}, we get
\begin{align}
\pa_{t}Z&=d_{W}\Phi_{\mathcal{B}_{\KG}}(W)\big[\pa_{t}W\big]\nonumber\\
&=d_{W}\Phi_{\mathcal{B}_{\KG}}(W)\big[X_{\mathcal{H}^{(\leq4)}_{\KG}}(W)\big]\label{robinKG1}\\
&+d_{W}\Phi_{\mathcal{B}_{\KG}}(W)\big[
-\ii E\opbw\big(\diag( \widetilde{a}_2^{+}(x,\x)\Lambda_{\KG}(\x))\big)W\big]
\label{robinKG2}\\
&+d_{W}\Phi_{\mathcal{B}_{\KG}}(W)\big[{R}^{(1)}(u)\big]\,.\label{robinKG3}
\end{align}
By estimates \eqref{stimaXG4KG} and \eqref{stimaRRRKKKG}
we have that the term in \eqref{robinKG3} can be absorbed in a
remainder satisfying the   \eqref{stimaRRRfinaKG}.
Consider the term in \eqref{robinKG2}. We write
\begin{equation}\label{alienKG2}
\begin{aligned}
&
\eqref{robinKG2}
=
-\ii E\opbw\big(\diag( \widetilde{a}_2^{+}(x,\x)\Lambda_{\KG}(\x))\big)Z+P_1+P_2\,,\\
&P_1:=-\ii E\opbw\big(\diag( \widetilde{a}_2^{+}(x,\x)\Lambda_{\KG}(\x))\big)
\big[W-Z\big]\,,\\
&P_2:=\Big((d_{W}\Phi_{\mathcal{B}_{\KG}}(W))-\uno \Big)
\big[-\ii E\opbw\big(\diag( \widetilde{a}_2^{+}(x,\x)\Lambda_{\KG}(\x))\big)W\big]\,.
\end{aligned}
\end{equation}
We have that, for $s\geq 2s_0+2$, 
\begin{equation*}
\|P_2\|_{H^{s}}\stackrel{\eqref{stimaXG4KG}}{\lesssim}
\|u\|_{H^{s}}^{2}\| \opbw\big(\widetilde{a}_2^{+}(x,\x)\Lambda_{\KG}(\x)\big)w\|_{H^{s-1}}
\stackrel{\eqref{mission1KG}, \eqref{actionSob}, 
\eqref{alienKG1}}{\lesssim}\|u\|_{H^{s}}^{6}\,,
\end{equation*}
which implies the \eqref{stimaRRRfinaKG}.
 By \eqref{approxinvKG2} in 
Lemma \ref{flussononlin} and estimate \eqref{stimaXG4KG} we deduce 
$
\|W-Z\|_{H^{s+1}}\lesssim\|u\|^{3}_{H^{s}}\,.
$
Hence using again \eqref{mission1KG}, \eqref{actionSob}, 
\eqref{alienKG1} we get $P_1$ satisfies \eqref{stimaRRRfinaKG}.
It remains to discuss the structure of the term in \eqref{robinKG1}.
By Lemma \ref{conjconjconj} we obtain
\begin{align}
d_{W}\Phi_{\mathcal{B}_{\KG}}(W)\big[X_{\mathcal{H}^{(\leq4)}_{\KG}}(W)\big]=
X_{\mathcal{H}^{(\leq4)}_{\KG}}(Z)
+\big[X_{\mathcal{B}_{\KG}}(Z),X_{\mathcal{H}^{(2)}_{\KG}}(Z)\big]\,,\label{gordon4}
\end{align}
modulo remainders that can be absorbed in $R_{4}^{(2)}$ satisfying \eqref{stimaRRRfinaKG}.
The \eqref{gordon4}, \eqref{robinKG1}-\eqref{robinKG3} and the discussion above
imply the  equation \eqref{eq:ZZZKG} where the cubic vector field, has the form
\begin{equation}\label{gothamKG5}
X_{\mathtt{H}^{(4)}_{\KG}}(Z)=X_{\mathcal{H}^{(4)}_{\KG}}(Z)+
\big[X_{\mathcal{B}_{\KG}}(Z),X_{\mathcal{H}^{(2)}_{\KG}}(Z)\big]\,.
\end{equation}
Using \eqref{nonlinCommu}, \eqref{Poisson}, 
we conclude that $X_{\mathtt{H}^{(4)}_{\KG}}$
is the Hamiltonian vector field of $\mathtt{H}^{(4)}_{\KG}$ in \eqref{HamfinaleKG}.
The 
 \eqref{cubicifineKG} follows by Lemma \ref{conj:hamfieldKG}.
\end{proof}

\begin{remark}\label{semilin3}
In view of Remarks \ref{semilin1}, \ref{semilin2}, following the same proof of 
Proposition \ref{prop:blockdiag00KG}, in the semi-linear case we obtain that
equation \eqref{eq:ZZZKG} reads
\[
\pa_{t}Z=-\ii E\opbw\big(\diag(\Lambda_{\KG}(\x))\big)Z+
X_{\mathtt{H}^{(4)}_{\KG}}(Z)+{R}^{(2)}_{4}(u)\,,
\]
where $X_{\mathtt{H}^{(4)}_{\KG}}$ has the form \eqref{cubicifineKG}
with $a_0(x,\x)$ a symbol of order $-1$ and $Q_{\mathtt{H}^{(4)}_{\KG}}$
a remainder of the form \eqref{restoQ3GGGKG} with coefficients satisfying 
\eqref{restoQ32GGGKG} with the better denominator
$\max\{\langle\x-\eta-\zeta\rangle,\langle \eta\rangle,\langle\zeta\rangle\}^{2}$.
\end{remark}

\section{Energy estimates}

\subsection{Estimates for the NLS}
In this section we  prove \emph{a priori}
energy estimates on the Sobolev norms of the variable $Z$ in \eqref{novavarV}.
In subsection \ref{sec:highder}
we  introduce a convenient  energy norm on $H^{s}(\mathbb{T}^{d};\mathbb{C}) $
which is equivalent to the classic $H^{s}$-norm.
This is the content of Lemma \ref{equivLemma}.
In subsection \ref{sec:stimenonres}, using the 
non-resonance conditions of Proposition \ref{NonresCond},
we provide
bounds on the non-resonant terms
appearing in the energy estimates.
We deal with 
 resonant interactions  in Lemma \ref{superazioni}.

\subsubsection{Energy norm}\label{sec:highder}
Let us define the symbol 
\begin{equation}\label{simboLLL}
\mathcal{L}=\mathcal{L}(x,\x):=|\x|^{2}+\Sigma\,,\quad \Sigma=\Sigma(x,\x):=
a_{2}^{(1)}(x)|\x|^2+\vec{a}_{1}^{(1)}(x)\cdot\x\,,
\end{equation}
where
the symbols $a_{2}^{(1)}(x)$, $\vec{a}_{1}^{(1)}(x)$
are given in Proposition \ref{diago2max}.
We have the following.

\begin{lemma}\label{c}
Assume the Hypothesis \ref{2epsilon} and let  $\gamma>0$. 
Then for 
$\e>0$ small enough  we have the
following.

\noindent
$(i)$ One has
\begin{equation}\label{LLLmeno1}
\begin{aligned}
&
|\Sigma|_{\mathcal{N}_{s_0}^{2}}\leq C\|u\|^{6}_{H^{2s_0+1}}\,,
\qquad \qquad 
|(1+\mathcal{L})^{\gamma}-(|\x|^{2}+1)^{\gamma}|_{\mathcal{N}_{s_0}^{2\gamma}}
\lesssim_{\gamma} C\|u\|^{6}_{H^{2s_0+1}}
\end{aligned}
\end{equation}
for some $C>0$ depending on $s_0$.

\noindent
$(ii)$
For any $s\in \mathbb{R}$ and any $h\in H^{s}(\mathbb{T}^{d};\mathbb{C})$, 
one has
\begin{equation}\label{stimaTL}
\begin{aligned}
&
\|T_{\mathcal{L}^{\gamma}}h\|_{H^{s-2\gamma}}
\leq \|h\|_{H^{s}}(1+C\|u\|^{6}_{H^{2s_0+1}})\,,\\
&\|T_{\Sigma}h\|_{H^{s-2}}
+\|T_{(1+\mathcal{L})^{\gamma}
-(|\x|^{2}+1)^{\gamma}}h\|_{H^{s-2\gamma}}
\lesssim_{\gamma} \|h\|_{H^{s}}\|u\|^{6}_{H^{2s_0+1}}\,,
\end{aligned}
\end{equation}
for some $C>0$ depending on $s$ and $\gamma$.

\noindent
$(iii)$ 
For any $t\in [0,T)$ one has
$|\pa_{t}\Sigma|_{\mathcal{N}^{2}_{s_0}}\lesssim \|u\|^{6}_{H^{2s_0+3}}\,.$
Moreover
\begin{equation}\label{LLLtempo2}
\|(T_{\pa_t(1+\mathcal{L})^{\gamma}})h\|_{H^{s-2\gamma}}\lesssim_{\gamma}
\|h\|_{H^{s}}\|u\|^{6}_{H^{2s_0+3}}\,,
\qquad \forall\, h\in H^{s}(\mathbb{T}^{d};\mathbb{C})\,.
\end{equation}

\noindent
$(iv)$ The operators $T_{\mathcal{L}}$, $T_{(1+\mathcal{L})^{\gamma}}$ are self-adjoint 
with respect to the $L^2$-scalar product \eqref{scalarL}.
\end{lemma}

\begin{proof} Items $(i)$-$(ii)$. The \eqref{LLLmeno1} 
follows by using \eqref{simboLLL}, the bounds 
\eqref{realtaAAA2KK} on the symbols $a_{2}^{(1)}$ and $\vec{a}_1^{(1)}\cdot\x$.
The \eqref{stimaTL} follows by 
 Lemma \ref{azioneSimboo}.
 
 \noindent
 Item $(iii)$.
The bound  on $\partial_t\Sigma$ follows by reasoning as in item $(iii)$ 
of Lemma \ref{propMAPPA} using the explicit formula of 
$a_2^{(1)}$ in \eqref{autovalori} and the formula for $a_{1}^{(1)}\cdot\x$
in \eqref{piccolaiena} (see also \eqref{matriceS}).
Then the \eqref{actionSob} implies the \eqref{LLLtempo2}.

\noindent
Item $(iv)$.  This follows by \eqref{simboAggiunto}
since the symbol $\mathcal{L}$
in \eqref{simboLLL} is real-valued.
\end{proof}

In the following we shall construct the \emph{energy norm}. 
By using this norm we are able to achieve 
the \emph{energy estimates} on the previously diagonalized system. 
For $s\in \mathbb{R}$ we define 
\begin{equation}\label{def:VVTL}
z_{n}:=T_{(1+\mathcal{L})^{n}}z\,,\qquad 
Z_{n}=\vect{z_n}{\ov{z_{n}}}
:=T_{(1+\mathcal{L})^{n}}\uno Z\,,\quad Z=\vect{z}{\bar{z}}\,,
\quad n:={s}/{2}\,.
\end{equation}

\begin{lemma}{\bf (Equivalence of the energy norm).}\label{equivLemma}
Assume Hypothesis \ref{2epsilon} with $s>2d+4$. Then,  for  $\e>0$ small enough enough, one has
\begin{equation}\label{equivNorme}
(1-\tfrac{1}{100})\|z\|_{H^{s}}\leq\|z_{n}\|_{L^{2}}\leq(1+\tfrac{1}{100})\|z\|_{H^{s}}\,.
\end{equation}
\end{lemma}

\begin{proof}
Let $s=2n$. Then by \eqref{stimaTL} and \eqref{def:VVTL} we have
$
\|z_{n}\|_{L^{2}}\leq \|z\|_{H^{s}}(1+C\|u\|^{6}_{H^{2s_0+1}})\leq 2\|z\|_{H^{2}}\,,
$
with $s_0>d$.
Moreover 
\[
 \|z\|_{H^{s}}=\|T_{(1+|\x|^{2})^{n}}z\|_{L^{2}}\stackrel{\eqref{stimaTL}}{\leq}
\|z_{n}\|_{L^{2}}+C\|z\|_{H^{s}}\|u\|^{6}_{H^{2s_0+1}}
 \]
 which implies
 $(1-C\|u\|^{6}_{H^{2s_0+1}})\|z\|_{H^{s}}\leq \|z_n\|_{L^{2}}\,$,
for some constant $C$ depending on $s$.
The discussion above implies the \eqref{equivNorme}
by taking $\e>0$ in Hyp. \ref{2epsilon} small enough.
\end{proof}

Recalling \eqref{eq:ZZZ}, \eqref{pollini1} and \eqref{simboLLL} we have
\begin{equation}\label{eq:ZZZbis}
(\pa_{t}+\ii\Lambda_{\NLS})z=-\ii T_{\Sigma}z+X_{\mathtt{H}_{\NLS}^{(4)} }^+(Z)+
R_{5}^{(2,+)}(U)\,,
\qquad Z=\vect{z}{\bar{z}}\,,
\end{equation}
where $X_{\mathtt{H}_{4}^{(4)}}$ is given in \eqref{cubicifine} 
(see also Remark \ref{campototale}) and $R_{5}^{(2,+)}$
is the remainder satisfying \eqref{stimaRRRfina}.
%

\begin{lemma}\label{equaVVnn}
Fix $s>2d+4$ and recall \eqref{eq:ZZZbis}.
One has that the function $z_n$ defined in \eqref{def:VVTL}
solves the problem
\begin{equation}\label{equa:VVTL}
\pa_{t}z_n=-\ii T_{\mathcal{L}}z_{n}-\ii V*z_n
+T_{(1+|\x|^2)^{n}}X^{+,{\rm res}}_{\mathtt{H}^{(4)}_{\NLS}}(Z)
+B_{n}^{(1)}(Z)+B_{n}^{(2)}(Z)
+R_{5,n}(U)\,,
\end{equation}
where $X^{+,{\rm res}}_{\mathtt{H}_{\NLS}^{(4)}}$ is defined as in Def. \ref{def:resonantSet},
 \begin{equation}\label{alien88}
\begin{aligned}
\widehat{B_{n}^{(1)}(Z)}(\x)&=\frac{1}{(2\pi)^{d}}
\sum_{\eta,\zeta\in \mathbb{Z}^{d}}b^{(1)}(\x,\eta,\zeta)
\hat{z}(\x-\eta-\zeta)\hat{\bar{z}}(\eta)\hat{z}_{n}(\zeta)\,,\\
\widehat{B_{n}^{(2)}(Z)}(\x)&=\frac{1}{(2\pi)^{ d}}
\sum_{\eta,\zeta\in \mathbb{Z}^{d}}b_{n}^{(2)}(\x,\eta,\zeta)
\hat{z}(\x-\eta-\zeta)\hat{\bar{z}}(\eta)\hat{z}(\zeta)\,,
\end{aligned}
\end{equation}
with 
\begin{align}
b^{(1)}(\x,\eta,\zeta)&:=-2\ii \chi_{\epsilon}\Big(\tfrac{|\x-\zeta|}{\langle\x+\zeta\rangle} \Big)
\mathtt{1}_{\mathcal{R}^{c}}(\x,\eta,\zeta)\,, \label{alien99}\\
|b_{n}^{(2)}(\x,\eta,\zeta)|&\lesssim 
\frac{\langle \x\rangle^{2n}\max_{2}\{|\x-\eta-\zeta|,|\eta|,|\zeta|\}^{4}}
{\max_1\{\langle\x-\eta-\zeta\rangle,\langle\eta\rangle,\langle\zeta\rangle\}}\mathtt{1}_{\mathcal{R}^{c}}(\x,\eta,\zeta)\,,
\label{alien9}
\end{align}
and where  the remainder $R_{5, n}$ satisfies
\begin{equation}\label{stimeRRRN}
\|R_{5,n}(U)\|_{L^{2}}\lesssim \|u\|^{5}_{H^{s}}\,.
\end{equation}
\end{lemma} 

\begin{proof}
Recalling \eqref{alien7} we
define
\begin{equation}\label{alien12}
X_{\mathtt{H}_{\NLS}^{(4)}}^{+,\perp}(Z):=X_{\mathtt{H}_{\NLS}^{(4)}}^{+}(Z)
-X_{\mathtt{H}_{\NLS}^{(4)}}^{+,{\rm res}}(Z)\,.
\end{equation}
By differentiating \eqref{def:VVTL} 
and using the \eqref{simboLLL} and \eqref{eq:ZZZbis} we get
\begin{equation}\label{alien11}
\begin{aligned}
\pa_{t}z_{n}&=T_{(1+\mathcal{L})^{n}}\pa_{t}z
+T_{\pa_{t}(1+\mathcal{L})^{n}}z\\
&=-\ii T_{\mathcal{L}}z_{n}-\ii T_{(1+\mathcal{L})^{n}}(V*z)
+T_{(1+\mathcal{L})^{n}}X_{\mathtt{H}_{\NLS}^{(4)}}^{+}(Z)+
T_{(1+\mathcal{L})^{n}}R_{5}^{(2,+)}(U)
\\&+T_{\pa_{t}(1+\mathcal{L})^{n}}z-\ii [T_{(1+\mathcal{L})^{n}},T_{\mathcal{L}}]z\,.
\end{aligned}
\end{equation}
By using Lemmata \ref{azioneSimboo}, \ref{c} 
 and Proposition \ref{prop:compo}, and the \eqref{equivNorme}, \eqref{alien1} 
 one proves that the last  summand gives a contribution to
 $R_{5,n}(U)$ satisfying \eqref{stimeRRRN}.
By using \eqref{LLLtempo2}, \eqref{alien1}, 
\eqref{stimaRRRfina} we deduce that
\begin{equation*}
\begin{aligned}
&\|T_{(1+\mathcal{L})^{n}}R_{5}^{(2,+)}(U)\|_{L^{2}}+\|T_{\pa_{t}(1+\mathcal{L})^{n}}z\|_{L^{2}}\lesssim \|u\|^{5}_{H^{s}}\,.
\end{aligned}
\end{equation*}
Secondly we write
\[
\ii T_{(1+\mathcal{L})^{n}}(V*z)=
\ii V*z_{n}+\ii V*\big(T_{(1+|\x|^2)^{n}-(1+\mathcal{L})^{n}}z\big)
+\ii T_{(1+\mathcal{L})^{n}-(1+|\x|^2)^{n}}(V*z)\,.
\]
By \eqref{stimaTL}, \eqref{alien1},  and recalling \eqref{insPot} we conclude
$
\|T_{(1+\mathcal{L})^{n}}(V*z)- V*z_{n}\|_{L^{2}}
\lesssim \|u\|_{H^{s}}^{7}\,.
$
We now study the third summand in \eqref{alien11}. 
We have (see \eqref{alien12})
\[
T_{(1+\mathcal{L})^{n}}X_{\mathtt{H}_{\NLS}^{(4)}}^{+}(Z)=
T_{(1+|\x|^2)^{n}}X_{\mathtt{H}_{\NLS}^{(4)}}^{+,{\rm res}}(Z)+
T_{(1+|\x|^2){n}}X_{\mathtt{H}_{\NLS}^{(4)}}^{+,\perp}(Z)+
T_{(1+\mathcal{L})^{n}-(1+|\x|^2)^{n}}X_{\mathtt{H}_{\NLS}^{(4)}}^{+}(Z)\,.
\]
By \eqref{stimaTL},  \eqref{cubicifine},  \eqref{actionSob}, Lemma \ref{lem:trilineare} 
and using 
the estimate \eqref{pollo2}, one obtains
$$
\|T_{(1+\mathcal{L})^{n}-(1+|\x|^2)^{n}}X_{\mathtt{H}_{\NLS}^{(4)}}^{+}(Z)\|_{L^{2}}
\lesssim\|u\|^{9}_{H^{s}}\,.
$$
Recalling  \eqref{alien6} and \eqref{alien12} we write
\begin{equation}\label{alien100}
\begin{aligned}
&T_{(1+|\x|^2)^{n}}X_{\mathtt{H}_{\NLS}^{(4)}}^{+,\perp}(Z)=\mathcal{C}_1+\mathcal{C}_2+
\mathcal{C}_3\,,
\qquad \widehat{\mathcal{C}}_i(\x)=\tfrac{1}{(2\pi)^{ d}}
\sum_{\eta,\zeta\in \mathbb{Z}^{d}}\mathtt{c}_{i}(\x,\eta,\zeta)
\hat{z}(\x-\eta-\zeta)\hat{\bar{z}}(\eta)\hat{z}(\zeta)\,,\\
&\mathtt{c}_{1}(\x,\eta,\zeta):=
-2\ii\chi_{\epsilon}\Big(\tfrac{|\x-\zeta|}{\langle\x+\zeta\rangle} \Big)(1+|\zeta|^2)^{n}
\mathtt{1}_{\mathcal{R}^{c}}(\x,\eta,\zeta)
\\
&\mathtt{c}_{2}(\x,\eta,\zeta):=-
2\ii\chi_{\epsilon}\Big(\tfrac{|\x-\zeta|}{\langle\x+\zeta\rangle} \Big)\Big[(1+|\x|^2)^{n}
-(1+|\zeta|^2)^{n}\Big]
\mathtt{1}_{\mathcal{R}^{c}}(\x,\eta,\zeta)
\\
&\mathtt{c}_{3}(\x,\eta,\zeta):=
\mathtt{q}_{\mathtt{H}^{(4)}_{\NLS}}(\x,\eta,\zeta)(1+|\x|^2)^{n}
\mathtt{1}_{\mathcal{R}^{c}}(\x,\eta,\zeta)
\end{aligned}
\end{equation}
We now consider the operator $\mathcal{C}_1$ with coefficients 
$\mathtt{c}_1(\x,\eta,\zeta)$. First of all we remark that 
it can be written as $\mathcal{C}_1=M(z,\bar{z},z)$
where $M$ is a trilinear 
operator of the form \eqref{trilinearop}. Moreover, setting
\[
z_n=T_{(1+|\x|^2)^{n}}z+h_{n}\,,\quad h_n:=T_{(1+\mathcal{L})^{n}-(1+|\x|^2)^{n}}z\,,
\]
we can write
$
\mathcal{C}_1=B_{n}^{(1)}(Z)-M(z,\bar{z},h_n)\,,
$
where $B_{n}^{(1)}$ has the form \eqref{alien88} with coefficients as in \eqref{alien99}.
Using that $|\mathtt{c}_1(\x,\eta,\zeta)|\lesssim1$, Lemma \ref{lem:trilineare}
(with $m=0$)
and \eqref{stimaTL} we deduce that
$
\|M(z,\bar{z},h_n)\|_{L^{2}}\lesssim \|u\|_{H^{s}}^{9}\,.
$
Therefore this is a contribution to $R_{5,n}(U)$ satisfying \eqref{stimeRRRN}.
The discussion above implies formula \eqref{equa:VVTL} by setting
$B_{n}^{(2)}$ as the operator of the form \eqref{alien88}
with coefficients
$b_{n}^{(2)}(\x,\eta,\zeta):=\mathtt{c}_2(\x,\eta,\zeta)+\mathtt{c}_3(\x,\eta,\zeta)$.
The coefficient $\mathtt{c}_3(\x,\eta,\zeta)$ satisfies the \eqref{alien9}
by \eqref{pollo2}.
For the coefficient $\mathtt{c}_2(\x,\eta,\zeta)$
 one has to apply Lemma \ref{lem:paratri} with $\mu=m=1$ and
 $f(\x,\eta,\zeta):=((1+|\x|^2)^{n}-(1+|\zeta|^2)^{n})\langle\x\rangle^{-2n}$.
This concludes the proof.
\end{proof}

In the following lemma we prove a key cancellation 
due to the fact that the \emph{super actions} 
are prime integrals of the resonant Hamiltonian vector field 
$X_{\mathtt{H}_4}^{+,{\rm res}}(Z)$ in the same spirit of \cite{Faouplane}.
We also prove an important algebraic property 
of the operator $B_{n}^{(1)}$ in \eqref{equa:VVTL}.

\begin{lemma}\label{superazioni}
For any $n\geq 0$ we have
\begin{align}
{\rm Re}(T_{\langle\x\rangle^{n}}X_{\mathtt{H}_{\NLS}^{(4)}}^{+,{\rm res}}(Z),T_{\langle\x\rangle^{n}}z)_{L^{2}}&=0\,,
\label{superazioni1}\\
{\rm Re}(B_n^{(1)}(Z),z_{n})_{L^{2}}&=0\,,\label{cancello1}
\end{align}
where $X_{\mathtt{H}_{\NLS}^{(4)}}^{+,{\rm res}}$ is defined in Lemma \ref{equaVVnn} and
$B_{n}^{(1)}$ in \eqref{alien88}, \eqref{alien99}.
\end{lemma}
\begin{proof}
The \eqref{superazioni1} follows by Lemma \ref{cancellazioneRes}.
Let us check the \eqref{cancello1}. By an explicit computation using \eqref{scalarL}, \eqref{alien88} we get
\[
\begin{aligned}
{\rm Re}(B_n^{(1)}(Z),z_{n})_{L^{2}}&=\frac{1}{(2\pi)^d}
\sum_{\x,\eta,\zeta\in \mathbb{Z}^{d}}
b^{(1)}(\x,\eta,\zeta)
\hat{z}(\x-\eta-\zeta)\hat{\bar{z}}(\eta)\hat{z}_{n}(\zeta)\hat{\bar{z}}_{n}(-\x)\\
&+\frac{1}{(2\pi)^d}
\sum_{\x,\eta,\zeta\in \mathbb{Z}^{d}}
\ov{b^{(1)}(\x,\eta,\zeta)}
\hat{\bar{z}}(-\x+\eta+\zeta)\hat{z}(-\eta)\hat{\bar{z}}_{n}(-\zeta)\hat{z}_{n}(\x)\\
&
=\frac{1}{(2\pi)^d}\sum_{\x,\eta,\zeta\in \mathbb{Z}^{d}}
\Big[b^{(1)}(\x,\eta,\zeta)+\ov{b^{(1)}(\zeta,\zeta+\eta-\x,\x)}\Big]
\hat{z}(\x-\eta-\zeta)\hat{\bar{z}}(\eta)\hat{z}_{n}(\zeta)\hat{\bar{z}}_{n}(-\x)\,.
\end{aligned}
\]
By \eqref{alien99} we have
\[
b^{(1)}(\x,\eta,\zeta)+\ov{b^{(1)}(\zeta,\zeta+\eta-\x,\x)}=
2\ii\chi_{\epsilon}\Big(\tfrac{|\x-\zeta|}{\langle\x+\zeta\rangle}\Big)\big[
\mathtt{1}_{\mathcal{R}^{c}}(\x,\eta,\zeta)
-\mathtt{1}_{\mathcal{R}^{c}}(\zeta,\zeta+\eta-\x,\x)
\big]=0\,,
\]
where we used the form of the resonant set $\mathcal{R}$ in 
\eqref{resonantSet}. This proves the lemma.
\end{proof}

We conclude the section with the following proposition.

\begin{proposition}\label{lem:basicenergy}
Let $u(t,x)$ be a solution of \eqref{NLS} 
satisfying Hypothesis \ref{2epsilon}
and consider the function $z_n$ in \eqref{def:VVTL} 
(see also \eqref{novavarV}, \eqref{cambio1}).
Then, setting $s=2n>2d+4$ we have 
\begin{equation}\label{pallone}
\frac{1}{2^{1/4}}\|u(t)\|_{H^{s}}\leq \|z_n(t)\|_{L^{2}}\leq2^{1/4} \|u(t)\|_{H^{s}}
\end{equation}
and
\begin{equation}\label{basicenergy}
\pa_{t}\|z_{n}(t)\|_{L^{2}}^{2}=
\mathcal{B}(t)+\mathcal{B}_{>5}(t)\,,\qquad t\in[0,T)\,,
\end{equation}
where

\vspace{0.3em}
\noindent
$\bullet$ the term $\mathcal{B}(t)$ has the form
\begin{equation}\label{BBBT}
\begin{aligned}
\mathcal{B}(t)&=\frac{2}{(2\pi)^{d}}
\sum_{\x,\eta,\zeta\in \mathbb{Z}^{d}}
\langle\x\rangle^{2n}\mathtt{b}(\x,\eta,\zeta)
\hat{z}(\x-\eta-\zeta)\hat{\bar{z}}(\eta)\hat{z}(\zeta)\hat{\bar{z}}(-\x)\,,\\
\mathtt{b}(\x,\eta,\zeta)&=
b_{n}^{(2)}(\x,\eta,\zeta)+\ov{b_{n}^{(2)}(\zeta,\zeta+\eta-\x,\x)}\,,
\qquad \x,\eta,\zeta\in \mathbb{Z}^{d}\,,
\end{aligned}
\end{equation}
where $b_{n}^{(2)}(\x,\eta,\zeta)$ are the coefficients in \eqref{alien88}, \eqref{alien9};

\vspace{0.3em}
\noindent
$\bullet$ the term $\mathcal{B}_{>5}(t)$  satisfies
\begin{equation}\label{quintestimate}
|\mathcal{B}_{>5}(t)|\lesssim\|u\|_{H^{s}}^{6}\,,
\quad t\in[0,T)\,.
\end{equation}
\end{proposition}

\begin{proof}
The norm $\|z_{n}\|_{L^{2}}$ is equivalent to $\|u\|_{H^{s}}$
by using Lemma \ref{equivLemma} and Remark \ref{equivLemmaRMK}.
By using \eqref{equa:VVTL} we get
\begin{align}
\frac{1}{2}\pa_{t}\|z_n(t)\|_{L^{2}}^{2}&=
{\rm Re}(T_{\langle\x\rangle^{2n}}X_{\mathtt{H}_{\NLS}^{(4)}}^{+,{\rm res}}(Z),z_n)_{L^{2}}\label{naso0}\\
&+{\rm Re}(-\ii T_{\mathcal{L}}z_n,z_n)_{L^{2}}
+{\rm Re}(B_{n}^{(1)}(Z),z_{n})_{L^{2}}
+{\rm Re}(-\ii V*z_n,z_n)_{L^{2}}\label{naso1}
\\&+
{\rm Re}(B_{n}^{(2)}(Z),z_{n})_{L^{2}}\label{naso2}
\\&
+{\rm Re}(R_{5, n}(Z),z_{n})_{L^{2}}\label{naso3}\,.
\end{align}
Recall that $T_{\mathcal{L}}$
is self-adjoint (see item $(iv)$ in Lemma \ref{c})
and the convolution potential $V$  has real Fourier coefficients. Then 
by using also Lemma \ref{superazioni} (see \eqref{cancello1}) 
we deduce $\eqref{naso1}=0$.
Moreover by Cauchy-Schwarz inequality, estimates \eqref{stimeRRRN}, 
\eqref{equivNorme} and \eqref{alien1}
we obtain that 
the  term in \eqref{naso3} is bounded form above by 
$ \|u\|^{6}_{H^{s}}$.
Consider the terms in \eqref{naso0} and \eqref{naso2}.
Recalling \eqref{def:VVTL} and \eqref{simboLLL} we write
\[
\begin{aligned}
{\rm Re}(T_{\langle\x\rangle^{2n}}X_{\mathtt{H}_{\NLS}^{(4)}}^{+,{\rm res}}(Z),z_n)_{L^{2}}&=
{\rm Re}(T_{\langle\x\rangle^{2n}}X_{\mathtt{H}_{\NLS}^{(4)}}^{+,{\rm res}}(Z),T_{\langle\x\rangle^{2n}}z)_{L^{2}}+
{\rm Re}(T_{\langle\x\rangle^{2n}}X_{\mathtt{H}_{\NLS}^{(4)}}^{+,{\rm res}}(Z),
T_{(1+\mathcal{L})^{n}-\langle\x\rangle^{2n}}z)_{L^{2}}\,,\\
&\stackrel{\mathclap{\eqref{superazioni1}}}{=}\,\,
{\rm Re}(T_{\langle\x\rangle^{2n}}X_{\mathtt{H}_{\NLS}^{(4)}}^{+,{\rm res}}(Z),
T_{(1+\mathcal{L})^{n}-\langle\x\rangle^{2n}}z)_{L^{2}}\,.
\end{aligned}
\]
Moreover we write
\[
{\rm Re}(B_{n}^{(2)}(Z),z_{n})_{L^{2}}=
{\rm Re}(B_{n}^{(2)}(Z),T_{\langle\x\rangle^{2n}}z)_{L^{2}}+
{\rm Re}(B_{n}^{(2)}(Z),T_{(1+\mathcal{L})^{n}-\langle\x\rangle^{2n}}z)_{L^{2}}\,.
\]
Using the bound \eqref{stimaTL} in Lemma \ref{c}  to estimate the operator
$T_{(1+\mathcal{L})^{n}-\langle\x\rangle^{2n}}$, Lemma \ref{lem:trilineare} and
\eqref{alien9} to estimate the operator $B_{n}^{(2)}(Z)$,
we get
\[
|{\rm Re}(T_{\langle\x\rangle^{2n}}X_{\mathtt{H}_{\NLS}^{(4)}}^{+,{\rm res}}(Z),
T_{(1+\mathcal{L})^{n}-\langle\x\rangle^{2n}}z)_{L^{2}}|+
|{\rm Re}(B_{n}^{(2)}(Z),T_{(1+\mathcal{L})^{n}-\langle\x\rangle^{2n}}z)_{L^{2}}|
\lesssim\|u\|_{H^{s}}^{10}\,,
\]
which means that these remainders 
can be absorbed in the term $\mathcal{B}_{>5}(t)$.
 Then we set 
\[
 \mathcal{B}(t):=2{\rm Re}(B_{n}^{(2)}(Z),T_{\langle\x\rangle^{2n}}z)_{L^{2}}\,.
 \]
 Formul\ae\, \eqref{BBBT} follow by an explicit computation using 
 \eqref{alien88}, \eqref{alien9}.
\end{proof}

\subsubsection{Estimates of non-resonant terms}\label{sec:stimenonres}
In this subsection we provide estimates on the term
$\mathcal{B}(t)$ appearing in \eqref{basicenergy}.
We  state the main 
result of this section.

\begin{proposition}\label{prop:stimeEnergia}
Let  $N>0$. 
Then there is $s_0= s_0(N_0)$, where $N_0>0$ is given 
by Proposition \ref{NonresCond}, 
 such that, if Hypothesis  \ref{2epsilon}
holds with $s\geq s_0$, one has
 \begin{equation}\label{stimaeneNonresCaso2}
 \left|\int_{0}^{t}\mathcal{B}(\s)d\s\right|\lesssim
 \|u\|^{10}_{L^{\infty}H^{s}}TN +\|u\|_{L^{\infty}H^{s}}^{6}T
 +\|u\|^{4}_{L^{\infty}H^{s}}TN^{-1}+\|u\|^{4}_{L^{\infty}H^{s}}\,,
 \end{equation}
 where $\mathcal{B}(t)$ is in \eqref{BBBT}.
\end{proposition}

\noindent
We need some preliminary results.
We consider the following trilinear maps:
\begin{align}
&\mathcal{B}_{i}=\mathcal{B}_i[z_1,{z}_2,z_3]\,,\quad
\widehat{\mathcal{B}_{i}}(\x)=\frac{1}{(2\pi)^d}\sum_{\eta,\zeta\in\mathbb{Z}^{d}}
\mathtt{b}_{i}(\x,\eta,\zeta)
\hat{z}_1(\x-\eta-\zeta)\hat{{z}}_2(\eta)\hat{z}_3(\zeta)\,,\;\;i=1,2,\label{triMaps1}\\
&\mathcal{T}_{<}=\mathcal{T}_{<}[z_1,{z}_2,z_3]\,,\quad
\widehat{\mathcal{T}_{<}}(\x)=\frac{1}{(2\pi)^d}\sum_{\eta,\zeta\in\mathbb{Z}^{d}}
\mathtt{t}_{<}(\x,\eta,\zeta)
\hat{z}_1(\x-\eta-\zeta)\hat{{z}}_2(\eta)\hat{z}_3(\zeta)\,,\label{triMaps2}
\end{align}
where
\begin{align}
\mathtt{b}_{1}(\x,\eta,\zeta)&=\mathtt{b}(\x,\eta,\zeta)
\mathtt{1}_{\{\max\{|\x-\eta-\zeta|, |\eta|,|\zeta|\}\leq N\}}\,,\label{triMaps3}\\
\mathtt{b}_{2}(\x,\eta,\zeta)&=\mathtt{b}(\x,\eta,\zeta)
\mathtt{1}_{\{\max\{|\x-\eta-\zeta|, |\eta|,|\zeta|\}> N\}}\,,\label{triMaps4}\\
\mathtt{t}_{<}(\x,\eta,\zeta)&=\tfrac{-1}{\ii \omega_{\NLS}(\x,\eta,\zeta)}\mathtt{b}_1(\x,\eta,\zeta)\,,
\label{triMaps5}
\end{align}
where $\mathtt{b}(\x,\eta,\zeta)$
are the coefficients in \eqref{BBBT}, and $\omega_{\NLS}$ is the phase in \eqref{fasePhi}.
We remark that if $(\x,\eta,\zeta)\in\mathcal{R}$ (see Def. \ref{def:resonantSet}) then the coefficients 
$\mathtt{b}(\x,\eta,\zeta)$ are equal to zero
(see \eqref{BBBT},  \eqref{alien88}, \eqref{alien9}).
Therefore, since $\omega_{\NLS}$ is non-resonant (see  Proposition  \ref{NonresCond}), the 
coefficients in \eqref{triMaps5}
are well-defined. We now prove an abstract 
results on the trilinear maps introduced 
in \eqref{triMaps1}-\eqref{triMaps2}.

\begin{lemma}\label{lem:iraq}
One has that, for $s=2n>d/2+4$,
\begin{equation}\label{iraq1}
\|\mathcal{B}_2[z_1,z_2,z_3]\|_{L^{2}}\lesssim N^{-1}
\sum_{i=1}^{3} \|z_i\|_{H^{s}}\prod_{i\neq k}\|z_k\|_{H^{d/2+4+\epsilon}}\,,\qquad \forall \epsilon>0\,.
\end{equation}
\noindent
There is $s_0(N_0)>0$ ($N_0>0$ given by Proposition \ref{NonresCond})
such that for $s\geq s_0(N_0)$ one has
 \begin{equation}\label{iraq2}
\|\mathcal{T}_{<}[z_1,z_2,z_3]\|_{H^{p}}\lesssim N
\sum_{i=1}^{3} \|z_i\|_{H^{s+p-2}}\prod_{i\neq k}\|z_k\|_{H^{s_0}}\,, \quad p\in\mathbb{N}\,,
\end{equation}
 \begin{equation}\label{iraq2bis}
\|\mathcal{T}_{<}[z_1,z_2,z_3]\|_{L^{2}}\lesssim 
\sum_{i=1}^{3} \|z_i\|_{H^{s}}\prod_{i\neq k}\|z_k\|_{H^{s_0}}\,.
\end{equation}
\end{lemma}

\begin{proof}
Using \eqref{triMaps4},  \eqref{BBBT}, \eqref{alien9}
we get that
\[
\begin{aligned}
\|\mathcal{B}_2[z_1,z_2,z_3]\|_{L^{2}}^{2}&\lesssim
\sum_{\x\in\mathbb{Z}^{d}}\Big(\sum_{\eta,\zeta\in \mathbb{Z}^{d}}
|\mathtt{b}_{2}(\x,\eta,\zeta)|
|\hat{z}_1(\x-\eta-\zeta)||\hat{{z}}_2(\eta)||\hat{z}_3(\zeta)|
\Big)^{2}\\
&\lesssim
N^{-2}\sum_{\x\in\mathbb{Z}^{d}}\Big(\sum_{\eta,\zeta\in \mathbb{Z}^{d}}
\langle \x\rangle^{2n}\max_{2}\{|\x-\eta-\zeta|,|\eta|,|\zeta|\}^{4}
|\hat{z}_1(\x-\eta-\zeta)||\hat{{z}}_2(\eta)||\hat{z}_3(\zeta)|
\Big)^{2}\,.
\end{aligned}
\]
Then, by reasoning as in the proof of Lemma \ref{lem:trilineare},
one obtains the \eqref{iraq1}.
Let us prove the bound \eqref{iraq2} for $p=0$, the others are similar.
Using \eqref{triMaps5}, \eqref{lowerPhistrong},  \eqref{BBBT}, \eqref{alien9}
we have
\[
\begin{aligned}
\|\mathcal{T}_{<}[&z_1,z_2,z_3]\|_{L^{2}}^{2}\lesssim
\sum_{\x\in\mathbb{Z}^{d}}\Big(\sum_{\eta,\zeta\in \mathbb{Z}^{d}}
|\mathtt{t}_{<}(\x,\eta,\zeta)|
|\hat{z}_1(\x-\eta-\zeta)||\hat{{z}}_2(\eta)||\hat{z}_3(\zeta)|
\Big)^{2}\\
&{\lesssim_{\gamma}}N^2
\sum_{\x\in\mathbb{Z}^{d}}\left(\sum_{\eta,\zeta\in \mathbb{Z}^{d}}
\tfrac{\langle \x\rangle^{2n}\max_{2}\{|\x-\eta-\zeta|,|\eta|,|\zeta|\}^{N_0+4}}
{\max_{1}\{|\x-\eta-\zeta|,|\eta|,|\zeta|\}^{2}}
|\hat{z}_1(\x-\eta-\zeta)||\hat{{z}}_2(\eta)||\hat{z}_3(\zeta)|
\right)^{2}\,.
\end{aligned}
\]
Again,  by reasoning as in the proof of Lemma \ref{lem:trilineare},
one obtains the \eqref{iraq2}.
The \eqref{iraq2bis} follows similarly.
\end{proof}

\begin{proof}[{\bf Proof of Proposition \ref{prop:stimeEnergia}}]
By \eqref{triMaps1}, \eqref{triMaps3}, \eqref{triMaps4}, 
and recalling the definition of $\mathcal{B}$
in \eqref{BBBT}, we can write
\begin{equation}\label{EnergiaTot}
\begin{aligned}
\int_{0}^{t}\mathcal{B}(\s)d\s&=
\int_{0}^{t}(\mathcal{B}_1[z,\bar{z},z],T_{\langle\x\rangle^{2n}}z)_{L^{2}}d\s
+\int_{0}^{t}(\mathcal{B}_2[z,\bar{z},z],T_{\langle\x\rangle^{2n}}z)_{L^{2}}d\s\,.
\end{aligned}
\end{equation}
By Lemma \ref{lem:iraq} we have
\begin{equation}\label{iraq7}
\left|\int_{0}^{t}(\mathcal{B}_2[z,\bar{z},z],T_{\langle\x\rangle^{2n}}z)_{L^{2}}d\s\right|
\stackrel{\eqref{iraq1}}{\lesssim}
N^{-1}\int_{0}^{t}\|z\|_{H^{s}}^{4}d\s
\stackrel{\eqref{alien1}}{\lesssim}
N^{-1}\int_{0}^{t}\|u\|^{4}_{H^{s}}\,.
\end{equation}
Consider now the first summand in the r.h.s. of \eqref{EnergiaTot}.
We claim that we have the following identity:
\begin{equation}\label{iraq10}
\begin{aligned}
\int_{0}^{t}
(\mathcal{B}_1[z,\bar{z},z],T_{\langle\x\rangle^{2n}}z)_{L^{2}}d\s&=
\int_{0}^{t}
(\mathcal{T}_{<}[z,\bar{z},z],T_{\langle\x\rangle^{2n}}(\pa_{t}+\ii\Lambda_{\NLS})z)_{L^{2}}d\s
\\&
+\int_{0}^{t}(\mathcal{T}_{<}[(\pa_{t}+\ii\Lambda_{\NLS})z,\bar{z},z],T_{\langle\x\rangle^{2n}}z)_{L^{2}}d\s
\\&
+\int_{0}^{t}(\mathcal{T}_{<}[z,\bar{z},(\pa_{t}+\ii\Lambda_{\NLS})z],T_{\langle\x\rangle^{2n}}z)_{L^{2}}d\s
\\&
+\int_{0}^{t}(\mathcal{T}_{<}[z,\ov{(\pa_{t}+\ii\Lambda_{\NLS})z},z],T_{\langle\x\rangle^{2n}}z)_{L^{2}}d\s
+O(\|u\|^{4}_{H^{s}})\,.
\end{aligned}
\end{equation}
We use the claim, postponing its proof.
Consider the first summand in the r.h.s. of \eqref{iraq10}. Using
the self-adjointness of $T_{\langle\x\rangle^{2}}$ and the
 \eqref{eq:ZZZbis}
we write
\[
\begin{aligned}
(\mathcal{T}_{<}[z,\bar{z},z],T_{\langle\x\rangle^{2n}}(\pa_{t}+\ii\Lambda_{\NLS})z)_{L^{2}}&=
(T_{\langle\x\rangle^{2}}\mathcal{T}_{<}[z,\bar{z},z],-T_{\langle\x\rangle^{2n-2}}\ii T_{\Sigma}z)_{L^{2}}\\
&+\big(\mathcal{T}_{<}[z,\bar{z},z],T_{\langle\x\rangle^{2n}}(X_{\mathtt{H}_{\NLS}^{(4)}}^{+}(Z)
+R_{5}^{(2,+)}(U))\big)_{L^{2}}\,.
\end{aligned}
\]
We estimate the first summand in the r.h.s. by means of the  Cauchy-Schwarz inequality, the  \eqref{iraq2} with $p=2$ and the 
\eqref{stimaTL}; analogously we estimate the second summand by means of the 
 Cauchy-Schwarz inequality, \eqref{iraq2bis}, the \eqref{cubicifine}
and the 
 \eqref{stimaRRRfina}, obtaining
\[
\left|
\int_{0}^{t}
(\mathcal{T}_{<}[z,\bar{z},z],T_{\langle\x\rangle^{2n}}(\pa_{t}+\ii\Lambda_{\NLS})z)_{L^{2}}d\s
\right|\leq \int_{0}^{t}\|u(\s)\|^{10}_{H^{s}}N+\|u(\s)\|_{H^{s}}^{6}
d\s\,.
\]
The other terms in \eqref{iraq10} are estimated in a similar way. We eventually obtain the \eqref{stimaeneNonresCaso2}.

\noindent
We now prove the claim \eqref{iraq10}.
Recalling \eqref{eq:ZZZbis} we have that
\begin{equation*}
\pa_{t}\hat{z}(\x)=-\ii \Lambda_{\NLS}(\x)\hat{z}(\x)+\widehat{\mathcal{Q}}(\x)\,,\qquad 
\x\in \mathbb{Z}^{d}\,,
\qquad
\mathcal{Q}:=-\ii T_{\Sigma}z+X_{\mathtt{H}_{\NLS}^{(4)}}^{+}(Z)+R_{5}^{(2,+)}(U)\,.
\end{equation*}
We define $\hat{g}(\x):=e^{\ii t\Lambda_{\NLS}(\x)}\hat{z}(\x)$, $\forall\x\in\mathbb{Z}^{d}$. 
One can note that $\hat{g}(\x)$ satisfies
\begin{equation}\label{eq:ZZZbisFou}
\pa_{t}\hat{g}(\x)=e^{\ii t\Lambda_{\NLS}(\x)}\widehat{\mathcal{Q}}(\x)=e^{\ii t\Lambda_{\NLS}(\x)}(\pa_{t}+\ii\Lambda_{\NLS})\hat z(\xi) \,,\qquad 
\forall\,\x\in\mathbb{Z}^{d}\,.
\end{equation}
According to this notation and using \eqref{triMaps1} and \eqref{fasePhi} 
we have
\[
\int_{0}^{t}
(\mathcal{B}_1[z,\bar{z},z],T_{\langle\x\rangle^{2n}}z)_{L^{2}}d\s=
\int_{0}^{t}\sum_{\x,\eta,\zeta\in\mathbb{Z}^{d}}\tfrac{1}{(2\pi)^d}
\mathtt{b}_{1}(\x,\eta,\zeta)e^{-\ii \s\omega_{\NLS}(\x,\eta,\zeta)}
\hat{g}(\x-\eta-\zeta)\hat{\bar{g}}(\eta)\hat{g}(\zeta)\hat{\bar{g}}(-\x)\langle\x\rangle^{2n}d\s\,.
\]
By integrating by parts in $\s$ and using \eqref{eq:ZZZbisFou} 
one gets the \eqref{iraq10} with
\[
O(\|u\|^{4}_{H^{s}})=
(\mathcal{T}_{<}[z(t),\bar{z}(t),z(t)],T_{\langle\x\rangle^{2n}}z(t))_{L^{2}}-
(\mathcal{T}_{<}[z(0),\bar{z}(0),z(0)],T_{\langle\x\rangle^{2n}}z(0))_{L^{2}}\,.
\]
The remainder above is bounded from above by $\|u\|^{4}_{L^{\infty}H^{s}}$
using Cauchy-Schwarz and the \eqref{iraq2bis}.
\end{proof}

\subsection{Estimates for the KG}\label{stimettineKG}
In this section we provide \emph{a priori} energy estimates on the variable 
$Z$ solving  \eqref{eq:ZZZKG}. This  implies similar estimates on the solution $U$
of the system \eqref{QLNLS444KG} thanks to the equivalence \eqref{alienKG1}.
In subsection \ref{sec:55KG} we introduce an equivalent energy norm
and we provide a first energy inequality. This is the content of 
Proposition \ref{lem:basicenergyKG}.
Then in subsection \ref{sec:55KG5KG} we give improved  bounds on the
non-resonant terms.
\subsubsection{First energy inequality}\label{sec:55KG}
We recall that the system \eqref{eq:ZZZKG} is diagonal up to smoothing terms plus some
higher degree of homogeneity remainder. 
Hence, for simplicity, we pass to the scalar equation
\begin{equation}\label{eq:ZZZKGscal}
\pa_{t}z+\ii \Lambda_{\KG} z=-\ii \opbw\big(\widetilde{a}_{2}^{+}(x,\x)\Lambda_{\KG}(\x)\big)z
+X_{\mathtt{H}^{(4)}_{\KG}}^{+}(Z)+{R}^{(2,+)}_{4}(u)
\end{equation}
where (recall \eqref{cubicifineKG})
$
X_{\mathtt{H}^{(4)}_{\KG}}^{+}(Z)=-\ii \opbw(a_0(x,\x))z+Q_{\mathtt{H}^{(4)}_{\KG}}^{+}(Z)\,.
$
For $n\in \mathbb{R}$ we define 
\begin{equation}\label{def:VVTLKG}
z_{n}:=\langle D\rangle^{n}z\,,\qquad 
Z_{n}=\vect{z_n}{\ov{z_{n}}}
:=\uno\langle D\rangle^{n} Z\,,\quad Z=\vect{z}{\bar{z}}\,. 
\end{equation}
We have the following.
\begin{lemma}\label{equaVVnnKG}
Fix $n:=n(d)\gg1$ large enough  and recall \eqref{eq:ZZZKGscal}.
One has that the function $z_n$ defined in \eqref{def:VVTLKG}
solves the problem
\begin{equation}\label{equa:VVTLKG}
\pa_{t}z_n=-\ii \opbw\big((1+\widetilde{a}_{2}^{+}(x,\x))\Lambda_{\KG}(\x)\big)z_{n}
+\langle D\rangle^{n}X^{+,{\rm res}}_{\mathtt{H}^{(4)}_{\KG}}(Z)
+B_{n}^{(1)}(Z)+B_{n}^{(2)}(Z)
+R_{4,n}(U)\,,
\end{equation}
where the resonant vector field $X^{+,{\rm res}}_{\mathtt{H}^{(4)}_{\KG}}$ is defined as 
in Def. \ref{def:resonantSet} (see also Rmk. \ref{rmkalg}), 
the cubic terms $B_{n}^{(i)}$, $i=1,2$, have the form
\begin{align}
\widehat{B_{n}^{(1)}(Z)}(\x)&=\frac{1}{(2\pi)^{d}}
\sum_{\substack{\s_1,\s_2\in\{\pm\} \\ \eta,\zeta\in \mathbb{Z}^{d}}}
b_1^{\s_1,\s_2}(\x,\eta,\zeta)
\hat{z^{\s_1}}(\x-\eta-\zeta)\hat{z^{\s_2}}(\eta)\hat{z}_{n}(\zeta)\,,\label{alienKG888}\\
\widehat{B_{n}^{(2)}(Z)}(\x)&=\frac{1}{(2\pi)^{ d}}
\sum_{\substack{\s_1,\s_2,\s_3\in\{\pm\}\\\eta,\zeta\in \mathbb{Z}^{d}}}
b_{2,n}^{\s_1,\s_2,\s_3}(\x,\eta,\zeta)
\hat{z^{\s_1}}(\x-\eta-\zeta)\hat{z^{\s_2}}(\eta)\hat{z^{\s_3}}(\zeta)\,,
\label{alienKG889}
\end{align}
with  (recall Rmk. \ref{strutturaA0KG})
\begin{align}
b_1^{\s_1,\s_2}(\x,\eta,\zeta)&:=-\ii a_0^{\s_1,\s_2}\big(\x-\zeta,\eta,\tfrac{\x+\zeta}{2}\big) \chi_{\epsilon}\Big(\tfrac{|\x-\zeta|}{\langle\x+\zeta\rangle} \Big)
\mathtt{1}_{\mathcal{R}^{c}}(\x,\eta,\zeta)\,, \label{alienKG99}\\
|b_{2,n}^{\s_1,\s_2,\s_3}(\x,\eta,\zeta)|&\lesssim 
\tfrac{\langle \x\rangle^{n}\max_{2}\{|\x-\eta-\zeta|,|\eta|,|\zeta|\}^{\mu}}
{\max_1\{\langle\x-\eta-\zeta\rangle,\langle\eta\rangle,\langle\zeta\rangle\}}
\mathtt{1}_{\mathcal{R}^{c}}(\x,\eta,\zeta)\,,
\label{alienKG9}
\end{align}
for some $\mu>1$. The remainder satisfies
\begin{equation}\label{stimeRRRNKG}
\|R_{4,n}(U)\|_{L^{2}}\lesssim \|u\|^{4}_{H^{n}}\,.
\end{equation}
\end{lemma}

\begin{proof}
Recalling the definition of resonant vector fields in Def. \ref{def:resonantSet}
we set
\begin{equation}\label{oldmanKG4}
X^{+,\perp}_{\mathtt{H}^{(4)}_{\KG}}(Z):=
X^{+}_{\mathtt{H}^{(4)}_{\KG}}(Z)-X^{+,{\rm res}}_{\mathtt{H}^{(4)}_{\KG}}(Z)\,,
\end{equation}
which represents the non resonant terms in the cubic 
vector field of \eqref{eq:ZZZKGscal}. By differentiating in $t$ the \eqref{def:VVTLKG}
and using the \eqref{eq:ZZZKGscal}
we get
\begin{align}
\pa_{t}z_{n}&=-\ii \opbw\big((1+\widetilde{a}_{2}^{+}(x,\x))\Lambda_{\KG}(\x)\big)z_{n}
+\langle D\rangle^{n}
X^{+,{\rm res}}_{\mathtt{H}^{(4)}_{\KG}}(Z)
\nonumber\\
&\quad-\ii \big[\langle D\rangle^{n}
,  \opbw\big((1+\widetilde{a}_{2}^{+}(x,\x))
\Lambda_{\KG}(\x)\big)\big]z\label{oldmanKG1}\\
&\quad+\langle D\rangle^{n}
X^{+,\perp}_{\mathtt{H}^{(4)}_{\KG}}(Z)\label{oldmanKG2}\\
&\quad+\langle D\rangle^{n}
{R}^{(2,+)}_{4}(u)\,.\label{oldmanKG3}
\end{align}
We analyse each summand above separately. By estimate \eqref{stimaRRRfinaKG} we deduce
 $\|\eqref{oldmanKG3}\|_{L^{2}}\lesssim \|u\|^{4}_{H^{n}}$.
Let us now consider the commutator term  in \eqref{oldmanKG1}.
By Lemma \ref{azioneSimboo}, Proposition  \ref{prop:compo} and the estimate on the 
semi-norm of the  symbol $\widetilde{a}_{2}^{+}(x,\x)$ in \eqref{mission1KG},
we obtain that
$
\|\eqref{oldmanKG1}
\|_{L^{2}}\lesssim \|u\|^{3}_{H^{n}}\|z\|_{H^{n}}
{\lesssim}\|u\|^{4}_{H^{n}}\,,
$ we have used also the \eqref{alienKG1}.
The term in \eqref{oldmanKG2} is the most delicate.
By \eqref{cubicifineKG} and \eqref{oldmanKG4} (recall also Rmk. \ref{strutturaA0KG} 
and \eqref{quantiWeyl})
\begin{equation}\label{alienKG100}
\begin{aligned}
&\langle D\rangle^nX_{\mathtt{H}^{(4)}_{\KG}}^{+,\perp}(Z)=B_{n}^{(1)}(Z)+
\mathcal{C}_1+\mathcal{C}_2\,,
\end{aligned}
\end{equation}
with $B_{n}^{(1)}(Z)$ as in \eqref{alienKG888} and coefficients as in \eqref{alienKG99},
the term $\mathcal{C}_1$ has the form
\begin{equation}\label{alienKG101}
\begin{aligned}
\widehat{\mathcal{C}}_1(\x)&=
\tfrac{1}{(2\pi)^{d}}
\sum_{\substack{\s_1,\s_2\in\{\pm\} \\ \eta,\zeta\in \mathbb{Z}^{d}}}
\mathtt{c}_1^{\s_1,\s_2}(\x,\eta,\zeta)
\hat{z^{\s_1}}(\x-\eta-\zeta)\hat{z^{\s_2}}(\eta)\hat{z}(\zeta)\,,\\
\mathtt{c}_1^{\s_1,\s_2}(\x,\eta,\zeta)&=
-\ii a_0^{\s_1,\s_2}\big(\x-\zeta,\eta,\tfrac{\x+\zeta}{2}\big) \chi_{\epsilon}\Big(\tfrac{|\x-\zeta|}{\langle\x+\zeta\rangle} \Big)\Big[\langle \xi\rangle^n-\langle \zeta\rangle^n\Big]
\mathtt{1}_{\mathcal{R}^{c}}(\x,\eta,\zeta)\,,
\end{aligned}
\end{equation}
and the term $\mathcal{C}_2$ has the form \eqref{alienKG889} 
with coefficients (see \eqref{restoQ3GGGKG})
\begin{equation}\label{alienKG102}
\begin{aligned}
&\mathtt{c}_{2}^{\s_1,\s_2,\s_3}(\x,\eta,\zeta):=
\mathtt{q}_{\mathtt{H}^{(4)}_{\KG}}^{\s_1,\s_2,\s_3}(\x,\eta,\zeta)
\langle \x\rangle^n
\mathtt{1}_{\mathcal{R}^{c}}(\x,\eta,\zeta)\,.
\end{aligned}
\end{equation}
In order to conclude the proof we need to show that the coefficients in \eqref{alienKG101},
\eqref{alienKG102} satisfy the bound \eqref{alienKG9}.
This is true for the coefficients in \eqref{alienKG102} thanks 
to the bound \eqref{restoQ32GGGKG}.
Moreover notice that
\[
|\langle \x\rangle^n-\langle \zeta\rangle^n|
\lesssim|\x-\zeta|\max\{\langle \x\rangle, \langle\zeta\rangle\}^{n-1}\,.
\]
Then the coefficients in \eqref{alienKG101}
 satisfy \eqref{alienKG9} by using Remark \ref{strutturaA0KG}
 and Lemma \ref{lem:paratri}.
\end{proof}

\begin{remark}\label{semilin4}
In view of Remarks \ref{semilin1}, \ref{semilin2}, \ref{semilin3}
if \eqref{KG} is semi-linear then the symbol $\tilde{a}^{+}_{2}$ in \eqref{equa:VVTLKG}
is equal to zero, the coefficients $b_{2,n}^{\s_1,\s_2,\s_3}(\x,\eta,\zeta)$ 
in \eqref{alienKG889} satisfies the bound \eqref{alienKG9}
the the better denominator
$\max_1\{\langle\x-\eta-\zeta\rangle,\langle\eta\rangle,\langle\zeta\rangle\}^{2}$.
\end{remark}

In view of Lemma \ref{equaVVnnKG} we deduce the following.

\begin{proposition}\label{lem:basicenergyKG}
Let $\psi(t,x)$ be a solution of \eqref{KG} 
satisfying Hypothesis \ref{2epsilonKG}
and consider the function $z_n$ in \eqref{def:VVTLKG} 
(see also \eqref{varZZZKG}, \eqref{cambio1KG}).
Then, setting $s=n=n(d)\gg1$ 
we have $\|z_n\|_{L^{2}}\sim \|\psi\|_{H^{s+1/2}}+\|\dot{\psi}\|_{H^{s-1/2}}$
and
\begin{equation}\label{basicenergyKG}
\pa_{t}\|z_{n}(t)\|_{L^{2}}^{2}=
\mathcal{B}(t)+\mathcal{B}_{>4}(t)\,,\qquad t\in[0,T)\,,
\end{equation}
where

\vspace{0.3em}
\noindent
$\bullet$ the term $\mathcal{B}(t)$ has the form
\begin{equation}\label{BBBTKG}
\begin{aligned}
\mathcal{B}(t)&=
\sum_{\substack{\s_1,\s_2,\s_3\in\{\pm\} \\\x,\eta,\zeta\in \mathbb{Z}^{d}}}
\langle \xi\rangle^{2n}\mathtt{b}^{\s_1,\s_2,\s_3}(\x,\eta,\zeta)
\hat{z^{\s_1}}(\x-\eta-\zeta)\hat{z^{\s_2}}(\eta)\hat{z^{\s_3}}(\zeta)\hat{\bar{z}}(-\x)\,,
\end{aligned}
\end{equation}
where $\mathtt{b}^{\s_1,\s_2,\s_3}(\x,\eta,\zeta)\in \mathbb{C}$
satisfy, for $\x,\eta,\zeta\in \mathbb{Z}^{d}$,
\begin{equation}\label{alienKG999}
|\mathtt{b}^{\s_1,\s_2,\s_3}(\x,\eta,\zeta)|\lesssim
\tfrac{\max_{2}\{|\x-\eta-\zeta|,|\eta|,|\zeta|\}^{\mu}}
{\max_1\{\langle\x-\eta-\zeta\rangle,\langle\eta\rangle,\langle\zeta\rangle\}}\mathtt{1}_{\mathcal{R}^{c}}(\x,\eta,\zeta)
\end{equation}
for some $\mu>1$;

\vspace{0.3em}
\noindent
$\bullet$ the term $\mathcal{B}_{>5}(t)$ 
satisfies
\begin{equation}\label{quintestimateKG}
|\mathcal{B}_{>4}(t)|\lesssim\|u\|_{H^{s}}^{5}\,,
\quad t\in[0,T)\,.
\end{equation}
\end{proposition}

\begin{proof}
The equivalence between $\|z_n\|_{L^2}$ and $\|\psi\|_{H^{s+1/2}}+\|\dot{\psi}\|_{H^{s-1/2}}$ follows by Remarks \ref{equivLemmaRMKKG} and \ref{WWUUUKG}.
By using \eqref{equa:VVTLKG} we get
\begin{align}
\tfrac{1}{2}\pa_{t}\|z_n(t)\|_{L^{2}}^{2}&=
{\rm Re}\big(-\ii   \opbw\big((1+\widetilde{a}_{2}^{+}(x,\x))\Lambda_{\KG}(\x)\big)z_{n} ,z_n\big)_{L^{2}}\label{nasoKG00}
\\
&+{\rm Re}(\langle D\rangle^nX_{\mathtt{H}^{(4)}_{\KG}}^{+,{\rm res}}(Z),z_n)_{L^{2}}\label{nasoKG0}\\
&
+{\rm Re}(B_{n}^{(1)}(Z),z_{n})_{L^{2}}\label{nasoKG1}
\\&+
{\rm Re}(B_{n}^{(2)}(Z),z_{n})_{L^{2}}\label{nasoKG2}
\\&
+{\rm Re}(R_{4,n}(Z),z_{n})_{L^{2}}\label{nasoKG3}\,.
\end{align}
By \eqref{nuovadiagKG}, \eqref{A2tildeKG} and \eqref{simboa2KG} we have that the symbol 
$(1+\widetilde{a}_{2}^{+}(x,\x))\Lambda_{\KG}(\x)$ is real-valued. Hence the operator
$\ii   \opbw\big((1+\widetilde{a}_{2}^{+}(x,\x))\Lambda_{\KG}(\x)\big)$
is skew-self-adjoint. We deduce \eqref{nasoKG00}$\equiv0$.
By Lemma \ref{cancellazioneRes} (see also Remark \ref{rmkalg})
we also have that \eqref{nasoKG0}$\equiv0$.
We also have that
$
\eqref{nasoKG1}\equiv0\,, 
$ to see this one can reason as done in the proof of Prop. \ref{superazioni}, by using Remark \ref{strutturaA0KG}, in particular \eqref{mazzo2KG}.
By formula \eqref{alienKG889} and estimates \eqref{alienKG9} we have 
that the term in \eqref{nasoKG2} has the form \eqref{BBBTKG} with coefficients
satisfying \eqref{alienKG999}.
By Cauchy-Schwarz inequality and estimate \eqref{stimeRRRNKG}
we get that the term in \eqref{nasoKG3} satisfies the bound \eqref{quintestimateKG}.
\end{proof}

\begin{remark}\label{semilin5}
In view of Remark \ref{semilin4}, if \eqref{KG} is semi-linear, then the 
coefficients $\mathtt{b}^{\s_1,\s_2,\s_3}(\x,\eta,\zeta)$ of the energy in \eqref{BBBTKG}
satisfy the bound \eqref{alienKG999} with the better denominator 
$\max_1\{\langle\x-\eta-\zeta\rangle,\langle\eta\rangle,\langle\zeta\rangle\}^{2}$.
\end{remark}

\subsubsection{Estimates of non-resonant terms}\label{sec:55KG5KG}
In Proposition \ref{lem:basicenergyKG} we provided a precise structure of
the term $\mathcal{B}(t)$ of degree $4$ in \eqref{basicenergyKG}.
In this section we show that, actually,
$\mathcal{B}(t)$ satisfies better bounds with respect to a general
quartic multilinear maps by using that it is \emph{non-resonant}.
We state the main result of this section.

\begin{proposition}\label{prop:stimeEnergiaKG}
Let  $N>0$ and let $\beta$ be as in Proposition \ref{NonresCondKG}.
Then there is $s_0= s_0(N_0)$, where $N_0>0$ is given 
by Proposition \ref{NonresCondKG}, 
 such that, if Hypothesis  \ref{2epsilonKG}
holds with $s\geq s_0$, one has
 \begin{equation}\label{stimaeneNonresCaso2KG}
 \left|\int_{0}^{t}\mathcal{B}(\s)d\s\right|\lesssim
 \|u\|^{6}_{L^{\infty}H^{s}}TN^{\beta-1} +\|u\|_{L^{\infty}H^{s}}^{7}N^{\beta}T
 +\|u\|^{4}_{L^{\infty}H^{s}}TN^{-1}+N^{\beta-1}\|u\|^{4}_{L^{\infty}H^{s}}\,,
 \end{equation}
 where $\mathcal{B}(t)$ is in \eqref{BBBTKG}.
\end{proposition}

\noindent
We firstly introduce some notation. Let $\vec{\s}:=(\s_1,\s_2,\s_3)\in\{\pm\}^{3}$ and 
consider the following trilinear maps:
\begin{align}
&{\mathcal{B}}^{\vec{\s}}_{i}={\mathcal{B}}^{\vec{\s}}_{i} [z_1,{z}_2,z_3]\,,\quad
\widehat{\mathcal{B}}^{\vec{\s}}_{i}(\x)=
\tfrac{1}{(2\pi)^d}\sum_{\eta,\zeta\in\mathbb{Z}^{d}}
\langle\x\rangle^s
\mathtt{b}^{\vec{\s}}_{i}(\x,\eta,\zeta)
\hat{z^{\s_1}_1}(\x-\eta-\zeta)\hat{z^{\s_2}_{2}}(\eta)
\hat{z^{\s_3}_{3}}(\zeta)\,,\label{triMapsKG1}\\
&\mathcal{T}^{\vec{\s}}_{<}=\mathcal{T}^{\vec{\s}}_{<}[z_1,{z}_2,z_3]\,,\quad
\widehat{\mathcal{T}}^{\vec{\s}}_{<}(\x)=\tfrac{1}{(2\pi)^d}
\sum_{\eta,\zeta\in\mathbb{Z}^{d}}
\langle\x\rangle^s
\mathtt{t}^{\vec{\s}}_{<}(\x,\eta,\zeta)
\hat{z_{1}^{\s_1}}(\x-\eta-\zeta)\hat{z_{2}^{\s_2}}(\eta)\hat{z^{\s_3}_{3}}(\zeta)\,,\label{triMapsKG2}
\end{align}
where
\begin{align}
\mathtt{b}_{1}^{\vec{\s}}(\x,\eta,\zeta)&=\mathtt{b}^{\s_1,\s_2,\s_3}(\x,\eta,\zeta)
\mathtt{1}_{\{\max\{|\x-\eta-\zeta|, |\eta|,|\zeta|\}\leq N\}}\,,\label{triMapsKG3}\\
\mathtt{b}^{\vec{\s}}_{2}(\x,\eta,\zeta)&=\mathtt{b}^{\s_1,\s_2,\s_3}(\x,\eta,\zeta)
\mathtt{1}_{\{\max\{|\x-\eta-\zeta|, |\eta|,|\zeta|\}> N\}}\,,\label{triMapsKG4}\\
\mathtt{t}^{\vec{\s}}_{<}(\x,\eta,\zeta)&=\tfrac{-1}{\ii \omega_{\KG}^{\vec{\s}}(\x,\eta,\zeta)}
\mathtt{b}^{\vec{\s}}_1
(\x,\eta,\zeta)\,,
\label{triMapsKG5}
\end{align}
where $\mathtt{b}^{\s_1,\s_2,\s_3}(\x,\eta,\zeta)$
are the coefficients in \eqref{BBBTKG}, and $\omega_{\KG}^{\vec{\s}}$ is the phase in \eqref{fasePhiKG}.
We remark that if $(\x,\eta,\zeta)\in\mathcal{R}$ (see Def. \ref{def:resonantSet}) 
then the coefficients 
$\mathtt{b}(\x,\eta,\zeta)$ are equal to zero
(see \eqref{BBBTKG},  \eqref{alienKG889}, \eqref{alienKG9}).
Therefore, since $\omega^{\vec{\s}}_{\KG}$ is non-resonant (see  Proposition  \ref{NonresCondKG}), the 
coefficients in \eqref{triMapsKG5}
are well-defined. We now state an abstract 
result on the trilinear maps introduced 
in \eqref{triMapsKG1}-\eqref{triMapsKG2}.

\begin{lemma}\label{lem:iraqKG}
Let $\mu>1$ as in \eqref{alienKG999}.
One has that, for $s>d/2+\mu$,
\begin{equation}\label{iraqKG1}
\|\mathcal{B}^{\vec{\s}}_2[z_1,z_2,z_3]\|_{L^{2}}\lesssim N^{-1}
\sum_{i=1}^{3} \|z_i\|_{H^{s}}\prod_{i\neq k}\|z_k\|_{H^{d/2+\mu+\epsilon}}\,,
\end{equation}
for any $\vec{\s}\in\{\pm\}^{3}$ and any $\epsilon>0$.
\noindent
There is $s_0(N_0)>0$ ($N_0>0$ given by Proposition \ref{NonresCondKG})
such that for $s\geq s_0(N_0)$ one has
 \begin{equation}\label{iraqKG2}
\|\mathcal{T}_{<}[z_1,z_2,z_3]\|_{H^{p}}\lesssim N^{\beta}
\sum_{i=1}^{3} \|z_i\|_{H^{s+p-1}}\prod_{i\neq k}\|z_k\|_{H^{s_0}}\,, \quad p\in\mathbb{N}\,,
\end{equation}
 \begin{equation}\label{iraqKG2bis}
\|\mathcal{T}_{<}[z_1,z_2,z_3]\|_{L^{2}}\lesssim N^{\beta-1}
\sum_{i=1}^{3} \|z_i\|_{H^{s}}\prod_{i\neq k}\|z_k\|_{H^{s_0}}\,. 
\end{equation}
where $\beta$ is defined in Proposition \ref{NonresCondKG}.
\end{lemma}

\begin{proof}
The proof is similar to the one of Lemma \ref{lem:iraq}.
One has to use  Proposition \ref{NonresCondKG} instead of  Proposition \ref{NonresCond}
to estimate the small divisors.
\end{proof}
\begin{remark}\label{semilin6}
In view of Remark \ref{semilin5}, if \eqref{KG} is semi-linear we may improve \eqref{iraqKG1} and  \eqref{iraqKG2bis} with
\begin{equation*}
\begin{aligned}
\|\mathcal{B}^{\vec{\s}}_2[z_1,z_2,z_3]\|_{L^{2}}&\lesssim N^{-2}
\sum_{i=1}^{3} \|z_i\|_{H^{s}}\prod_{i\neq k}\|z_k\|_{H^{d/2+\mu+\epsilon}}\,,
\\
\|\mathcal{T}_{<}[z_1,z_2,z_3]\|_{L^{2}}&\lesssim N^{\beta-2}
\sum_{i=1}^{3} \|z_i\|_{H^{s}}\prod_{i\neq k}\|z_k\|_{H^{s_0}}\,. 
\end{aligned}
\end{equation*}
\end{remark}

We are now in position to prove the main Proposition \ref{prop:stimeEnergiaKG}.

\begin{proof}[{\bf Proof of Proposition \ref{prop:stimeEnergiaKG}}]
By \eqref{triMapsKG1}, \eqref{triMapsKG3}, \eqref{triMapsKG4}, 
and recalling the definition of $\mathcal{B}$
in \eqref{BBBTKG}, we can write
\begin{equation}\label{EnergiaTotKG}
\begin{aligned}
\int_{0}^{t}\mathcal{B}(\tau)d\tau&=\sum_{\vec{\s}\in\{\pm\}^{3}}
\int_{0}^{t}(\mathcal{B}^{\vec{\s}}_1[z,{z},z],\langle D\rangle^sz)_{L^{2}}d\tau
+\sum_{\vec{\s}\in\{\pm\}^{3}}\int_{0}^{t}
(\mathcal{B}^{\vec{\s}}_2[z,{z},z],\langle D\rangle^sz)_{L^{2}}d\tau\,.
\end{aligned}
\end{equation}
By Lemma \ref{lem:iraqKG} we have
\begin{equation}\label{iraqKG7}
\left|\int_{0}^{t}(\mathcal{B}^{\vec{\s}}_2[z,{z},z],\langle D\rangle^sz)_{L^{2}}d\s\right|
\stackrel{\eqref{iraqKG1}}{\lesssim}
N^{-1}\int_{0}^{t}\|z\|_{H^{s}}^{4}d\tau
\stackrel{\eqref{alienKG1}}{\lesssim}
N^{-1}\int_{0}^{t}\|u\|^{4}_{H^{s}}d\tau\,.
\end{equation}
Consider now the first summand in the r.h.s. of \eqref{EnergiaTotKG}.
Integrating by parts as done in the proof of Prop. \ref{prop:stimeEnergia}   
we have
\begin{equation}\label{iraqKG10}
\begin{aligned}
\int_{0}^{t}
(\mathcal{B}^{\vec{\s}}_1[z,\bar{z},z],\langle D\rangle^sz)_{L^{2}}d\tau&=
\int_{0}^{t}
(\mathcal{T}^{\vec{\s}}_{<}[z,{z},z],\langle D\rangle^s(\pa_{t}+\ii\Lambda_{\KG})z)_{L^{2}}d\tau
\\&
+\int_{0}^{t}(\mathcal{T}^{\vec{\s}}_{<}[(\pa_{t}+\ii\Lambda_{\KG})z,{z},z],
\langle D\rangle^sz)_{L^{2}}d\tau
\\&
+\int_{0}^{t}(\mathcal{T}^{\vec{\s}}_{<}[z,\bar{z},(\pa_{t}+\ii\Lambda_{\KG})z],
\langle D\rangle^sz)_{L^{2}}d\tau
\\&
+\int_{0}^{t}(\mathcal{T}^{\vec{\s}}_{<}[z,(\pa_{t}+\ii\Lambda_{\KG})z,z],
\langle D\rangle^sz)_{L^{2}}d\tau
+R\,,
\end{aligned}
\end{equation}
where 
\[
R=
(\mathcal{T}^{\vec{\s}}_{<}[z(t),{z}(t),z(t)],\langle D\rangle^sz(t))_{L^{2}}-
(\mathcal{T}^{\vec{\s}}_{<}[z(0),{z}(0),z(0)],\langle D\rangle^sz(0))_{L^{2}}\,.
\]
The remainder $R$ above is bounded from above by $N^{\beta}\|u\|^{4}_{L^{\infty}H^{s}}$
using Cauchy-Schwarz and the \eqref{iraqKG2}.
Let us now consider the first summand in the r.h.s. of \eqref{iraqKG10}.
Using that the operator $\langle D\rangle$
is self-adjoint and recalling the equation \eqref{eq:ZZZKGscal} we have
\begin{align}
(\mathcal{T}^{\vec{\s}}_{<}[z,{z},z],\langle D\rangle^s(\pa_{t}+\ii\Lambda_{\KG})z)_{L^{2}}&=
(\langle D\rangle\mathcal{T}^{\vec{\s}}_{<}[z,{z},z],\langle D\rangle^{s-1}(\pa_{t}+\ii\Lambda_{\KG})z)_{L^{2}}
\nonumber\\
&=
(\langle D\rangle\mathcal{T}^{\vec{\s}}_{<}[z,{z},z],\langle D\rangle^{s-1}
 \opbw\big(-\ii\widetilde{a}_{2}^{+}(x,\x)\Lambda_{\KG}(\x)\big)z)_{L^{2}}\label{lemansKG1}\\
 &+(\mathcal{T}^{\vec{\s}}_{<}[z,{z},z],\langle D\rangle^s
 \big(X_{\mathtt{H}_{\KG}^{(4)}}^{+}(Z)+{R}^{(2,+)}_{4}(u)\big))_{L^{2}}
 \label{lemansKG2}\,.
\end{align}
By Cauchy-Schwarz inequality, estimate \eqref{iraqKG2} with $p=1$, 
estimate \eqref{mission1KG} on the semi-norm of the symbol $\widetilde{a}_2^{+}(x,\x)$
Lemma \ref{azioneSimboo} and the equivalence \eqref{alienKG1},
we get 
$
|\eqref{lemansKG1}|\lesssim \|u\|_{H^{s}}^{7}N^{\beta}\,.
$
Consider the term in \eqref{lemansKG2}. First of all notice that,
by \eqref{calA00KG} and Lemma \ref{azioneSimboo}, and by \eqref{restoQ32GGGKG}
and Lemma \ref{lem:trilineare}, the field  $X_{\mathtt{H}^{(4)}_{\KG}}(Z)$ 
in  \eqref{cubicifineKG} 
satisfies
the same estimates \eqref{stimaGGKG} as the field $X_{\mathcal{H}^{(4)}_{\KG}}$.
Therefore, using \eqref{iraqKG2bis} and  \eqref{stimaRRRfinaKG}, we obtain 
$
|\eqref{lemansKG2}|\lesssim \|u\|^{6}_{H^{s}}N^{\beta-1}\,.
$ 
Using that (see Hyp. \ref{2epsilonKG}) $\|u\|_{H^{s}}\ll1$,
we conclude that the first summand in the r.h.s. of \eqref{iraqKG10} 
is bounded from above by 
$N^{\beta}\int_{0}^{t}\|u(\tau)\|^{7}d\tau+N^{\beta-1}\int_{0}^{t}\|u(\tau)\|^{6}d\tau$.
The other terms in \eqref{iraqKG10} 
are estimated in a similar way. We eventually 
obtain the \eqref{stimaeneNonresCaso2KG}.
\end{proof}

\begin{remark}\label{semilin7}
In view of Remarks \ref{semilin1}, \ref{semilin2}, \ref{semilin3}, \ref{semilin4}, 
\ref{semilin5} and \ref{semilin6}, if \eqref{KG} is semi-linear 
we have the better (w.r.t. \eqref{stimaeneNonresCaso2KG}) 
estimate 
 \begin{equation}\label{stimaeneNonresCaso2KGsemilin}
 \left|\int_{0}^{t}\mathcal{B}(\s)d\s\right|\lesssim
 \|u\|^{6}_{L^{\infty}H^{s}}TN^{\beta-2} 
 +\|u\|^{4}_{L^{\infty}H^{s}}TN^{-2}+N^{\beta-2}\|u\|^{4}_{L^{\infty}H^{s}}\,.
 \end{equation}
\end{remark}

\section{Proof of the main results}

In this section we conclude the proof of our main theorems.

\vspace{0.2em}
\noindent
{\bf Proof of Theorem \ref{main:NLS}}.
Consider \eqref{NLS} and let $u_0$ be as in the statement of Theorem \ref{main:NLS}. 
By the result in \cite{Feola-Iandoli-local-tori}
we have that
 there is $T>0$ and a unique solution $u(t,x)$ of $\eqref{NLS}$ with $V\equiv0$ 
 such that Hypothesis \ref{2epsilon}
is satisfied. To recover the  result when $V\neq0$
one can argue as done in \cite{Feola-Iandoli-Loc}.
Consider a potential $V$ as in \eqref{insPot} with 
$\vec{x}\in \mathcal{O}\setminus\mathcal{N}$
with $\mathcal{N}$ is the zero measure set given in Proposition \ref{NonresCond}.
We claim that we have the following \emph{a priori} estimate:
fix any  $0<N$, then 
for any $t\in [0,T)$, with $T$ as in Hyp. \ref{2epsilon}, one has
\begin{equation}\label{claimUU}
\|u(t)\|_{H^{s}}^{2}\leq 2\|u_0\|_{H^{s}}^{2}
+C\Big( \|u\|^{10}_{L^{\infty}H^{s}}TN+\|u\|^{6}_{L^{\infty}H^{s}}T
+\|u\|^{4}_{L^{\infty}H^{s}}TN^{-1}+\|u\|^{4}_{L^{\infty}H^{s}}\Big)\,,
\end{equation}
for some $C>0$ depending on $s$.
To prove the claim we reason as follows.
By Proposition \ref{NLSparapara} we have that 
\eqref{NLS} is equivalent to the system 
\eqref{QLNLS444}.
By Propositions \ref{diago2max}, \ref{prop:blockdiag00} and Lemma 
\ref{equaVVnn}
we can construct a function $z_{n}$ with $2n=s$ such that if
$u(t,x)$ solves the \eqref{NLS} then $z_{n}$ solves the equation 
\eqref{equa:VVTL}. Moreover by Proposition \ref{lem:basicenergy}
we have the equivalence \eqref{pallone}.
and we deduce
\begin{equation}\label{darma2}
\|u(t)\|_{H^{s}}^{2}\leq 2^{1/2}\|z_{n}(t)\|_{L^{2}}^{2}\leq 2 \|u_0\|^{2}_{H^{s}}
+2\left|\int_{0}^{t}\mathcal{B}(\s)d\s\right|+2 \left|\int_{0}^{t}\mathcal{B}_{>5}(\s)d\s\right|\,,
\end{equation}
Propositions \ref{lem:basicenergy} and \ref{prop:stimeEnergia} apply, therefore, by \eqref{stimaeneNonresCaso2}
and \eqref{quintestimate}, we obtain
the \eqref{claimUU}.
The thesis of Theorem
\ref{main:NLS} follows from the 
following lemma.
 \begin{lemma}{\bf (Main Bootstrap).}
Let $u(t,x)$ be a solution of \eqref{NLS} with $t\in[0,T)$ and initial condition 
$u_0\in H^{s}(\mathbb{T}^{d};\mathbb{C})$. 
Then, for $s\gg1$ large enough,  there exist $\e_0,c_0>0$ such that, 
for any $0<\e\leq\e_0$, if 
\begin{equation}\label{boot1}
\|u_0\|_{H^{s}}\leq \tfrac{1}{4}\e\,,\qquad \sup_{t\in[0,T)}\|u(t)\|_{H^{s}}\leq \e\,,
\qquad T\leq c_0\e^{-4}\,,
\end{equation}
then we have the improved bound $\sup_{t\in[0,T)}\|u(t)\|_{H^{s}}\leq \e/2$.
\end{lemma}

\begin{proof}
For $\e$ small enough the bound \eqref{claimUU} holds true, 
and we fix $N:=\e^{-3}$.
 Therefore, there is $C=C(s)>0$
such that, for any $t\in[0,T)$,
\begin{equation}\label{boot3}
\begin{aligned}
\|u(t)\|_{H^{s}}^{2}&\leq 2\|u_0\|_{H^{s}}^{2}+C\big( \|u\|^{4}_{L^{\infty}H^{s}}
+\|u\|^{10}_{L^{\infty}H^{s}}T\e^{-3}+\|u\|^{6}_{L^{\infty}H^{s}}T
+ \|u\|^{4}_{L^{\infty}H^{s}}T\e^{3}\big)\\
&\stackrel{\mathclap{\eqref{boot1}}}{\leq} 
\tfrac{1}{8}\e^{2}+
C\big(\e^{4}+ 2\e^{7}T
+\e^{6}T\big)
\\&
\leq \tfrac{\e^{2}}{4}(\tfrac{1}{2}+4C(\e^{4}+2\e c_0+c_0))\leq \e^{2}/4
\end{aligned}
\end{equation}
where in the last inequality we have chosen  $c_0$ and $\e$ sufficiently small. 
This implies the thesis.
\end{proof}

\vspace{0.2em}
\noindent
{\bf Proof of Theorem \ref{main2}}.
One has to  follow almost word by word the proof of 
Theorem \ref{main:NLS}. The only difference relies on the estimates on the small divisors
which in this case are given by item $(ii)$ of Proposition \ref{NonresCond}.

\vspace{0.2em}
\noindent
{\bf Proof of Theorem \ref{main:KG}}.
Consider \eqref{KG} and let $(\psi_0,\psi_1)$ as in the statement of Theorem \ref{main:KG}. 
Let  $\psi(t,x)$ be a solution of \eqref{KG} satisfying the condition in Hyp. \ref{2epsilonKG}.
By Proposition \ref{KGparaparaKG}, recall \eqref{CVWW}, 
the function $U:=\vect{u}{\bar{u}}$ 
solves \eqref{QLNLS444} 
with initial condition $u_0=\tfrac{1}{\sqrt{2}}(\Lambda_{\KG}^{\frac{1}{2}}\psi_0+\ii \Lambda_{\KG}^{-\frac{1}{2}}\psi_1)$. Moreover, by Hyp. \ref{2epsilonKG} and Remark \ref{WWUUUKG}
one has 
$\|u_0\|_{H^{s}}\leq 1/16\e$.
By Remark \ref{WWUUUKG}, in order to get the \eqref{tesiKG}, we have to show that 
the bound $\sup_{t\in[0,T)}\|u\|_{H^{s}}\leq \e/4$
 holds for  time  
 $T\gtrsim \e^{-3^{-}}$ if $d=2$ and $T\gtrsim \e^{-8/3^{-}}$  if $d\geq3$.
Fix $\beta$ as in Proposition \ref{NonresCondKG} and let $m\in \mathcal{C}_{\beta}$.
By Propositions \ref{diagoOrd1}, \ref{prop:blockdiag00KG} and Lemma 
\ref{equaVVnnKG}
we can construct a function $z_{n}$ with $n=s$ such that if
$\psi(t,x)$ solves the \eqref{KG} then $z_{n}$ solves the equation 
\eqref{equa:VVTLKG}. 
By Proposition \ref{lem:basicenergyKG} and Remark \ref{equivLemmaRMKKG} we get
\begin{equation}\label{darmaKG2}
\|u(t)\|_{H^{s}}^{2}\leq
2^{1/2} \|z_{n}(t)\|_{L^{2}}^{2}\leq 2\|u_0\|^{2}_{H^{s}}
+2 \left|\int_{0}^{t}\mathcal{B}(\s)d\s\right|+ 2\left|\int_{0}^{t}\mathcal{B}_{>5}(\s)d\s\right|.
\end{equation}
Propositions \ref{lem:basicenergyKG} and \ref{prop:stimeEnergiaKG} apply, therefore, by \eqref{stimaeneNonresCaso2KG}
and \eqref{quintestimateKG}, we obtain
 the following \emph{a priori} estimate:
fix any  $0<N$, then 
for any $t\in [0,T)$, with $T$ as in Hyp. \ref{2epsilonKG}, one has
\begin{equation}\label{claimUUKG}
\begin{aligned}
\|u(t)\|_{H^{s}}^{2}&\leq2 \|u_0\|_{H^{s}}^{2}
\\&+C\big( \|u\|^{6}_{L^{\infty}H^{s}}TN^{\beta-1}+
\|u\|^{7}_{L^{\infty}H^{s}}TN^{\beta}
+
\|u\|^{6}_{L^{\infty}H^{s}}T
+\|u\|^{4}_{L^{\infty}H^{s}}TN^{-1}
+N^{\beta-1}\|u\|^{4}_{L^{\infty}H^{s}}\big)\,,
\end{aligned}
\end{equation}
for some $C>0$ depending on $s$.
The thesis of Theorem
\ref{main:KG} follows from the 
following lemma.
 \begin{lemma}{\bf (Main bootstrap).}
Let $u(t,x)$ be a solution of \eqref{QLNLS444KG} 
with $t\in[0,T)$ and initial condition 
$u_0\in H^{s}(\mathbb{T}^{d};\mathbb{C})$. 
Define $\mathtt{a}=3$ if $d=2$ and $\mathtt{a}=8/3$ if $d\geq3$.
Then, for $s\gg1$ large enough and any $\delta>0$,  
there exist $\e_0=\e_0(d,s,m,\delta)>0$ 
such that, 
for any $0<\e\leq\e_0$, if 
\begin{equation}\label{bootKG1}
\|u_0\|_{H^{s}}\leq 1/16\e\,,\qquad \sup_{t\in[0,T)}\|u(t)\|_{H^{s}}\leq \e/4\,,
\qquad T\leq \e^{-\mathtt{a}+\delta}\,,
\end{equation}
then we have the improved bound $\sup_{t\in[0,T)}\|u(t)\|_{H^{s}}\leq \e/8$.
\end{lemma}

\begin{proof}
We start with $d\geq 3$. 
For $\e$ small enough the bound \eqref{claimUUKG} holds true.
Let $\delta>0$ and $0<\s\ll\delta$.
Define
\begin{equation}\label{francoKG}
\beta:=3+\s\,,\quad 
N:=\e^{-\frac{2}{3+\s}}\,.
\end{equation}
By \eqref{claimUUKG}, \eqref{bootKG1}, \eqref{francoKG},  there is $C=C(s)>0$
such that, for any $t\in[0,T)$,
\begin{equation}\label{bootKG3}
\begin{aligned}
\|u(t)\|_{H^{s}}^{2}&\leq 2\frac{1}{16^{2}}\e^{2}+
C\e^{2}\e^{\frac{2}{3+\s}}
+2 CT\e^{2}(\e^{3}+\e^{2+\frac{2}{3+\s}} )\leq {\e^{2}}/{64}
\end{aligned}
\end{equation}
where in the last inequality we have chosen  
$\e$ sufficiently small
and we used the choice of $T$ in \eqref{bootKG1} and that $\s\ll\delta$. 
This implies the thesis for $d\geq 3$.
In the case $d=2$ the proof is similar setting $\beta=2+\s$
and $N=\e^{-2/(2+\s)}$.
\end{proof}

\noindent
{\bf Proof of Theorem \ref{main:KGsemi}}.
Using the Remarks \ref{semilin1}, \ref{semilin2}, \ref{semilin3}, \ref{semilin4}, \ref{semilin5}, \ref{semilin6}, \ref{semilin7}
one deduces the result by reasoning as in the proof of Theorem \ref{main:KG}
and using in particular the estimate \eqref{stimaeneNonresCaso2KGsemilin}.


\def\cprime{$'$}




\begin{thebibliography}{10}
\bibitem{baldi-BO} P. Baldi 
\newblock Periodic solutions of fully nonlinear autonomous equations of {B}enjamin-{O}no type.
\newblock {\em Ann. I. H. Poincaré}, 30: 33–77, 2013.


\bibitem{BHM}
P.~Baldi, E.~Haus, and R.~Montalto.
\newblock Controllability of quasi-linear {H}amiltonian {N}{L}{S} equations.
\newblock {\em J. Differential Equations}, 264(3):1789-1840, 2018.


\bibitem{Bam03}
D.~Bambusi. {Birkhoff normal form for some nonlinear {PDE}s}, {\em Comm. Math.
  Physics} \textbf{234}, 253--283, 2003.

\bibitem{BDGS}
D.~Bambusi, J.~M. Delort, B.~Gr\'ebert, and J.~Szeftel.
\newblock {A}lmost global existence for {H}amiltonian semi-linear
  {K}lein-{G}ordon equations with small {C}auchy data on {Z}oll manifolds.
\newblock {\em Comm. Pure Appl. Math.}, 60:1665--1690, 2007.

\bibitem{BG}
D.~Bambusi and B.~Gr\'ebert.
\newblock {B}irkhoff normal form for partial differential equations with tame
  modulus.
\newblock {\em Duke Math. J.}, 135 n. 3:507-567, 2006.

\bibitem{BFG} 
J. Bernier, E. Faou and B. Gr\'ebert.
\newblock{{L}ong time behavior of the solutions of {N}{L}{W} on the $d$-dimensional torus.}
\newblock{\em {F}orum of {M}athematics, {Sigma}}, 8:12, 2020.


\bibitem{BD}
M.~Berti and J.M. Delort.
\newblock { {A}lmost global solutions of capillary-gravity water waves
  equations on the circle}.
\newblock {\em UMI Lecture Notes}, 2017.

\bibitem{BFF1}
M.~Berti, R.~Feola, and L.~Franzoi.
\newblock Quadratic life span of periodic gravity-capillary water waves.
\newblock {\em Water Waves}, 
3:85-115, 2021.

\bibitem{BFP}
M.~Berti, R.~Feola, and F.~Pusateri.
\newblock Birkhoff normal form and long time existence for periodic gravity
  water waves.
\newblock {\em preprint arXiv:1810.11549}, 2018. (accepted on CPAM)

\bibitem{BFP1}
M.~Berti, R.~Feola, and F.~Pusateri.
\newblock Birkhoff normal form for gravity water waves.
\newblock {\em Water waves},
3:117-126, 2021.

\bibitem{BMM1}
M. Berti, A. Maspero and F. Murgante. 
\newblock{Local well posedness of the Euler-Korteweg
equations on $\mathbb{T}^{d}$ }. 
\newblock{\em Journal of Dynamics and Differential equations}, 33:1475-1513, 2021.


\bibitem{bony}
J.~M. Bony.
\newblock {C}alcul symbolique et propagation des singularit\'es pour les
  \'equations aux d\'eriv\'ees partielle non lin\'eaire.
\newblock {\em Ann. Sci. \'Ecole Norm. Sup.}, 14:209--246, 1981.

\bibitem{boro-ga}
A.~V. Borovskii and A.L. Galkin.
\newblock {D}ynamical Modulation of an Ultrashort High-Intensity Laser Pulse in Matter. 
\newblock {\em JETP}, 77(4):209--246, 1993.




\bibitem{colin}
M.~Colin.
\newblock {O}n the local well-posedness of quasilinear {S}chr\"odinger
  equations in arbitrary space dimension.
\newblock {\em Communications in Partial Differential Equations}, 27:325--354,
  2002.
  
  \bibitem{colin2003}
M.~Colin.
\newblock Stability of stationary waves for a quasilinear {S}chr\"odinger equation
  in space dimension 2.
\newblock {\em Adv. Differential Equations}, 8(1):1--28, 2003.


\bibitem{colinjj} 
M. ~Colin and L. ~Jeanjean. 
\newblock{ S}olutions for a quasilinear {S}chr\"odinger equation: a dual approach.
\newblock{\em Nonlinear Analysis}, vol. 56(2), 213--226, 2004.


\bibitem{saut-glob}
N.~de~Bouard, A. Hayashi and J.C. Saut.
\newblock {G}lobal existence of small solutions to a relativistic nonlinear
  {S}chr\"odinger equation.
\newblock {\em Comm. Math. Phys.}, 189:73--105, 1997.

\bibitem{Delort-glob}
J.M. Delort.
\newblock{Semiclassical microlocal normal forms and global solutions of modified one-dimensional KG equations}
\newblock{\em Annales de l'Institut Fourier},  66:4,  1451-1528, 2016.

\bibitem{Delort-Tori}
J.M. Delort.
\newblock{On long time existence for small solutions of semi-linear Klein-Gordon equations on the torus.}
\newblock{\em Journal d'Analyse Math\`ematique}, 107,  161--194, 2009.

\bibitem{Delort-circle}
J.~M. Delort.
\newblock {A} quasi-linear {B}irkhoff normal forms method. {A}pplication to the
  quasi-linear {K}lein-{G}ordon equation on $\mathds{S}^1$.
\newblock {\em Ast\'erisque}, 341, 2012.



\bibitem{Delort-sphere}
J.M. Delort.
\newblock {{Q}uasi-{L}inear {P}erturbations of {H}amiltonian
  {K}lein-{G}ordon {E}quations on {S}pheres}.
\newblock {\em American Mathematical Society}, 
\newblock 234 (1103), 2015.

\bibitem{DelortSzeft1}
J.~M. Delort and J.~Szeftel.
\newblock {L}ong-time existence for small data nonlinear {K}lein--{G}ordon
  equations on tori and spheres.
\newblock {\em Internat. Math. Res. Notices}, 
\newblock 2004 (37):1897-1966, 2004.

\bibitem{DelortSzeft2}
J.~M. Delort and J.~Szeftel.
\newblock {L}ong-time existence for semi-linear {K}lein--{G}ordon equations
  with small cauchy data on {Z}oll manifolds.
\newblock {\em Amer. J. Math.}, 128 (5), 2006.

 \bibitem{H98}
H.  Eliasson. {Perturbations of linear quasi-periodic systems}. 
   \newblock{\em  Dynamical Systems and Small Divisors (Cetraro, Italy, 1998)}, 
   1-60, Lect. Notes Math. 1784, Springer, 2002.

\bibitem{EGK}
H. Eliasson, B. Gr\'ebert and S. Kuksin.
{KAM for nonlinear beam equation},
{\em Geom. Funct. Anal.}, 26 : 1588--1715, 2016.






\bibitem{fang}
D. Fang and Q. Zhang.
\newblock{{L}ong-time existence for semi-linear Klein-Gordon equations on tori.}
\newblock{\em Journal of differential equations}, 249:151-179, 2010.


\bibitem{Faouplane}
E.~Faou, L.~Gauckler, and C.~Lubich.
\newblock {S}obolev stability of plane wave solutions to the cubic nonlinear
  {S}chr{\"o}dinger equation on a torus.
\newblock {\em Comm. Partial Differential Equations}, 38:1123--1140, 2013.

\bibitem{FaouGreNek}
E.~Faou and B. Gr\'ebert. 
\newblock A Nekhoroshev-type theorem for the nonlinear {S}chr\"odinger equation on the torus.
\newblock {\em Analysis \& PDE}, 6(6), 2013.


\bibitem{Feola-Iandoli-Totale}
R.~Feola and F.~Iandoli.
\newblock {A} non-linear {E}gorov theorem and {P}oincar\'e-{B}irkhoff normal
  forms for quasi-linear pdes on the circle.
\newblock {\em preprint, arXiv:2002.12448}, 2020.

\bibitem{Feola-Iandoli-Loc}
R.~Feola and F.~Iandoli.
\newblock {L}ocal well-posedness for quasi-linear {N}{L}{S} with large {C}auchy
  data on the circle.
\newblock {\em Annales de l'Institut Henri Poincar\'e (C) Analyse non
  lin\'eaire}, 36(1):119--164, 2018.
  
  \bibitem{Feola-Iandoli-local-tori}
R.~Feola and F.~Iandoli.
\newblock {L}ocal well-posedness for the {H}amiltonian quasi-linear {S}chr\"odinger equation on tori.
\newblock{\em{Journal de Math\'ematiques pures et appliqu\'ees.}} 157, 243-281, 2022.


\bibitem{Feola-Iandoli-Long}
R.~Feola and F.~Iandoli.
\newblock {L}ong time existence for fully nonlinear {N}{L}{S} with small
  {C}auchy data on the circle.
\newblock {\em Ann. Sc. Norm. Super. Pisa Cl. Sci.}, 
22(5): 109-182, 2021.



\bibitem{FP}
R.~Feola and M.~Procesi.
\newblock Quasi-periodic solutions for fully nonlinear forced reversible
  {S}chr{\"o}dinger equations.
\newblock {\em Journal of Differential Equations}, 259, 3389-3447, 2015.

\bibitem{Gre}
{\rm B. Gr\'ebert}. 
\newblock Birkhoff normal form and Hamiltonian PDEs.
\newblock {\em S\'eminaires et Congr\`es} \textbf{15}, 1--46, 2007.


\bibitem{goldman}
M.~Goldman and M.V.~Porkolab.
\newblock Upper hybrid solitons and oscillating two-stream instabilities.
\newblock {\em Physics of Fluids}, 19:872--881, 1976.

\bibitem{hasse}
R.W. ~Hasse.
\newblock A general method for the solution of nonlinear soliton and kink
  {S}chr\"odinger equations.
\newblock {\em Z. Physik B}, 3:83--87, 1980.

 \bibitem{IPtori} A.D. ~Ionescu and F. ~Pusateri.
{Long-time existence for multi-dimensional periodic water waves}. 
{\emph{Geom. Funct. Anal.}}, 29: 811--870, 2019. 

\bibitem{pusa1}
A.~Ionescu and F.~Pusateri.
\newblock {G}lobal solutions for the gravity water waves system in 2d.
\newblock {\em Invent. Math.}, 199(3):653--804, 2015.

\bibitem{pusa2}
A.~Ionescu and F.~Pusateri.
\newblock Global regularity for 2d water waves with surface tension.
\newblock {\em Mem. Amer. Math. Soc.}, 256 (1227), 2018.

\bibitem{kato}
T.~Kato.
\newblock { {S}pectral Theory and Differential Equations. {L}ecture Notes in
  Mathematics, (eds.) Everitt, W. N.}, volume 448, chapter ``{Q}uasi-linear
  equations evolutions, with applications to partial differential equations''.
\newblock {\em Springer, Berlin, Heidelberg}, 1975.

\bibitem{KPV}
C.~E. Kenig, G.~Ponce, and L.~Vega.
\newblock {T}he {C}auchy problem for quasi-linear {S}chr\"odinger equations.
\newblock {\em Invent. Math.}, 158:343--388, 2004.

%
\bibitem{LLNR}
J.~Laurie, V.S. ~L'Vov, S.V. ~Nazarenko, and O.~Rudenko.
\newblock {I}nteraction of kelvin waves and nonlocality of energy transfer in
  superfluids.
\newblock {\em Phys. Rev. B}, 81, 2010.
%
\bibitem{sergio}
A.M.~Litvak and A.G.~Sergeev.
\newblock One dimensional collapse of plasma waves.
\newblock {\em JETP, Letters}, 194:517--520, 1978.
%
\bibitem{fedayn}
V.K. ~Makhankov and V.G.~Fedyanin.
\newblock Non-linear effects in quasi-one- dimensional models of condensed
  matter theory. 
  \newblock{\em Physics reports}, 104: 1--86, 1984.
  
 
  








%

\bibitem{MMT3}
J.~Marzuola, J.~Metcalfe, and D.~Tataru.
\newblock {Q}uasilinear {S}chr\"odinger equations {I}{I}{I}: large data and
  short time.
\newblock {Archive of Rational mechanics and analysis}, 242 (2):1119-1175, 2021.




\bibitem{Metivier}
G.~M\'etivier.
\newblock {\em {P}ara-{D}ifferential {C}alculus and {A}pplications to the
  {C}auchy {P}roblem for {N}onlinear {S}ystems}, volume~5.
\newblock Edizioni della Normale, 2008.
  
\bibitem{Pop1}
M.~Poppenberg.
\newblock {S}mooth solutions for a class of fully nonlinear {S}chr{\"o}dinger
  type equations.
\newblock {\em Nonlinear Anal., Theory Methods Appl.}, 45(6):723--741, 2001.


\bibitem{Sti}
A. Stingo. 
\newblock Global existence and asymptotics for quasi-linear one-dimensional Klein-Gordon equations with mildly decaying Cauchy data.
\newblock{\em Bulletin de la SMF} 146, 1: 155--213, 2018.


\bibitem{Tay-Para}
M.~Taylor.
\newblock {\em {T}ools for {P}{D}{E}}.
\newblock Amer. Math. Soc., 2007.
%














\end{thebibliography}
\end{document}